\documentclass[online]{brandiss}

\usepackage{graphicx}

\usepackage{amssymb}
\usepackage{nicefrac}

\usepackage[all]{xy}                   

\numberwithin{section}{chapter}

\numberwithin{figure}{chapter}

\disstitle{The Flag Descent Algebra and the Colored Eulerian Descent~Algebra}
\dissauthor{Matthew Moynihan}
\dissadvisor{Ira Gessel}
\dissdepartment{Mathematics}
\dissmonth{August}
\dissyear{2012}
\dissdean{Malcolm Watson}

\theoremstyle{remark}

\newtheorem{thm}{Theorem}[section]
\newtheorem{cor}[thm]{Corollary}
\newtheorem{prop}[thm]{Proposition}
\newtheorem{lem}[thm]{Lemma}
\newtheorem{ex}[thm]{Example}
\newtheorem{definition}[thm]{Definition}

\newcommand{\A}{\mathcal{A}}
\newcommand{\Z}{\mathbb{Z}}
\newcommand{\R}{\mathbb{R}}

\newcommand{\N}{\mathbb{N}}
\newcommand{\pos}{\mathbb{P}}
\newcommand{\zeroplus}{0^{+}}
\newcommand{\zerominus}{0^{-}}
\newcommand{\Lin}{\mathcal{L}}
\newcommand{\collin}{\mathcal{CL}}

\newcommand{\symn}{\mathfrak{S}_{n}}
\newcommand{\hypn}{\mathfrak{B}_n}
\newcommand{\colpermrn}{G_{r,n}}
\newcommand{\multn}{\mathfrak{M}_\alpha}
\newcommand{\multrn}{\mathfrak{M}_\alpha^r}
\newcommand{\hypmultn}{\overline{\mathfrak{M}}_\alpha}
\newcommand{\zigipi}{Z(I,\pi)}
\newcommand{\zigaipi}{Z^A(I,\pi)}
\newcommand{\zigbipi}{Z^B(I,\pi)}
\newcommand{\czigipi}{Z^c(I,\pi)}
\newcommand{\augzigipi}{Z^{(a)}(I,\pi)}

\newcommand{\chainipi}{C(I,\pi)}
\newcommand{\chainaipi}{C^A(I,\pi)}
\newcommand{\chainbipi}{C^B(I,\pi)}
\newcommand{\cchainipi}{C^c(I,\pi)}
\newcommand{\augchainipi}{C^{(a)}(I,\pi)}

\newcommand{\zerovec}{\vec{0}}
\newcommand{\dsum}{\displaystyle\sum}

\newcommand\qbinom[2]{\genfrac{[}{]}{0pt}{}{#1}{#2}_q}

\newcommand\qbinomplain[2]{\genfrac{[}{]}{0pt}{}{#1}{#2}}
\newcommand\mchoose[2]{\left(\!\! \binom{#1}{#2}\!\!\right)}
\newcommand\ceiling[1]{\left \lceil{#1}\right\rceil}

\newcommand\aversion[1]{{#1}^A}
\newcommand\bversion[1]{{#1}^B}
\newcommand\rversion[1]{{#1}_{(r)}}
\newcommand\augversion[1]{{#1}^{(a)}}
\newcommand\shuffle[1]{\shuf(#1)}

\newcommand{\shufflezeropi}{\shuf(\pi,\zerovec)}
\newcommand{\anchoredcolpermrn}{\colpermrn^0}
\newcommand{\rfoldx}{X_{(r)}}

\newcommand{\augzigsets}{Z(n)}
\newcommand{\augchainsets}{C(n)}
\newcommand{\OmegaA}{\Omega^A}
\newcommand{\OmegaB}{\Omega^B}
\newcommand{\Omegaaug}{\Omega^{(a)}}
\newcommand{\Omegaflag}{\Omega^{\mathrm{flag}}}
\newcommand{\cycliczigsets}{Z^c(n)}
\newcommand{\cyclicchainsets}{C^c(n)}

\DeclareMathOperator{\co}{co}

\DeclareMathOperator{\des}{des}
\DeclareMathOperator{\Des}{Des}
\DeclareMathOperator{\ides}{ides}
\DeclareMathOperator{\intdes}{intdes}
\DeclareMathOperator{\intDes}{intDes}
\DeclareMathOperator{\cDes}{cDes}
\DeclareMathOperator{\Desa}{Des_A}
\DeclareMathOperator{\desa}{des_A}
\DeclareMathOperator{\desb}{des_B}
\DeclareMathOperator{\Desb}{Des_B}

\DeclareMathOperator{\cdes}{cdes}
\DeclareMathOperator{\ades}{ades}

\DeclareMathOperator{\aDes}{aDes}
\DeclareMathOperator{\fdes}{fdes}

\DeclareMathOperator{\maj}{maj}
\DeclareMathOperator{\imaj}{imaj}
\DeclareMathOperator{\amaj}{amaj}
\DeclareMathOperator{\comaj}{comaj}
\DeclareMathOperator{\icomaj}{icomaj}

\DeclareMathOperator{\fmaj}{fmaj}
\DeclareMathOperator{\famaj}{famaj}

\DeclareMathOperator{\shuf}{Sh}

\begin{document}


\frontmatter

\makedisstitle

\begin{disssignatures}
\committeemember Susan Parker, Dept.~of Mathematics
\committeemember Jonathan Novak, Dept.~of Mathematics, Massachusetts Institute of Technology
\end{disssignatures}

\disscopyright 

\begin{dissdedication}
  \begin{center}
  To my parents, the first teachers I ever had.
  \end{center}
\end{dissdedication}

\begin{dissacknowledgments} 

I would like to first thank my advisor Ira Gessel for his support over these past few years.
He has vast quantities of both knowledge and patience, two traits not often found together.
I would also like to thank the members of my dissertation defense committee and the faculty and staff of the Brandeis Mathematics Department.
In particular, I thank Susan Parker for her guidance and advice on both teaching and life.
Without the amazing faculty at St.\ Olaf College, I would never have seen the beauty of mathematics and never chosen to pursue graduate research.

Thank you to my friends both in and out of the mathematics department.
I always felt your support, even when I disappeared into my work.
While not my friends (yet), I would like to thank Doomtree for providing the soundtrack to my graduate career.  
 
I would also like to thank my family.
While it was hard being away from my Minnesota family, I was extremely fortunate to spend the last five years getting to know my Boston family.
They took me in immediately and I always felt loved and included.
Finally, I would like to thank my parents.
It was only through their constant support and sacrifice that I was able to make it this far.
\end{dissacknowledgments}

\begin{dissabstract}

We prove that the group algebra of the hyperoctahedral group contains a subalgebra corresponding to the flag descent number of Adin, Brenti, and Roichman.
This algebra is in fact the span of the basis elements of the type~$A$ and type~$B$ Eulerian descent algebras.
We describe a set of orthogonal idempotents which span the flag~descent algebra and prove that it contains the type~$A$ Eulerian descent algebra as a two-sided ideal.

Using a new colored analogue of Stanley's $P$-partitions, we prove the existence of a colored Eulerian descent algebra which is a subalgebra of the Mantaci-Reutenauer algebra.
We also describe a set of orthogonal idempotents that spans the colored Eulerian descent algebra and includes, as a special case, the familiar Eulerian idempotents in the group algebra of the symmetric group.

\end{dissabstract}

\begin{disspreface}

In 1976, Solomon \cite{Solomon1976} proved the existence of a subalgebra of the group algebra of the symmetric group whose basis elements are formal sums of permutations with the same descent set (the set of all $i$ such that $\pi(i) > \pi(i+1)$).
Solomon also proved that an analogous algebra can be defined for all finite Coxeter groups using the length function associated to a specified generating set.
We examine an ``Eulerian'' subalgebra of Solomon's descent algebra which is related to the descent number of permutations and investigate the connection between permutation enumeration and the group algebra of the symmetric group.
Our proof of the existence of the Eulerian descent algebra is a modified version of Petersen's proof in~\cite{Petersen2005}, which was in turn inspired by Gessel's work on multipartite $P$-partitions in~\cite{Gessel1984}.
We also study the relationship between the Eulerian descent algebra and the cyclic Eulerian descent algebra defined by Cellini in~\cite{Cellini1998}, showing that the cyclic Eulerian descent algebra is a left module over the Eulerian descent algebra.

Next we extend the results of Chapter~\ref{chpt: Symmetric Group} to the hyperoctahedral group, which has a combinatorial interpretation as the group of ``signed'' permutations of length~$n$.
This group is also the Coxeter group of type~$B_n$ and contains a type~$B$ analogue of the Eulerian descent algebra.
By altering the definition of the descent set of a signed permutation, we prove the existence of multiple algebras.
These proofs are extensions of those from Chapter~\ref{chpt: Symmetric Group} and use different versions of ``signed'' $P$-partitions, including Chow's $P$-partitions of type~$B$ defined in~\cite{ChowThesis2001} and Petersen's augmented $P$-partitions defined in~\cite{Petersen2005}.
While the existence of these algebras is not new, our examination of the relationships between the algebras contains new results.
We prove that some of these descent algebras are ideals in larger algebras, expanding on a result of Petersen~\cite{Petersen2005}.
For example, we show that the type~$A$ Eulerian descent algebra is a two-sided ideal in the algebra spanned by the basis elements from the type~$A$ and type~$B$ Eulerian descent algebras.
This fact implies that the flag descent number introduced by Adin, Brenti, and Roichman~\cite{AdinBrentiRoichman2001} can be used to define an algebra, a fact we also prove directly.

In 1995, Mantaci and Reutenauer~\cite{MantaciReutenauer1995} proved the existence of an algebra whose basis elements are formal sums of ``colored permutations'' with the same associated colored compositions (an extension of the descent set).
Colored permutations can be thought of as permutations from the symmetric group together with a choice of a ``color'' from a finite cyclic group for each letter in a given permutation.
The Mantaci-Reutenauer algebra is well studied and fulfills much the same role as Solomon's descent algebra does in the theory of the group algebra of the symmetric group (see~\cite{BaumannHohlweg2008, Poirier1998}).
While colored permutations do not form a Coxeter group, we show that there still exists an Eulerian subalgebra of the Mantaci-Reutenauer algebra when descents are defined using Steingr{\'{\i}}msson's definition in~\cite{Steingrimsson1994}.
Our proof uses colored posets and colored $P$-partitions which are different in character from those defined by Hsiao and Petersen in \cite{HsiaoPetersen2010}.
Finally, we describe a set of colored Eulerian idempotents which spans this colored Eulerian descent algebra and reduces to the well-known Eulerian idempotents in the group algebra of the symmetric group when considering permutations of a single color.

\end{disspreface}

\tableofcontents 

\listoftables 

\listoffigures 

\mainmatter


\chapter{The Symmetric Group}\label{chpt: Symmetric Group}

In 1976, Solomon \cite{Solomon1976} proved the existence of a subalgebra of the group algebra of the symmetric group whose basis elements are defined by comparing the length of permutations before and after composition with the adjacent transpositions that generate the symmetric group.
This length property has a combinatorial interpretation related to the set of ``descents'' of a permutation and thus Solomon's descent algebra can be connected to permutation enumeration.
We examine an ``Eulerian'' subalgebra of Solomon's descent algebra which is connected to the descent number of permutations and was first proven to exist by Loday \cite{Loday1989}.
We also study the relationship between the Eulerian descent algebra and Cellini's~\cite{Cellini1998} cyclic Eulerian descent algebra.
Chapter~\ref{chpt: Symmetric Group} is meant to be both an introduction to the some of these ``descent'' subalgebras of the group algebra of the symmetric group as well as a demonstration of techniques which we later extend to both ``signed'' and ``colored'' permutations.
The chapter is organized as follows.

We begin with permutation enumeration.
In Section~\ref{sec: P-partitions}, we review Stanley's theory of $P$-partitions and show how they can be used to count permutations by number of descents.
We extend this in Section~\ref{sec: descent and major index} to count permutations by descent number and major index.
These same methods can be used to count multiset permutations by descent number and major index but we instead use Gessel's barred permutation technique to derive such results in Section~\ref{sec: multiset}.

After using both barred permutations and $P$-partitions, we combine the techniques in Section~\ref{sec: Eulerian descent algebra} to prove the existence of the Eulerian descent algebra and describe certain Eulerian idempotents.
In Section~\ref{sec: cyclic Eulerian module}, we define the cyclic Eulerian descent algebra whose existence was originally proven by Cellini~\cite{Cellini1998}.
We then show that the cyclic Eulerian descent algebra is a left module over the Eulerian descent algebra.
Finally, we show in~Section \ref{sec: des ides maj imaj} how barred posets and $P$-partitions can be used to count permutations by descent number, inverse descent number, major index, and inverse major index.
We obtain a generating function originally due to Garsia and Gessel~\cite{GarsiaGessel1979}.


\section{$P$-partitions}\label{sec: P-partitions}

We begin with the symmetric group.
Let $[n]$ denote the set $\{1,2,\ldots,n\}$.
The \emph{symmetric group} $\symn$ is defined to be the group of all bijections $\pi: [n] \rightarrow [n]$.
We write permutations in one-line notation as $\pi = \pi(1)\pi(2)\cdots\pi(n)$.
The descent set $\Des(\pi)$ of a permutation $\pi$ is defined to be the set of all $i \in [n-1]$ such that $\pi(i) > \pi(i+1)$ and the number of descents of $\pi$ is $\des(\pi) = |\Des(\pi) |$.
We define the major index of a permutation $\pi$ by
\[
\maj(\pi) = \sum_{i\in \Des(\pi)} i
\]
and the comajor index by
\[
\comaj(\pi) = \sum_{i\in \Des(\pi)} (n-i).
\]

\begin{ex}
The permutation $\pi = 51423$ is in $\mathfrak{S}_5$ and $\Des(\pi) = \{1,3\}$. Thus $\des(\pi) = 2$, $\maj(\pi) = 4$, and $\comaj(\pi) = 6$.
\end{ex}

The definition of the comajor index can be interpreted as counting the positions of the descents from right to left as opposed to the more traditional left to right.
Our first goal is to count permutations by descent number and for that we turn to the theory of $P$-partitions.

Suppose we wish to count nonnegative integer solutions to the inequalities $x_2 < x_1$ and $x_2 \leq x_3$.
There is no relation between $x_1$ and $x_3$ and so we break up solutions into two cases.
First, if we assume that $x_1\leq x_3$ then we have $x_2<x_1\leq x_3$.
Second, if we assume that $x_3 < x_1$ then we have $x_2 \leq x_3 < x_1$.
These assumptions neatly partition solutions into two disjoint sets.
However, if we had instead chosen to split solutions according to $x_1 <x_3$ and $x_3 \leq x_1$ then we would have obtained a less elegant division.
Assuming that $x_3 \leq x_1$ gives the inequality $x_2 \leq x_3 \leq x_1$ but with the added restriction that $x_2 < x_1$.

By splitting up solutions to the inequalities, we are in fact dividing $\R^n$ into half spaces by cutting along the hyperplanes $x_i=x_j$.
An elegant partition of solutions corresponds to a consistent way to assign points on the boundary hyperplanes to the solution sets.
The easiest way to describe such an assignment is through $P$-partitions.

The original definition of $P$-partitions is due to Stanley \cite{Stanley1972ordered} and the following overview is derived from his discussion in Chapter 3 of \cite{StanleyEC1} as well as Petersen's exposition in \cite{PetersenThesis}.
Unless otherwise stated, $P$ will represent a partial order on $[n]$ with relation $<_P$.
We represent posets visually through their Hasse diagrams in~which the relation $x <_P y$ is represented by placing $x$ below $y$ in the diagram and connecting them with a line.
See Figure \ref{fig: poset and extensions} for the Hasse diagram of the poset $1 >_P 2 <_P 3$.
Let $\N = \{0,1,2,\ldots\}$ denote the set of nonnegative integers.

\begin{definition}
Let $P$ be a partial order on $[n]$. A \emph{$P$-partition} is a function $f:[n]\rightarrow \N$ satisfying the following:
\begin{enumerate}
\item[(i)] $f(a) \leq f(b)$ if $a <_P b$
\item[(ii)] $f(a) < f(b)$ if $a<_P b$ and $a > b$ in $\Z$.
\end{enumerate}
\end{definition}

Stanley originally defined $P$-partitions to be order reversing which nicely yields results about the major index of permutations.
However, we choose to define $P$-partitions to be order preserving.
Order-preserving $P$-partitions seem more natural but yield results about the comajor index of permutations.
Moving between comajor index and major index is fairly straightforward as long as one keeps track of the number of descents.
We will comment more on this later in Lemma \ref{lem: maj comaj swap}.

We define $\Lin(P)$ to be the set of linear extensions of $P$, i.e., the set of permutations in $\symn$ that extend $P$ to a linear order.
A permutation $\pi$ is in $\Lin(P)$ if and only if whenever $a <_P b$ then $a$ comes before $b$ in $\pi$ and thus $\pi^{-1}(a) < \pi^{-1}(b)$.
Figure \ref{fig: poset and extensions} shows that $213$ and $231$ are the two linear extensions of the poset $1 >_P 2 <_P 3$.

\begin{figure}[htbp]
\[\xymatrix @!R @!C @=12pt{  & & & & & & 3\ar@{-}[dd] & 1\ar@{-}[dd] \\ & 1 \ar@{-}[ddr] & & 3 & & & &   \\  P: & &  & & & \Lin(P): & 1\ar@{-}[dd] & 3\ar@{-}[dd] \\  & & 2 \ar@{-}[uur] & & & &  &  \\  & & & & & & 2 & 2 }\]      
\caption{\; Linear extensions of the poset $P$.}
\label{fig: poset and extensions}
\end{figure}
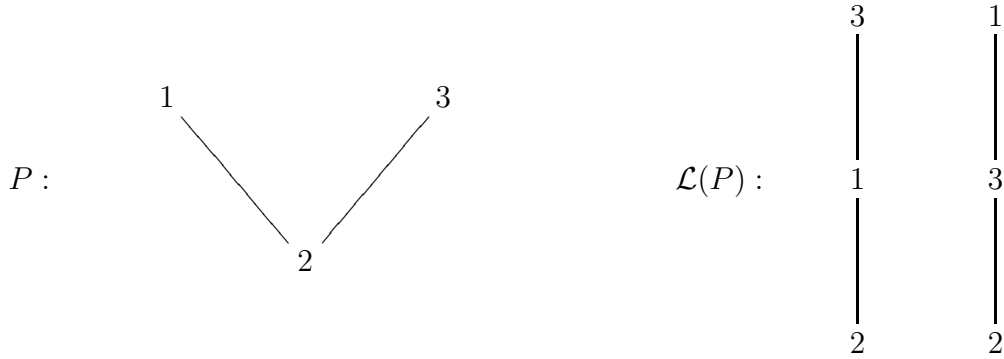

Every permutation $\pi \in \symn$ corresponds to the linear order
\[
\pi(1) <_P \pi(2) <_P \cdots <_P \pi(n).
\]
Thus $P$-partitions corresponding to this linear order are called $\pi$-partitions and can be thought of as functions $f:[n]\rightarrow \N$ such that $f(\pi(i)) \leq f(\pi(i+1))$ for every $i \in [n-1]$ and $f(\pi(i)) < f(\pi(i+1))$ whenever $\pi(i) > \pi(i+1)$.
Hence strict inequalities correspond to descents in permutations.
It is this correspondence that allows us to count permutations by number of descents.

Let $\A(P)$ denote the set of all $P$-partitions and let $\A(\pi)$ denote the set of all $\pi$-partitions.
We have the following Fundamental Theorem of $P$-partitions due to Stanley.

\begin{thm}(FTPP)
\[
\A(P) = \sum_{\pi \in \Lin(P)} \A(\pi)
\]
\end{thm}

\begin{proof}
The proof is by induction on the number of incomparable elements.
If $i$~and $j$ are incomparable in $P$ with $i < j$ in $\Z$ then let $P_{ij}$ denote the partial order $P$ together with the additional relation $i <_P j $ and let $P_{ji}$ denote the partial order $P$ together with the additional relation $j <_P i$.
Then $\Lin(P) = \Lin(P_{ij}) \sqcup \Lin(P_{ji})$ and so $\A(P) = \A(P_{ij}) \sqcup \A(P_{ji})$ and the theorem follows by induction.
\end{proof}

We define the order polynomial $\Omega_P(j)$ to be the number of $P$-partitions $f$ such that $f(i) \leq j$ for every $i \in [n]$.
We say that such an $f$ has parts less than or equal to~$j$.
The Fundamental Theorem of $P$-partitions yields the following corollary.

\begin{cor}
\[
\Omega_P (j) = \sum_{\pi \in \Lin(P)} \Omega_{\pi} (j)
\]
\end{cor}

Next suppose that $P$ is the union of two disjoint posets $P_1$ and $P_2$.
The map which sends a $P$-partition $f$ to the ordered pair $(g,h)$ where $g=f|_{P_1}$ and $h=f|_{P_2}$ is a bijection between $\A(P)$ and $\A(P_1)\times \A(P_2)$.
Thus we have the following theorem.

\begin{thm}
If $P_1$ and $P_2$ are partial orders on disjoint sets then
\[
\Omega_{P_1\sqcup P_2}(j) = \Omega_{P_1}(j) \Omega_{P_2}(j).
\]
\end{thm}

Our goal in this section is to show how $P$-partitions can be used to count permutations in $\symn$ by number of descents and so we define the $n$th Eulerian polynomial~$A_n(t)$~by
\[
A_n(t) = \sum_{\pi \in \symn} t^{\des(\pi)}.
\]
This is a somewhat nontraditional definition and what is classically referred to as the $n$th Eulerian polynomial would be, in our notation, $t A_n(t)$.

First suppose that $P$ is a linear order and $\Lin(P) = \{\pi\}$ with $\pi \in \symn$.
Then $\Omega_{\pi}(j) $ counts solutions to
\[
0\leq f(\pi(1)) \leq f(\pi(2)) \leq \cdots \leq f(\pi(n)) \leq j
\]
with $f(\pi(i)) < f(\pi(i+1))$ if $i \in \Des(\pi)$.
We want to reduce every strict inequality to a weak inequality and thus reduce to an easier counting problem.
Given a strict inequality in position~$i$, we subtract one from $f(\pi(k))$ for every $k > i$.
This shifts the strict inequality between $f(\pi(i))$ and $f(\pi(i+1))$ to a weak inequality.
After doing this for every strict inequality, we have a weakly increasing sequence of length $n$ with parts less than or equal to $j-\des(\pi)$.

\begin{ex}\label{ex: inequality shift}
Let $\pi = 51423$ and note that $\des(\pi) = 2$. Then we have
\[
0\leq f(5) < f(1) \leq f(4) < f(2) \leq f(3) \leq j
\]
which becomes
\[
0\leq f(5) \leq f(1)-1 \leq f(4)-1 \leq f(2)-2 \leq f(3)-2 \leq j-2.
\]
\end{ex}

The shifting above shows that $\Omega_\pi(j)$ is equal to the number of solutions to the inequalities
\[
0 \leq i_1 \leq \cdots \leq i_n \leq j-\des(\pi).
\]
Thus
\begin{equation}\label{eq: pi order poly formula}
\Omega_\pi(j) = \mchoose{j-\des(\pi)+1}{n} = \binom{j + n -\des(\pi)}{n}
\end{equation}
where $\left(\! \binom{a}{b} \! \right) = \binom{a+b-1}{b}$ is ``$a$ multichoose $b$'' and counts the number of ways to choose $b$ elements from a set of $a$ elements with repetitions.
In our case, we are choosing $n$ elements from the set $\{0,1,\ldots,j-\des(\pi)\}$ with repetitions.
Thus we have
\begin{align*}
\sum_{j\geq 0} \Omega_\pi (j)t^j &= \sum_{j\geq 0} \binom{j + n -\des(\pi)}{n} t^j  \\
&= t^{\des(\pi)} \sum_{j\geq 0} \binom{j + n}{n} t^j \\
&= \frac{t^{\des(\pi)}}{(1-t)^{n+1}}
\end{align*}
where the last equality follows from the binomial theorem
\[
\sum_{j \geq 0} \binom{j+n}{n} t^j = \frac{1}{(1-t)^{n+1}}.
\]
Since $\Omega_P(j) = \sum_{\pi \in \Lin(P)} \Omega_{\pi}(j)$, we see that
\[
\sum_{j\geq 0} \Omega_P(j)t^j = \frac{\dsum_{\pi \in \Lin(P)} t^{\des(\pi)}}{(1-t)^{|P|+1}}
\]
where $|P|$ denotes the cardinality of $|P|$.

To compute the $n$th Eulerian polynomial $A_n(t)$, we let $P$ be an antichain of $n$ elements, i.e., the set $[n]$ with no relations between the elements.
Then $\Lin(P) = \symn$ and we have
\[
\sum_{j\geq 0} \Omega_P(j)t^j = \frac{\dsum_{\pi \in \symn}t^{\des(\pi)}}{(1-t)^{n+1}} = \frac{A_n(t)}{(1-t)^{n+1}}.
\]
We know that $\Omega_P(j)$ is the product of the order polynomials of the $n$ singleton posets.
If $P(a)$ is the singleton poset $a\in [n]$ then $\Omega_{P(a)}(j)$ counts solutions to $0 \leq f(a) \leq j$ and so $\Omega_{P(a)}(j) = j+1$.
Since the order polynomials $\Omega_{P(a)}(j)$ are the same for all $a\in [n]$, we know that $\Omega_P(j) = (j+1)^n$ and we arrive at the following well-known theorem.

\begin{thm}
\[
\sum_{j\geq 0} (j+1)^nt^j = \frac{A_n(t)}{(1-t)^{n+1}}
\]
\end{thm}


\section{Counting permutations in $\symn$ by descents and major index}\label{sec: descent and major index}

It is well-known that one can use $P$-partitions to count permutations in $\symn$ according to both major index and number of descents.
We include an overview of the relevant arguments here because we will need some of the results in Section \ref{sec: des ides maj imaj}.
We denote the sum $1+q+q^2+\cdots + q^{n-1}$ by $[n]_q$ and define $[n]_q ! := [n]_q[n-1]_q\cdots [1]_q$.
For $a \geq b \geq 0$, we define the $q$-binomial coefficient $\qbinom{a}{b}$ by
\[
\qbinom{a}{b} = \frac{[a]_q!}{[b]_q! \, [a-b]_q!}.
\]
The $q$-binomial coefficient $\qbinom{j+n}{n}$ ``$q$-counts'' the number of weakly increasing sequences of length $n$ with parts less than or equal to $j$ and thus
\[
\qbinom{j+n}{n} = \sum_{0 \leq i_1 \leq  \cdots \leq i_n \leq j} q^{i_1+\cdots +i_n}.
\]
Let $(t;q)_{n+1}$ denote the product $\prod_{i=0}^n (1-tq^i)$.
The $q$-binomial theorem states that
\[
\sum_{j \geq 0} \qbinom{j+n}{n} t^j = \frac{1}{(t;q)_{n+1}}
\]
which reduces to the binomial theorem when $q=1$.

We define the $q$-Eulerian polynomials $A_n(t,q)$ by
\[
A_n(t,q) = \sum_{\pi \in \symn} t^{\des(\pi)}q^{\maj(\pi)}.
\]
Stanley's order-reversing $P$-partitions yield results about the major index but, as we have defined them, $P$-partitions yield results about the comajor index.
We instead compute $A_n^{\co} (t,q) := \sum_{\pi \in \symn} t^{\des(\pi)}q^{\comaj(\pi)}$ and will be able to retrieve $A_n(t,q)$ due to the following lemma.

\begin{lem}\label{lem: maj comaj swap}
Let
\[
A_P(t,q) = \sum_{\pi \in \Lin (P)} t^{\des(\pi)} q^{\maj(\pi)} \qquad \text{and} \qquad A_P^{\co} (t,q) = \sum_{\pi \in \Lin (P)} t^{\des(\pi)} q^{\comaj(\pi)}
\]
for some poset $P$ with $|P| = n$.
Then $A_P(t,q) = A_P^{\co}(tq^n,q^{-1})$.
\end{lem}

\begin{proof}
Given a permutation $\pi \in \Lin(P)$, each $i \in \Des(\pi)$ contributes $i$ to $\maj(\pi)$ and $n-i$ to $\comaj(\pi)$.
Thus $\maj(\pi) + \comaj(\pi) = n \des(\pi)$ and so $\maj(\pi) = n \des(\pi) - \comaj(\pi)$.
Hence
\[
A_P^{\co}(tq^n, q^{-1}) = \sum_{\pi \in \Lin(P)} t^{\des(\pi)}q^{n \des(\pi)}q^{-\comaj(\pi)} = A_P(t,q).  \qedhere
\]
\end{proof}

Next define the weight of a $P$-partition $f$ to be $q^{f(1)+\cdots + f(|P|)}$ and define the order $q$-polynomial $\Omega_P(j,q)$ to be the sum of the weights of all $P$-partitions with parts less~than or equal to $j$.
As before, we reduce strict inequalities to weak inequalities but we need to see what this does to the weight of a given $P$-partition.
Using example~\ref{ex: inequality shift} with $\pi = 51423$, we end up with
\[
0\leq f(5) \leq f(1)-1 \leq f(4)-1 \leq f(2)-2 \leq f(3)-2 \leq j-2.
\]
Adding up these shifted parts of $f$ gives
\[
f(1)+f(2)+\cdots+f(5) -6 = f(1)+\cdots +f(5) - \comaj(\pi).
\]
In general, $\Omega_\pi(j,q)$ is $q^{\comaj(\pi)}$ times the sum of the weights of all weakly increasing sequences of length $n$ with parts less than or equal to $j-\des(\pi)$.
Thus
\[
\Omega_\pi(j,q) = q^{\comaj(\pi)} \qbinom{j + n -\des(\pi)}{n}
\]
and so
\begin{align*}
\sum_{j \geq 0} \Omega_\pi (j,q)t^j & = \sum_{j \geq 0} q^{\comaj(\pi)} \qbinom{ j + n -\des(\pi)}{n} t^j \\
& = t^{\des(\pi)}q^{\comaj(\pi)} \sum_{j \geq 0} \qbinom{ j + n }{ n }  t^j \\
& = \frac{ t^{\des(\pi)}q^{\comaj(\pi)}}{(t;q)_{n+1}}. 
\end{align*}
Hence we have the following theorem.

\begin{thm}
\[
\sum_{j \geq 0} \Omega_P (j,q)t^j = \frac{\dsum_{\pi \in \Lin(P)}t^{\des(\pi)}q^{\comaj(\pi)}}{(t;q)_{n+1}}
\]
\end{thm}

To compute the $q$-Eulerian polynomial $A_n(t,q)$, we again let $P$ be an antichain of $n$ elements.
Then $\Lin(P) = \symn$ and we can easily compute $\Omega_{P(a)} (j,q)$ for a singleton poset $a \in [n]$.
The valid $P(a)$-partitions are of the form $0\leq f(a) \leq j$ and so $\Omega_{P(a)} (j,q) = [j+1]_q$.
The Fundamental Theorem of $P$-partitions implies that if $P_1$ and $P_2$ are disjoint posets then $\Omega_{P_1\sqcup P_2}(j,q) = \Omega_{P_1}(j,q)\Omega_{P_2}(j,q)$.
Thus $\Omega_P (j,q) = [j+1]_q^n$ and we arrive at the following theorem.

\begin{thm}\label{thm: fake Carlitz}
\begin{equation}\label{eq: fake Carlitz}
\sum_{j\geq 0} [j+1]_q^n t^j = \frac{\dsum_{\pi \in \symn}t^{\des(\pi)}q^{\comaj(\pi)}}{(t;q)_{n+1}}
\end{equation}
\end{thm}

Using Lemma~\ref{lem: maj comaj swap}, we substitute $t \leftarrow tq^n$ and $q \leftarrow q^{-1}$ into equation~\eqref{eq: fake Carlitz} and have the following familiar result.

\begin{cor}\label{cor: Carlitz}
\begin{equation}\label{eq: Carlitz}
\sum_{j\geq 0} [j+1]_q^n t^j = \frac{A_n(t,q)}{(t;q)_{n+1}}
\end{equation}
\end{cor}

\begin{proof}
Note that $(tq^n;q^{-1})_{n+1} = (t;q)_{n+1}$ and
\[
[j+1]_{q^{-1}}^n (tq^n)^j = ([j+1]_{q^{-1}} q^j)^n t^j = [j+1]_q^n t^j.  \qedhere
\]
\end{proof}

Equation~\eqref{eq: Carlitz} is often referred to as the Carlitz identity since it appears in \cite{Carlitz1975} but it is actually a special case of a theorem by MacMahon~\cite{MacMahonCombAnalysis}.
In later chapters, we will discuss permutation statistics which were introduced to generalize this equation to the hyperoctahedral group and the colored permutation groups.
Theorem~\ref{thm: fake Carlitz} and Corollary~\ref{cor: Carlitz} suggest that there should be a bijection from $\symn$ to itself preserving descent number and exchanging major index with comajor index.
In fact, one relatively simple bijection is given by writing each permutation in reverse order and then exchanging $i$ with $n+1-i$ for each letter of the permutation.

\begin{thm}\label{thm: maj comaj bijection}
\[
\sum_{\pi \in \symn} t^{\des(\pi)} q^{\maj(\pi)} = \sum_{\pi \in \symn} t^{\des(\pi)} q^{\comaj(\pi)}
\]
\end{thm}

\begin{proof}
We define a bijection from $\symn$ to itself as follows. Given $\pi$ in $\symn$, define a permutation $\sigma$ by setting $\sigma(i) = n+1-\pi(n+1-i)$ for $i = 1,\ldots, n$.
This is clearly an involution.
To see that this map exchanges major index with comajor index, suppose $i \in \Des(\pi)$.
Then $\pi(i) > \pi(i+1)$ and
\[
\sigma(n-i) = n+1-\pi(i+1) > n+1-\pi(i) = \sigma(n-i+1).
\]
Hence $n-i \in \Des(\sigma)$ and we conclude that $\maj(\pi) = \comaj(\sigma)$.
\end{proof}


\section{Multiset enumeration}\label{sec: multiset}

One way to generalize permutations (thought of as words on letters in $[n]$) is to allow for the repetition of letters.
Words allowing such repetition are called \emph{multiset permutations} and are defined as follows.
If $\alpha = (\alpha_1, \ldots, \alpha_m)$ is a weak composition of $n$ then $\{1^{\alpha_1}, 2^{\alpha_2}, \ldots, m^{\alpha_m}\}$ denotes the multiset consisting of $\alpha_1$ 1's, $\alpha_2$ 2's, and so on.
The set of all permutations of this multiset is denoted by $\multn$.
Given $\pi \in \multn$, it is possible that $\pi(i) = \pi(i+1)$ and we do not count this as a descent.
Thus the descent set $\Des(\pi)$ of a multiset permutation $\pi \in \multn$ with $|\alpha| = n$ is defined to be the set of all $i \in [n-1]$ such that $\pi(i) > \pi(i+1)$.
The descent number, major index, and comajor index of a multiset permutation (being determined by the descent set and $n$) are defined as they are for permutations in $\symn$.

Our first goal is to count permutations of a multiset by descent number and major index.
Define
\[
M_\alpha(t,q)=\dsum_{\pi \in \multn} t^{\des(\pi)}q^{\maj(\pi)}
\]
to be the $q$-Eulerian polynomial for $\multn$.
The following theorem gives us a way to compute these $q$-Eulerian polynomials and is originally due to MacMahon~\cite{MacMahonCombAnalysis}.
The barred permutation technique we use to prove the theorem is due to Gessel~\cite{GesselThesis} and will be used and extended in later sections.
We note that the following theorem could also be derived from the work in the previous section.
Though $P$-partitions and barred permutations are basically equivalent, as we discuss after the proof, we~include this detail because we combine the two techniques in the next section to prove the existence of the Eulerian descent algebra.

Define a \emph{barred multiset permutation} to be a shuffle of a permutation $\pi \in \multn$ with a sequence of bars such that the letters of $\pi$ are weakly increasing between bars.
For example, if $\alpha = (2,1,1,2)$ then $241134 \in \multn$ and
\[
24\mid 113\mid 4 \mid \mid
\]
is an example of a barred multiset permutation with $4$ bars. 

\begin{thm}[MacMahon]\label{thm: MacMahon's Formula}
If $\alpha = (\alpha_1, \ldots, \alpha_m)$ and $|\alpha| = n$ then
\begin{equation}\label{eq: MacMahon's Formula}
\sum_{j\geq 0} t^j \prod_{i=1}^m  \qbinom{j+\alpha_i}{\alpha_i}  =\frac{M_\alpha(t,q)}{(t;q)_{n+1}}.
\end{equation}
\end{thm}

\begin{proof}
We prove this theorem by showing that both sides of the equation count weighted barred multiset permutations.
To create a barred multiset permutation given a multiset permutation $\pi \in \multn$, we must first place a bar in each descent of~$\pi$ to ensure that the letters are weakly increasing between bars.
After that, we are free to place any number of bars in any of the spaces between the letters of $\pi$ and on the ends.
These spaces are labeled $0,1,\ldots,n$ from left to right with space $i$ between $\pi(i)$~and~$\pi(i+1)$.
The label of a space counts the number of letters of $\pi$ to the left of the space. 

In order to count multiset permutations by descent number and major index, we weight each bar in space $i$ by $tq^i$ and define the weight of a barred multiset permutation to be the product of the weights of the bars.
Thus the initial bars placed in the descents of $\pi$ contribute a factor of $t^{\des(\pi)}q^{\maj(\pi)}$.
The sum of the weights of all barred multiset permutations with underlying permutation $\pi$ is
\[
t^{\des(\pi)}q^{\maj(\pi)}\prod_{i=0}^{n}(1+tq^i+(tq^i)^2+\cdots) = \frac{t^{\des(\pi)}q^{\maj(\pi)}}{(t;q)_{n+1}}
\]
and a choice of $(tq^i)^k$ in the expansion represents placing $k$ additional bars in space~$i$.
Summing over all $\pi \in \multn$ gives the right side of equation~\eqref{eq: MacMahon's Formula}.

Next we count barred multiset permutations in a different way and reinterpret the weighting of the bars.
Rather than beginning with a multiset permutation, we could instead begin with a fixed number of bars and think of placing letters between the bars.
If we start with $j$ bars then there are $j+1$ slots between the bars (and on the ends).
These slots are called \emph{compartments} and are labeled $0,1, \ldots, j$ from right to~left.
Note that this is the opposite of the labeling convention used for spaces.

Before, each barred multiset permutation was weighted by the product of the weights of each bar and each bar in space $i$ was weighted by $tq^i$.
Thus the exponent of $t$ counts the number of bars and if a bar is in position $i$ then there are $i$ letters of the permutation to the left of the bar.
However, we obtain the same weighting if we instead weight each bar by a factor of $t$ and weight each letter in compartment~$i$ by~$q^i$.
This is because weighting bars by the number of letters to the left of them counts each letter once for each bar to the right of the letter (which is the same as the compartment number).
Thus if we define the order $q$-polynomial $\Omega_\alpha(j,q)$ to be the sum of the weights of the letters over all barred multiset permutations  containing $j$ bars, weighting each letter in compartment $i$ by $q^i$, then we see that
\[
\sum_{j\geq 0} \Omega_\alpha(j,q)t^j  =\frac{M_\alpha(t,q)}{(t;q)_{n+1}}.
\]
The only step remaining is to compute $\Omega_\alpha(j,q)$. 

In the simplest case, $\alpha = (a)$ with $a \geq 0$ and $\multn$ consists of a single multiset permutation given by $a$ identical letters.
It is clear that $\Omega_{(a)}(j,q)$ should be the coefficient of $t^j$ in $1/(t;q)_{a+1}$ since $\multn$ consists of a single multiset permutation $\pi$ with $\des(\pi) = 0 = \maj(\pi)$ and hence $M_{(a)}(t,q)=1$.
The $q$-binomial theorem tells us that
\[
\frac{1}{(t;q)_{a+1}} = \sum_{j \geq 0} \qbinom{j+a}{a} t^j
\]
and hence
\[
\Omega_{(a)}(j,q) = \qbinom{j+a}{a}.
\]

Lastly, suppose $\alpha = (\alpha_1,\ldots, \alpha_m)$.
If we choose compartments for each of the $\alpha_1$~$1$'s, $\alpha_2$~$2$'s, and so on then there is a unique way to place the letters within a given compartment in weakly increasing order (up to permutation of identical letters).
Hence we can think of choosing compartments for the $\alpha_1$ $1$'s first, then choosing compartments for the $\alpha_2$ $2$'s, and so on.
The unique ordering implies that the order $q$-polynomial factors as $\Omega_\alpha(j,q) = \Omega_{(\alpha_1)}(j,q)\Omega_{(\alpha_2)}(j,q)\cdots \Omega_{(\alpha_m)}(j,q)$ and we see that
\[
\Omega_\alpha(j,q) = \prod_{i=1}^m  \qbinom{j+\alpha_i}{\alpha_i}.  \qedhere
\]
\end{proof}

If we set $\alpha = (1,1,\ldots, 1) $ with $|\alpha|=n$ then $\multn = \symn$ and the previous theorem (proven using barred permutations) reduces to Corollary~\ref{cor: Carlitz} (proven using $P$-partitions).
As previously stated, the two techniques are equivalent except for the ever-present choice of order-preserving $P$-partitions and comajor index versus order-reversing $P$-partitions and major index.
To count multiset permutations by number of descents and comajor index, we simply label spaces from right to left and compartments from left to right.
The rest of the proof holds in its entirety and is equivalent to ``$q$-counting'' order-preserving $P$-partitions.
Placing a letter $k$ in compartment $i$ is the same as requiring that the $P$-partition $f$ maps $k$ to $i$.
In the proof of the previous theorem, we $q$-counted order-reversing $P$-partitions.
The ordering choice made this explanation more complicated but prevented the need for the bijection from Theorem~\ref{thm: maj comaj bijection}.


\section{The Eulerian descent algebra}\label{sec: Eulerian descent algebra}

Solomon's descent algebra contains a subalgebra induced by descent number called the \emph{Eulerian descent algebra} whose existence was originally proven by Loday \cite{Loday1989}.
Here the term ``induced'' means that if we define
\[
E_i = \sum_{\des(\pi) = i} \pi
\]
then the product $E_j E_k$ can always be written as a linear combination of the $E_i$ and so the $E_i$ form a basis for an algebra.

To prove that the Eulerian descent algebra exists, we compute
\[
A_\pi(s,t) := \sum_{\sigma \tau = \pi} s^{\des(\sigma)}t^{\des(\tau)}
\]
and show that $A_\pi(s,t)$ is determined by $\des(\pi)$.
This proves that the $E_i$ form a basis for an algebra because the coefficient of $s^jt^k$ in $A_{\pi}(s,t)$ is the coefficient of $\pi$ in the expansion of the product $E_j E_k$.
If we can show that $A_{\pi_1}(s,t) = A_{\pi_2}(s,t)$ whenever $\des(\pi_1)=\des(\pi_2)$ then every permutation with a given number of descents will have the same coefficient in $E_j E_k$ and thus we can always write $E_j E_k$ as a linear combination of the $E_i$.

Every permutation statistic describes an equivalence relation on $\symn$ and we can construct formal sums of the permutations within a given equivalence class as we have done here.
In general, we say a permutation statistic \emph{induces} an algebra if these formal sums of permutations form a basis for an algebra.

The proof of the existence of the Eulerian descent algebra presented here is a modified version of Petersen's proof in \cite{Petersen2005}, which was in turn inspired by Gessel's work on multipartite $P$-partitions in \cite{Gessel1984}.
Petersen's proof uses ``zig-zag'' posets and exploits properties of the set of linear extensions of these posets.
We use the same zig-zag posets as well as a second related collection of posets.

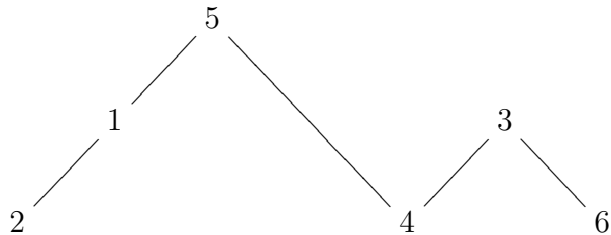
\begin{figure}[htbp]
\[\xymatrix @!R @!C {  & &  5 \ar@{-}[ddrr] & & & & \\   &  1\ar@{-}[ur] & &  &  & 3\ar@{-}[dr] & \\  2\ar@{-}[ur]   &  &  & &  4 \ar@{-}[ur]& & 6   }\] 
\caption{\; The zig-zag poset $\zigipi$ for $\pi = 215436$ and $I = \{3,5\}$.}
\label{fig: zig poset}
\end{figure}

For every $I \subseteq [n-1]$ and $\pi \in \symn$, define the \emph{zig-zag poset} $\zigipi$ by setting $\pi(i) <_Z \pi(i+1)$ if $i \notin I$ and $\pi(i) >_Z \pi(i+1)$ if $i \in I$.
See Figure \ref{fig: zig poset} for an example of a zig-zag poset.
Zig-zag posets are useful to us because of the following lemma describing their linear extensions.

\begin{lem}\label{lem: Zig-Zag Extensions}
Let $\sigma \in \symn$ and consider $\zigipi$ with $\pi \in \symn$ and $I \subseteq [n-1]$.
Then $\sigma \in \Lin(\zigipi)$ if and only if $\Des(\sigma^{-1}\pi) = I$.
\end{lem}

\begin{proof}
Suppose $\sigma \in \Lin(\zigipi)$. We know $\pi(i) <_Z \pi(i+1)$ for $i \notin I$ and so $\sigma^{-1}(\pi(i)) < \sigma^{-1}(\pi(i+1))$ for $i \notin I$. Similarly, $\pi(i) >_Z \pi(i+1)$ for $i \in I$ and so $\sigma^{-1}(\pi(i)) > \sigma^{-1}(\pi(i+1))$ for $i \in I$.
This is equivalent to $\Des(\sigma^{-1}\pi) = I$.
\end{proof}

\begin{ex}
We know $Z(\{3\},2314)$ is the poset with $2<_Z 3 <_Z 1$ and $1 >_Z 4$.
Comparing $4$ with $2$ and $3$, we see that $\Lin(\zigipi) = \{2341,2431,4231\}$ and $\{\sigma^{-1}\pi \, : \, \sigma \in \Lin(\zigipi)\} = \{1243,1342,2341\}$.
Here $\Des(\sigma^{-1}\pi) =  \{3\}$ for every $\sigma \in \Lin(\zigipi)$.
\end{ex}

Next we consider a second collection of posets.
For every $I \subseteq [n-1]$ and $\pi \in \symn$, define the \emph{chain poset} $\chainipi$ by setting $\pi(i) <_C \pi(i+1)$ if $i \notin I$.
These posets are so~named because they consist of disjoint chains, each labeled by a subword of $\pi$.
The following lemma about linear extensions of chain posets is similar to that for zig-zag posets and follows from the same proof.

\begin{lem}\label{lem: Chain Extensions}
Let $\sigma \in \symn$ and consider $\chainipi$ with $\pi \in \symn$ and $I \subseteq [n-1]$.
Then $\sigma \in \Lin(\chainipi)$ if and only if $\Des(\sigma^{-1}\pi) \subseteq I$.
\end{lem}

\begin{ex}
We know $C(\{3\},2314)$ is the poset with $2<_C 3 <_C 1$ and with no relation between $4$ and any other element.
Here $\Lin(\chainipi) = \Lin(\zigipi) \cup \{\pi\}$ and $\Des(\pi^{-1}\pi) = \varnothing \subseteq \{3\}$.
\end{ex}

The key step in proving the existence of the Eulerian descent algebra is extending the usefulness of the bars in barred multiset permutations to posets themselves.
We define barred versions of both zig-zag and chain posets.
First, a \emph{barred zig-zag poset} is defined to be a zig-zag poset $\zigipi$ with an arbitrary number of bars placed in each of the $n+1$ spaces such that between any two bars the elements of the poset (not necessarily their labels) are increasing.
This can be thought of as requiring at least one bar in each ``descent'' of the zig-zag poset, much like requiring bars in descents of barred multiset permutations.
Figure \ref{fig: barred zig poset} gives an example of a barred $\zigipi$ poset with $I = \{3\}$ and $\pi = 21543$.
In this example, we require at least one bar in space $3$ because $\pi(3) >_Z \pi(4)$.

\begin{figure}[htbp]
\[\xymatrix @!R @!C @R=18pt @C=10pt{ \ar@{-}[dddd] &  & \ar@{-}[dddd] & & &  &  \ar@{-}[dddd] &  \ar@{-}[dddd] & &  & \\  &  & & & &  5 \ar@{-}[ddrrr] &  &  & &  & \\  &  &  & 1\ar@{-}[urr] &  & & &  & & & 3 \\  & 2\ar@{-}[urr]  &  &  & & &  &  & 4 \ar@{-}[urr]& &   \\  &   &  &  & & &  &  & & &  }\] 
\caption{\; A barred $\zigipi$ poset for $I = \{3\}$ and $\pi = 21543$.}
\label{fig: barred zig poset}
\end{figure}

To construct a barred zig-zag poset we begin with a zig-zag poset $\zigipi$ then we must first place a bar in space $i$ for each $i \in I$.
We are then free to place any number of bars in any of the $n+1$ spaces.
Define $\Omega_{\zigipi}(j,k)$ to be the number of ordered pairs $(f,P)$ where $P$ is a barred $\zigipi$ poset with $k$ bars and $f$ is a $\zigipi$-partition with parts less than or equal to $j$.
Recall that $\sigma \in \Lin(\zigipi)$ if and only if $\Des(\sigma^{-1}\pi) = I$ and note that a barred $\zigipi$ poset must have at least $|I|$ bars.
If we begin with the zig-zag poset $\zigipi$ then there is a unique way to place the first $|I|$ bars (place one bar in space $i$ for each $i\in I$).
Next, there are
\[
\mchoose{n+1}{k-|I|} = \binom{n+k - |I|}{n}
\]
ways to place the remaining $k-|I|$ bars in the $n+1$ available spaces and hence there are $\binom{n+k-|I|}{n}$ barred $\zigipi$ posets with $k$ bars.
Thus
\[
\Omega_{\zigipi}(j,k) = \sum_{\sigma \in \Lin(\zigipi)} \Omega_{\sigma}(j) \binom{n+k-\des(\sigma^{-1}\pi)}{n} = \sum_{\sigma \in \Lin(\zigipi)} \Omega_{\sigma}(j) \Omega_{\sigma^{-1}\pi}(k).
\]
If we set $\tau = \sigma^{-1}\pi$ and sum over all $I \subseteq [n-1]$ then we have
\begin{equation}\label{eq: barred zig order sum}
\sum_{I \subseteq [n-1]} \Omega_{\zigipi}(j,k) = \sum_{\sigma \tau = \pi} \Omega_{\sigma}(j) \Omega_{\tau}(k).
\end{equation}

Next we define a \emph{barred chain poset} to be a chain poset $\chainipi$ with at least one bar in space $i$ for each $i \in I$ and with an arbitrary number of bars placed on either end.
No bars are allowed in space $j$ for $j \in [n-1] \setminus I$.
Thus we place at least one bar between each chain of the chain poset and allow for bars on either end.
Figure~\ref{fig: barred chain poset} gives an example of a barred $\chainipi$ poset with $I = \{1,3\}$ and $\pi = 21543$.

\begin{figure}[htbp]
\[\xymatrix @!R @!C @C=10pt{  \ar@{-}[ddd] &  & \ar@{-}[ddd] & & &   &  \ar@{-}[ddd] &  \ar@{-}[ddd] & &  & \\  &  &  & & &  5  &  &   & &   & 3 \\   & 2  &  & 1\ar@{-}[urr]  & & & &  & 4 \ar@{-}[urr] &  &  \\  & & & & & & &  & & &  }\] 
\caption{\; A barred $\chainipi$ poset for $I = \{1,3\}$ and $\pi = 21543$.}
\label{fig: barred chain poset}
\end{figure}

If we begin with a chain poset $\chainipi$ then we must first place a bar in space $i$ for each $i \in I$.
From there we are free to place any number of bars in space $i$ for any $i \in I$ and any number of bars on either end.
This creates a collection of bars with each compartment between adjacent bars containing at most one nonempty chain labeled by a subword of $\pi$.
Define $\Omega_{\chainipi}(j,k)$ to be the number of ordered pairs $(f,P)$ where $P$ is a barred $\chainipi$ poset with $k$ bars and $f$ is a $\chainipi$-partition with parts less than or equal to $j$.
Recall that $\sigma \in \Lin(\chainipi)$ if and only if $\Des(\sigma^{-1}\pi) \subseteq I$ and note that a barred $\chainipi$ poset must have at least $|I|$ bars.
If we begin with the chain poset $\chainipi$ then there is a unique way to place the first $|I|$ bars (place one bar in space $i$ for each $i\in I$).
Next there are
\[
\mchoose{|I|+2}{k-|I|} = \binom{k + 1}{k-|I|}
\]
ways to place the remaining $k-|I|$ bars in the $|I|+2$ available spaces and hence there are $\binom{k+1}{k-|I|}$ barred $\chainipi$ posets with $k$ bars.
Thus
\[
\Omega_{\chainipi}(j,k) = \sum_{\sigma \in \Lin(\chainipi)} \Omega_{\sigma}(j) \binom{k+1}{k-|I|}.
\]

The following lemma allows us to compare $P$-partitions for barred zig-zag posets and barred chain posets.

\begin{lem}\label{Barred Poset Extensions}
For every $\pi \in \symn$,
\[
\sum_{I \subseteq [n-1]}  \Omega_{\zigipi}(j,k) = \sum_{I \subseteq [n-1]} \Omega_{\chainipi}(j,k).
\]
\end{lem}

\begin{proof}
We first rewrite the equation as
\[
\sum_{\sigma \in \symn}  \Omega_{\sigma}(j)Z_{\sigma}(k) = \sum_{\sigma \in \symn} \Omega_{\sigma}(j)C_{\sigma}(k)
\]
where $Z_{\sigma}(k)$ is the number of barred zig-zag posets with $k$ bars such that $\sigma$ is a linear extension of the underlying zig-zag poset and $C_{\sigma}(k)$ is the number of barred chain posets with $k$ bars such that $\sigma$ is a linear extension of the underlying chain poset.
We must find a bijection showing that for every $\sigma \in \symn$, the number of barred zig-zag posets with $k$ bars such that $\sigma$ is a linear extension of the underlying zig-zag poset is the same as the number of barred chain posets with $k$ bars such that $\sigma$ is a linear extension of the underlying chain poset.
This bijection is simply the map that sends each barred zig-zag poset to the barred chain poset obtained by removing the relation between $\pi(i)$ and $\pi(i+1)$ for every space $i$ containing at least one bar.
Thus a barred zig-zag poset with underlying poset $\zigipi$ is mapped to a barred chain poset with underlying poset $C(J,\pi)$ and $J \supseteq I$ since barred zig-zag posets are required to have bars in space $i$ for every $i \in I$.
\end{proof}

The previous bijection maps Figure~\ref{fig: barred zig poset} to Figure~\ref{fig: barred chain poset}.
Now that all the pieces are in place, we are ready to prove the existence of the Eulerian descent algebra.

\begin{thm}\label{Symmetric Algebra Thm}
For every $\pi \in \symn$,
\begin{equation}\label{eq: Symmetric Algebra}
\sum_{j,k \geq 0} \binom{(j+1)(k+1) +n-1-\des(\pi)}{n} s^j t^k = \frac{A_{\pi}(s,t)}{(1-s)^{n+1}(1-t)^{n+1}}.
\end{equation}
\end{thm}

\begin{proof}
We have
\begin{align*}
\frac{\dsum_{\sigma\tau = \pi} s^{\des(\sigma)} t^{\des(\tau)}}{(1-s)^{n+1}(1-t)^{n+1}} &=  \sum_{j,k \geq 0} \sum_{\sigma \tau = \pi} \binom{j+n}{n}\binom{k+n}{n}s^{j+\des(\sigma)} t^{k+\des(\tau)}\\
&= \sum_{j,k \geq 0} \sum_{\sigma \tau = \pi} \binom{j+n-\des(\sigma)}{n}\binom{k+n-\des(\tau)}{n}s^j t^k \\ 
&= \sum_{j,k \geq 0} \sum_{I \subseteq [n-1]} \Omega_{\zigipi}(j,k) s^j t^k
\end{align*}
where the last equality follows from equations~\eqref{eq: pi order poly formula} and \eqref{eq: barred zig order sum}.
By Lemma~\ref{Barred Poset Extensions}, we can switch from zig-zag posets to chain posets and have
\[
\frac{\dsum_{\sigma\tau = \pi} s^{\des(\sigma)} t^{\des(\tau)}}{(1-s)^{n+1}(1-t)^{n+1}} = \sum_{j,k \geq 0} \sum_{I \subseteq [n-1]} \Omega_{\chainipi}(j,k) s^j t^k.
\]
The only remaining step is to prove that
\[
\sum_{I \subseteq [n-1]} \Omega_{\chainipi}(j,k) =  \binom{(j+1)(k+1) +n-1-\des(\pi)}{n}.
\]

First we note that $\sum_{I \subseteq [n-1]} \Omega_{\chainipi}(j,k)$ counts ordered pairs $(f,P)$ where $P$ is a barred $\chainipi$ poset with $k$ bars for some $I \subseteq [n-1]$ and $f$ is a $\chainipi$-partition with parts less than or equal to $j$.
Fix a barred $\chainipi$ poset with $k$ bars and use the bars to define compartments labeled $0,\ldots,k$ from left to right.
Then define $\pi_i$ to be the (possibly empty) subword of $\pi$ in compartment $i$ and denote the length of $\pi_i$ by $L_i$.
Then
\[
\Omega_{\chainipi}(j) = \prod_{i=0}^{k} \Omega_{\pi_i}(j) 
\]
For $i=0,\ldots,k$, we let $\Omega_{\pi_i}(j)$ count solutions to the inequalities
\[
i(j+1) \leq s_{i_1} \leq  \cdots \leq s_{i_{L_i}} \leq i(j+1) +j
\]
with $s_{i_l} < s_{i_{l+1}}$ if $l \in \Des(\pi_i)$.
By concatenating these inequalities, we see that if we sum over all $I \subseteq [n-1]$ and all barred $\chainipi$ posets with $k$ bars then $\sum_{I \subseteq [n-1]} \Omega_{\chainipi}(j,k)$ is equal to the number of solutions to the inequalities
\[
0 \leq s_1 \leq \cdots \leq s_n \leq jk+j+k
\]
with $s_i < s_{i+1}$ if $i \in \Des(\pi)$.
Hence we conclude that
\[
\sum_{I \subseteq [n-1]} \Omega_{\chainipi}(j,k)  = \binom{(j+1)(k+1)+n-1-\des(\pi)}{n} .
\qedhere
\]
\end{proof}

Figure \ref{fig: every barred chain poset} depicts every barred $C(I,21)$ poset with $2$ bars and $I \subseteq \{1\}$.
Thus we can identify each barred $\chainipi$ poset with a barred permutation with underlying permutation $\pi$.
This provides a visual representation of the last step in the proof of the previous theorem.

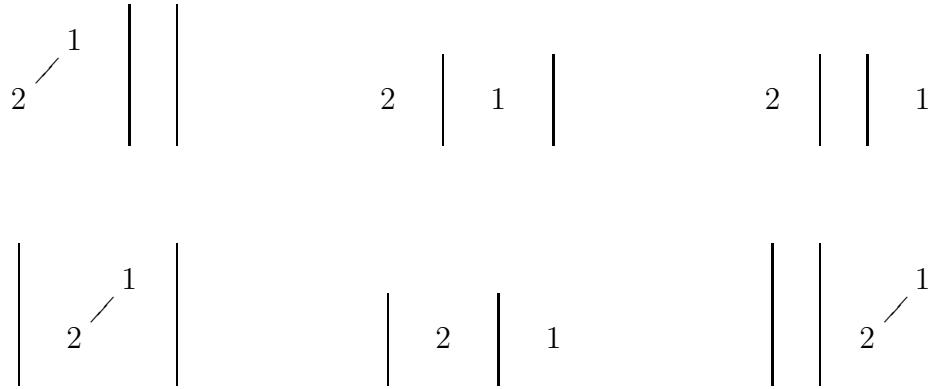
\begin{figure}[htbp]
\[\xymatrix  @R=8pt @C=8pt{ && \ar@{-}[ddd] & \ar@{-}[ddd] &&&&&&&&&&  \\   & 1 &&&&& \ar@{-}[dd]  && \ar@{-}[dd]  &&& \ar@{-}[dd] & \ar@{-}[dd] & \\  2\ar@{-}[ur] &&&& \qquad\qquad & 2 && 1 && \qquad\qquad & 2 &&& 1  \\    &&&&&&&&&&&&& \\ &&&&&&&&&&&&& \\  \ar@{-}[ddd] &&& \ar@{-}[ddd] &&&&&&& \ar@{-}[ddd] & \ar@{-}[ddd] && \\ && 1 &&& \ar@{-}[dd] && \ar@{-}[dd] &&&&&& 1 \\  & 2\ar@{-}[ur] &&&&& 2 &  & 1 &&&& 2 \ar@{-}[ur]  \\ &&&&&&&&&&&&& \\     }\] 
\caption{\; Barred $C(I,21)$ posets with $2$ bars and $I \subseteq \{1\}$.}
\label{fig: every barred chain poset}
\end{figure}

The previous theorem shows that $A_{\pi}(s,t)$ is determined by $\des(\pi)$ and thus proves the existence of the Eulerian descent algebra.
With a little work, we can find another basis for this algebra with nice properties.
Define the structure polynomial $\phi(x)$ in the group algebra of $\symn$ by
\[
\phi(x) = \dsum_{\pi \in \symn} \binom{x+n-1-\des(\pi)}{n} \pi.
\]
We have the following theorem of Petersen \cite[Theorem 1.1]{Petersen2005} although the idea originally appears in \cite{MielnikPlebanski1970}.

\begin{thm}
As polynomials in $x$ and $y$ with coefficients in the group algebra of the symmetric group,
\[
\phi(x)\phi(y) = \phi(xy).
\]
\end{thm}

\begin{proof}
By comparing coefficients of $s^jt^k$ on both sides of equation \eqref{eq: Symmetric Algebra}, we see that
\[
\binom{(j+1)(k+1) +n-1-\des(\pi)}{n} = \sum_{\sigma\tau = \pi}  \binom{j+n -\des(\sigma)}{n} \binom{k+n - \des(\tau)}{n}
\]
for all $j,k \geq 0$.
Multiplying both sides by $\pi$ and summing over all $\pi \in \symn$ shows that $\phi(x)\phi(y) = \phi(xy)$ for all positive integers. This implies that the theorem holds as polynomials in $x$ and $y$.
\end{proof}

If we expand the structure polynomial
\[
\phi(x) = \dsum_{\pi \in \symn} \binom{x+n-1-\des(\pi)}{n} \pi = \sum_{i=1}^n e_i x^i
\]
then the property $\phi(x)\phi(y) = \phi(xy)$ can be rewritten as
\[
\sum_{i,j = 1}^n e_ie_jx^iy^j = \sum_{i=1}^n e_i (xy)^i.
\]
From this we conclude that $e_i^2 = e_i$ and $e_i e_j = 0$ if $i \neq j$.
Hence the $e_i$ are orthogonal idempotents which also span the Eulerian descent algebra.


\section{The cyclic Eulerian descent module}\label{sec: cyclic Eulerian module}

Define the \emph{cyclic descent set} $\cDes(\pi)$ of a permutation $\pi \in \symn$ to be the set of all $i \in [n]$ such that $\pi(i) > \pi(i+1)$ with $\pi(n+1) := \pi(1)$.
Thus $n \in \cDes(\pi)$ if and only if $\pi(n) > \pi(1)$.
We let $\cdes(\pi) = |\cDes(\pi)|$ denote the number of cyclic descents of $\pi$.
The group algebra of the symmetric group $\symn$ contains a subalgebra called the \emph{cyclic Eulerian descent algebra} which is induced by cyclic descent number.
The cyclic Eulerian descent algebra is the span of the basis elements $C_i$ defined by
\[
C_i = \sum_{\cdes(\pi) = i} \pi
\]
for $i=1,\ldots,n-1$.
This algebra is in fact isomorphic to the Eulerian descent algebra on $\mathfrak{S}_{n-1}$.
For a proof of the existence of the cyclic Eulerian descent algebra see Cellini \cite{CelliniI1995, Cellini1998} or Petersen \cite{Petersen2005}.
In this section, we prove an unexpected result about the relationship between the Eulerian descent algebra and the cyclic Eulerian descent algebra.
We compute
\[
A_\pi^c (s,t) := \sum_{\sigma \tau = \pi} s^{\des(\sigma)}t^{\cdes(\tau)}
\]
and show that $A_\pi^c (s,t)$ is determined by $\cdes(\pi)$.
This shows that the product $E_j C_k$ can be written as a linear combination of the basis elements for the cyclic Eulerian descent algebra.
Hence the cyclic Eulerian descent algebra is a left module over the Eulerian descent algebra.
Cellini \cite{Cellini1998} mentioned that these two algebras do not commute but it is not clear whether she knew about the module relationship or not.

As before, we define zig-zag and chain posets but the rules are slightly modified to correspond to cyclic descents.
Let $\cycliczigsets$ denote the set of all nonempty proper subsets of $[n]$.
For every $I \in \cycliczigsets$ and $\pi \in \symn$, define the \emph{cyclic zig-zag poset} $\czigipi$ by setting $\pi(i) <_{Z^c} \pi(i+1)$ if $i \notin I$ and $\pi(i) >_{Z^c} \pi(i+1)$ if $i \in I$ where $\pi(n+1) = \pi(1)$.
Both $I = \varnothing$ and $I=[n]$ would give $\pi(1) <_{Z^c} \pi(1)$ and thus are excluded from $\cycliczigsets$.
We include both $\pi(1)$ and $\pi(n+1)$ in diagrams of our cyclic zig-zag posets even though they represent the same element in the poset.
We choose to include both because it presents a cleaner looking Hasse diagram and because it helps with later discussion.
Figure \ref{fig: cyc zig poset} shows a cyclic zig-zag poset for $I = \{2,3\}$ and $\pi = 2143$.
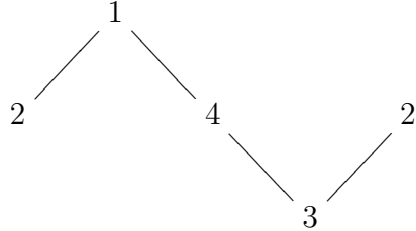
\begin{figure}[htbp]
\[\xymatrix @!R @!C {  & 1\ar@{-}[dr] & & &  \\  2\ar@{-}[ur] &  & 4\ar@{-}[dr] &   & 2  \\   & & & 3\ar@{-}[ur] &  }\] 
\caption{\; The cyclic zig-zag poset $\czigipi$ for $I = \{2,3\}$ and $\pi = 2143$.}
\label{fig: cyc zig poset}
\end{figure}

\begin{lem}\label{lem: Cyclic Zig-Zag Extensions}
Let $\sigma \in \symn$ and consider $\czigipi$ with $\pi \in \symn$ and $I \in \cycliczigsets$.
Then $\sigma \in \Lin(\czigipi)$ if and only if $\cDes(\sigma^{-1}\pi) = I$.
\end{lem}

\begin{proof}
The same arguments from the proof of Lemma \ref{lem: Zig-Zag Extensions} also show that $\sigma^{-1}(\pi(n)) < \sigma^{-1}(\pi(1))$ for $n \notin I$ and $\sigma^{-1}(\pi(n)) > \sigma^{-1}(\pi(1))$ for $n \in I$.
\end{proof}

\begin{ex}
If $\czigipi$ is the poset in Figure~\ref{fig: cyc zig poset} with $\pi = 2143$ and $I=\{2,3\}$ then $\Lin(\czigipi) = \{3241, 3421\} $ and $\{\sigma^{-1}\pi \, : \, \sigma \in \Lin(\czigipi)\} = \{2431, 3421\}$.
Here $\cDes(\sigma^{-1}\pi) =  \{2, 3\}$ for all $\sigma \in \Lin(\czigipi)$.
\end{ex}

Next we consider a second collection of posets.
Let $\cyclicchainsets$ denote the set of all nonempty subsets of $[n]$.
For every $I \in \cyclicchainsets$ and $\pi \in \symn$, define the \emph{cyclic chain poset} $\cchainipi$ by setting $\pi(i) <_{C^c} \pi(i+1)$ if $i \notin I$.
Note that we now allow $I=[n]$ since this corresponds to an antichain.
As before, $\pi(n+1) = \pi(1)$ and it will be useful to think about two different ways to visually represent a cyclic chain poset.
We can either leave $\pi(1)$ and $\pi(n+1)$ distinct in the diagram of a given cyclic chain poset even though they represent they same element or we can equate them in the diagram.
Doing so prefixes the chain ending in $\pi(n+1)$ to the chain beginning with $\pi(1)$.
Figure~\ref{fig: cyc chain poset} shows the two diagrams of the cyclic chain poset $\cchainipi$ for $I = \{2,3\}$ and $\pi = 2143$.
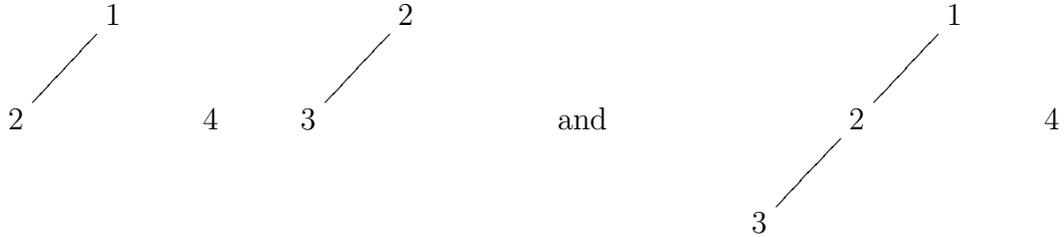
\begin{figure}[htbp]
\[\xymatrix @!R  {  & 1 & &  & 2 & & & & 1 & \\  2\ar@{-}[ur] &  & 4 & 3\ar@{-}[ur]   & & \qquad \text{and} \qquad & & 2\ar@{-}[ur] & & 4  \\  &&&&&& 3\ar@{-}[ur] &&&   }\] 
\caption{\; Two diagrams of $\cchainipi$ for $I = \{2,3\}$ and $\pi = 2143$.}
\label{fig: cyc chain poset}
\end{figure}

\begin{lem}\label{Cyclic Chain Extensions}
Let $\sigma \in \symn$ and consider $\cchainipi$ with $\pi \in \symn$ and $I \in \cyclicchainsets$.
Then $\sigma \in \Lin(\cchainipi)$ if and only if $\cDes(\sigma^{-1}\pi) \subseteq I$.
\end{lem}

\begin{ex}
If $n=4$, $\pi = 2143$, and $I = \{2,3\}$ then $\cchainipi$ is the poset in Figure \ref{fig: cyc chain poset}.
Then $\Lin(\cchainipi) = \Lin(\czigipi) \cup \{3214, 4321\}$ and
\[
\{\sigma^{-1}\pi \, : \, \sigma = 3214, 4321\} = \{2341, 3412\}.
\]
Here $\cDes(2341) = \{3\}$ and $\cDes(3412) = \{2\}$.
\end{ex}

Next we define barred versions of both cyclic zig-zag posets and cyclic chain posets.
First, a \emph{barred cyclic zig-zag poset} is defined to be a cyclic zig-zag poset $\czigipi$ with an arbitrary number of bars placed in each of the $n$ internal spaces between $\pi(1)$ and $\pi(n+1)$ such that between any two bars the elements of the poset (not necessarily their labels) are increasing.
This can be thought of as requiring at least one bar in each ``descent'' of the cyclic zig-zag poset.
We do not allow bars on either end and so $\pi(1)$ will always be in the leftmost compartment and $\pi(n+1)$ will always be in the rightmost compartment.
We will see in later discussion that the equivalence of $\pi(1)$ and $\pi(n+1)$ does not cause problems with our definition of barred cyclic zig-zag posets.
Figure~\ref{fig: barred cyc zig poset} gives an example of a barred $\czigipi$ poset with $I = \{2,3\}$ and $\pi = 2143$.
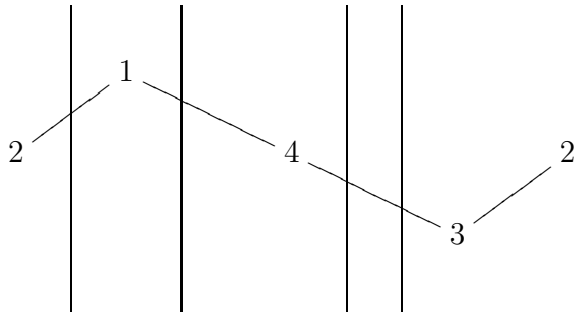
\begin{figure}[htbp]
\[\xymatrix @!R @!C @R=16pt @C=8pt{ & \ar@{-}[dddd] && \ar@{-}[dddd] &&& \ar@{-}[dddd] & \ar@{-}[dddd] &&& \\ && 1\ar@{-}[drrr] &&&&&&&& \\  2\ar@{-}[urr] &&&&& 4\ar@{-}[drrr] &&&&& 2  \\  &&&&&&&& 3\ar@{-}[urr] && \\  &&&&&&&&&&  }\] 
\caption{\; A barred $\czigipi$ poset with $I = \{2,3\}$ and $\pi = 2143$.}
\label{fig: barred cyc zig poset}
\end{figure}

To create a barred cyclic zig-zag poset, we begin with a cyclic zig-zag poset $\czigipi$ and must first place a bar in space $i$ for each $i \in I$.
From there we are free to place any number of bars in any of the $n$ spaces.
Define $\Omega_{\czigipi}(j,k)$ to be the number of ordered pairs $(f,P)$ where $P$ is a barred $\czigipi$ poset with $k$ bars and $f$ is a $\czigipi$-partition with parts less than or equal to $j$.
Recall that $\sigma \in \Lin(\czigipi)$ if and only if $\cDes(\sigma^{-1}\pi) = I$.
If we begin with the cyclic zig-zag poset $\czigipi$ then there is a unique way to place the first $|I|$ bars (place one bar in space $i$ for each $i\in I$).
Next there are
\[
\mchoose{n}{k-|I|} = \binom{k+n-1 - |I|}{n-1}
\]
ways to place the remaining $k-|I|$ bars in the $n$ allowable spaces and hence there are $\binom{k+n-1-|I|}{n-1}$ barred $\czigipi$ posets with $k$ bars.
Thus
\[
\Omega_{\czigipi}(j,k) = \sum_{\sigma \in \Lin(\czigipi)} \Omega_{\sigma}(j) \binom{k+n-1-\cdes(\sigma^{-1}\pi)}{n-1}.
\]
If we set $\tau = \sigma^{-1}\pi$ and sum over all $I \in \cycliczigsets$ then we have
\begin{equation}\label{eq: barred cyclic equiv}
\sum_{I \in \cycliczigsets} \Omega_{\czigipi}(j,k) = \sum_{\sigma \tau = \pi} \Omega_{\sigma}(j)\binom{k+n-1-\cdes(\tau)}{n-1} .
\end{equation}

We define a \emph{barred cyclic chain poset} to be a cyclic chain poset $\cchainipi$ with at least one bar in space $i$ for each $i \in I$.
No bars are allowed in space $j$ for any $j \notin I$.
Thus we place at least one bar between each chain of the chain poset.
We can again choose whether or not to equate $\pi(1)$ and $\pi(n+1)$ in our visual depiction of a given barred cyclic chain poset.
Figure \ref{fig: barred cyc chain poset} shows the two diagrams of a barred $\cchainipi$ poset with $I = \{1,2,3\}$ and $\pi = 2143$.
\begin{figure}[htbp]
\[\xymatrix @!R  @C=10pt{ & \ar@{-}[ddd] && \ar@{-}[ddd] && \ar@{-}[ddd] & \ar@{-}[ddd] &&&&&& \ar@{-}[ddd] && \ar@{-}[ddd] && \ar@{-}[ddd] & \ar@{-}[ddd]  \\  &&&&&&&& 2 &&& 2 &&&&&& \\  2 && 1 && 4 &&& 3\ar@{-}[ur]  && \qquad \text{and} \qquad & 3\ar@{-}[ur]  &&& 1 && 4 && \\  &&&&&&&&&&&&&&&&& \\ }\] 
\caption{\;  Two diagrams of a barred $\cchainipi$ poset with $I = \{1,2,3\}$ and $\pi = 2143$. }
\label{fig: barred cyc chain poset}
\end{figure}
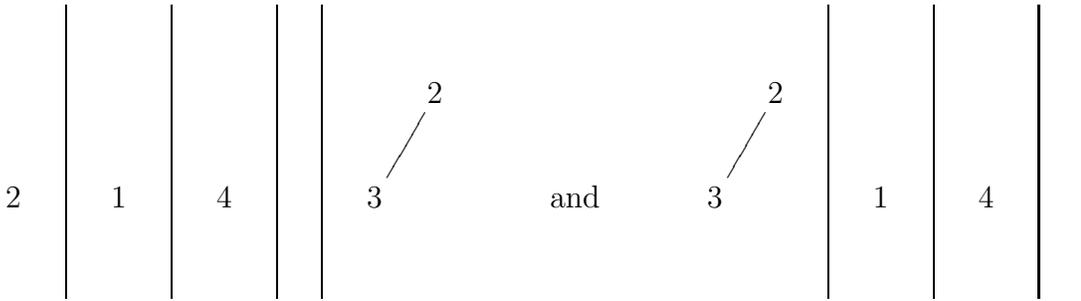

To create a barred cyclic chain poset, we begin with a cyclic chain poset $\cchainipi$ and must first place a bar in space $i$ for each $i \in I$.
From there we are free to place any number of bars in space $i$ for any $i \in I$.
This creates a collection of bars with each compartment between adjacent bars containing at most one nonempty chain.
Define $\Omega_{\cchainipi}(j,k)$ to be the number of ordered pairs $(f,P)$ where $P$ is a barred $\cchainipi$ poset with $k$ bars and $f$ is a $\cchainipi$-partition with parts less than or equal to $j$.
Recall that $\sigma \in \Lin(\cchainipi)$ if and only if $\cDes(\sigma^{-1}\pi) \subseteq I$.
If we begin with the cyclic chain poset $\cchainipi$ then there is a unique way to place the first $|I|$ bars (place one bar in space $i$ for each $i\in I$).
Next there are
\[
\mchoose{|I|}{k-|I|} = \binom{k - 1}{k-|I|}
\]
ways to place the remaining $k-|I|$ bars in the $|I|$ allowable spaces and hence there are $\binom{k-1}{k-|I|}$ barred $\cchainipi$ posets with $k$ bars.
Thus
\[
\Omega_{\cchainipi}(j,k) = \sum_{\sigma \in \Lin(\cchainipi)} \Omega_{\sigma}(j) \binom{k-1}{k-|I|}.
\]

The following lemma allows us to compare $P$-partitions for barred cyclic zig-zag posets and barred cyclic chain posets.

\begin{lem}\label{Barred Cyclic Poset Extensions}
For every $\pi \in \symn$,
\[
\sum_{I \in \cycliczigsets}  \Omega_{\czigipi}(j,k) = \sum_{I \in \cyclicchainsets} \Omega_{\cchainipi}(j,k).
\]
\end{lem}

\begin{proof}
We prove this lemma by showing that for every $\sigma \in \symn$, the number of barred cyclic zig-zag posets with $k$ bars such that $\sigma$ is a linear extension of the underlying cyclic zig-zag poset is the same as the number of barred cyclic chain posets with $k$ bars such that $\sigma$ is a linear extension of the underlying cyclic chain poset.
The desired bijection is the map that sends each barred cyclic zig-zag poset with underlying poset $\czigipi$ to the barred cyclic chain poset obtained by removing the relation between $\pi(i)$ and $\pi(i+1)$ for every space $i$ containing at least one bar.
\end{proof}

The bijection in the previous lemma maps Figure~\ref{fig: barred cyc zig poset} to Figure~\ref{fig: barred cyc chain poset}.
We are now ready to compute $A_{\pi}^c (s,t)$.

\begin{thm}\label{des cdes module}
For every $\pi \in \symn$,
\begin{equation}\label{eq: des cdes module}
\sum_{j,k \geq 0} (j+1)\binom{(j+1)k +n-1-\cdes(\pi)}{n-1} s^j t^k = \frac{A_{\pi}^c (s,t)}{(1-s)^{n+1}(1-t)^{n}}.
\end{equation}
\end{thm}

\begin{proof}
We have
\begin{align*}
\frac{\dsum_{\sigma\tau = \pi} s^{\des(\sigma)} t^{\cdes(\tau)}}{(1-s)^{n+1}(1-t)^n} &=  \sum_{j,k \geq 0} \sum_{\sigma \tau = \pi} \binom{j+n}{n}\binom{k+n-1}{n-1}s^{j+\des(\sigma)} t^{k+\cdes(\tau)}\\
&= \sum_{j,k \geq 0} \sum_{\sigma \tau = \pi} \binom{j+n-\des(\sigma)}{n}\binom{k+n-1-\cdes(\tau)}{n-1}s^j t^k \\ 
&= \sum_{j,k \geq 0} \sum_{I \in \cycliczigsets}  \Omega_{\czigipi}(j,k) s^j t^k
\end{align*}
where the last equality follows from equation \eqref{eq: barred cyclic equiv}.
Switching from cyclic zig-zag posets to cyclic chain posets by Lemma \ref{Barred Cyclic Poset Extensions}, we have
\[
\frac{\dsum_{\sigma\tau = \pi} s^{\des(\sigma)} t^{\cdes(\tau)}}{(1-s)^{n+1}(1-t)^n} = \sum_{j,k \geq 0} \sum_{I \in \cyclicchainsets}  \Omega_{\cchainipi}(j,k) s^j t^k.
\]
The only remaining step is to prove that
\[
\sum_{I \in \cyclicchainsets}  \Omega_{\cchainipi}(j,k) =  (j+1)\binom{(j+1)k +n-1-\cdes(\pi)}{n-1}.
\]

First we note that $\sum_{I \in \cyclicchainsets} \Omega_{\cchainipi}(j,k)$ counts ordered pairs $(f,P)$ where $P$ is a barred $\cchainipi$ poset with $k$ bars for some $I \in \cyclicchainsets$ and $f$ is a $\cchainipi$-partition with parts less than or equal to $j$.
Fix a barred $\cchainipi$ poset with $k$ bars, equating $\pi(1)$ and $\pi(n+1)$, and use the bars to define compartments labeled $0,\ldots,k-1$ from left to right.
Then define $\pi_i$ to be the (possibly empty) word in compartment $i$ and denote the length of $\pi_i$ by $L_i$.
Then
\[
\Omega_{\cchainipi}(j) = \prod_{i=0}^{k-1} \Omega_{\pi_i}(j) 
\]
For $i=0,\ldots,k-1$, we let $\Omega_{\pi_i}(j)$ count solutions to the inequalities
\[
i(j+1) \leq s_{i_1} \leq  \cdots \leq s_{i_{L_i}} \leq i(j+1) +j
\]
with $s_{i_l} < s_{i_{l+1}}$ if $l \in \Des(\pi_i)$.
By concatenating these inequalities, we see that if we sum over all $I \in \cyclicchainsets$ and all barred $\cchainipi$ posets with $k$ bars then we have $\pi$-partitions which seem to wrap around and describe a cyclic shuffle of $\pi$ with $k$ bars.
To compute a formula, we instead think of barred cyclic chain posets in which $\pi(1)$ and $\pi(n+1)$ are not visually equated.
The $k$ bars define $k+1$ compartments labeled $0,\ldots,k$ from left to right.
Again, define $\pi_i$ to be the (possibly empty) word in compartment $i$ and denote the length of $\pi_i$ by $L_i$.
For $i=0,\ldots,k-1$, we let $\Omega_{\pi_i}(j)$ count solutions to the inequalities
\[
i(j+1) \leq s_{i_1} \leq  \cdots \leq s_{i_{L_i}} \leq i(j+1) +j
\]
with $s_{i_l} < s_{i_{l+1}}$ if $l \in \Des(\pi_i)$.
For $i=k$, we count solutions to the inequalities
\[
k(j+1) \leq s_{k_1} \leq  \cdots \leq s_{n} \leq s_1+k(j+1)
\]
with $s_{k_l} < s_{k_{l+1}}$ if $l \in \Des(\pi_k)$ and with $s_{n} < s_1 + k(j+1)$ if $n \in \cDes(\pi)$.
There are $j+1$ choices for the value of $s_1$ and for each choice we concatenate these inequalities.
Thus if we sum over all $I \in \cyclicchainsets$ and all barred $\cchainipi$ posets with $k$ bars then we see that $\sum_{I \in \cyclicchainsets} \Omega_{\cchainipi}(j,k)$ is equal to the number of solutions to the inequalities
\[
a = s_1 \leq s_2 \leq \cdots \leq s_n \leq s_{n+1} = a+k(j+1)
\]
with $s_i < s_{i+1}$ if $i \in \cDes(\pi)$.
Hence we conclude that
\[
\sum_{I \in \cyclicchainsets}  \Omega_{\cchainipi}(j,k)  = (j+1)\binom{(j+1)k+n-1-\cdes(\pi)}{n-1}.  \qedhere
\]
\end{proof}

Figure \ref{fig: every barred cyclic chain poset} depicts every barred $C^c(I,213)$ poset with $2$ bars and $I \subseteq \{1,2,3\}$.
This gives a visual representation of the last step in the proof of the previous theorem.

\begin{figure}[htbp]
\[\xymatrix  @R=11pt @C=7pt{  &&& \ar@{-}[dddd] & \ar@{-}[dddd] &&&&&&&&&&& \ar@{-}[dddd] & \ar@{-}[dddd] &&& \\  && 3 &&&&&&& \ar@{-}[ddd] & \ar@{-}[ddd] &&&&&&&&& 2 \\  & 1\ar@{-}[ur] &&&&&&& 1 &&&& 2 &&&&&& 3\ar@{-}[ur] & \\  2\ar@{-}[ur] &&&&& 2 & \qquad & 2\ar@{-}[ur] &&&& 3\ar@{-}[ur] && \qquad & 2 &&& 1\ar@{-}[ur] && \\  &&&&&&&&&&&&&&&&&&& \\  &&&&&&&&&&&&&&&&&&& \\  & \ar@{-}[ddd] && \ar@{-}[ddd] &&&&& \ar@{-}[ddd] &&& \ar@{-}[ddd] &&&&& \ar@{-}[ddd] && \ar@{-}[ddd] & \\  &&&&& 2 &&&&& 3 &&&&& 1 &&&& \\  2 && 1 && 3\ar@{-}[ur] &&& 2 && 1\ar@{-}[ur] &&& 2 && 2\ar@{-}[ur] &&& 3 && 2 \\  &&&&&&&&&&&&&&&&&&& \\ \ar@{-}[rrrrrrrrrrrrrrrrrrr]  &&&&&&&&&&&&&&&&&&& \\ &&& \ar@{-}[dddd] & \ar@{-}[dddd] &&&&&&&&&&&&&&  \ar@{-}[dddd] &  \ar@{-}[dddd] \\  && 3 &&&&&&& \ar@{-}[ddd] && \ar@{-}[ddd] &&&&&& 1 && \\  & 1\ar@{-}[ur] &&&&&&& 1 &&&&&&&& 2\ar@{-}[ur] &&& \\  2\ar@{-}[ur] &&&&&&& 2\ar@{-}[ur] &&& 3 &&&&& 3\ar@{-}[ur] &&&& \\  &&&&&&&&&&&&&&&&&&& \\  &&&&&&&&&&&&&&&&&&& \\  &&&&&&&&&&&&&&&&&& \ar@{-}[dddd] & \ar@{-}[dddd] \\  & \ar@{-}[ddd] &&& \ar@{-}[ddd] &&&&& \ar@{-}[ddd] && \ar@{-}[ddd] &&&&&& 2 && \\  &&& 3 &&&&& 2 &&&&&&&& 3\ar@{-}[ur] &&& \\  2 && 1\ar@{-}[ur] &&&&& 3\ar@{-}[ur] &&& 1 &&&&& 1\ar@{-}[ur] &&&& \\  &&&&&&&&&&&&&&&&&&&   }\] 
\caption{\; Barred $C^c(I,213)$ posets with $2$ bars and $I \subseteq \{1,2,3\}$.}
\label{fig: every barred cyclic chain poset}
\end{figure}
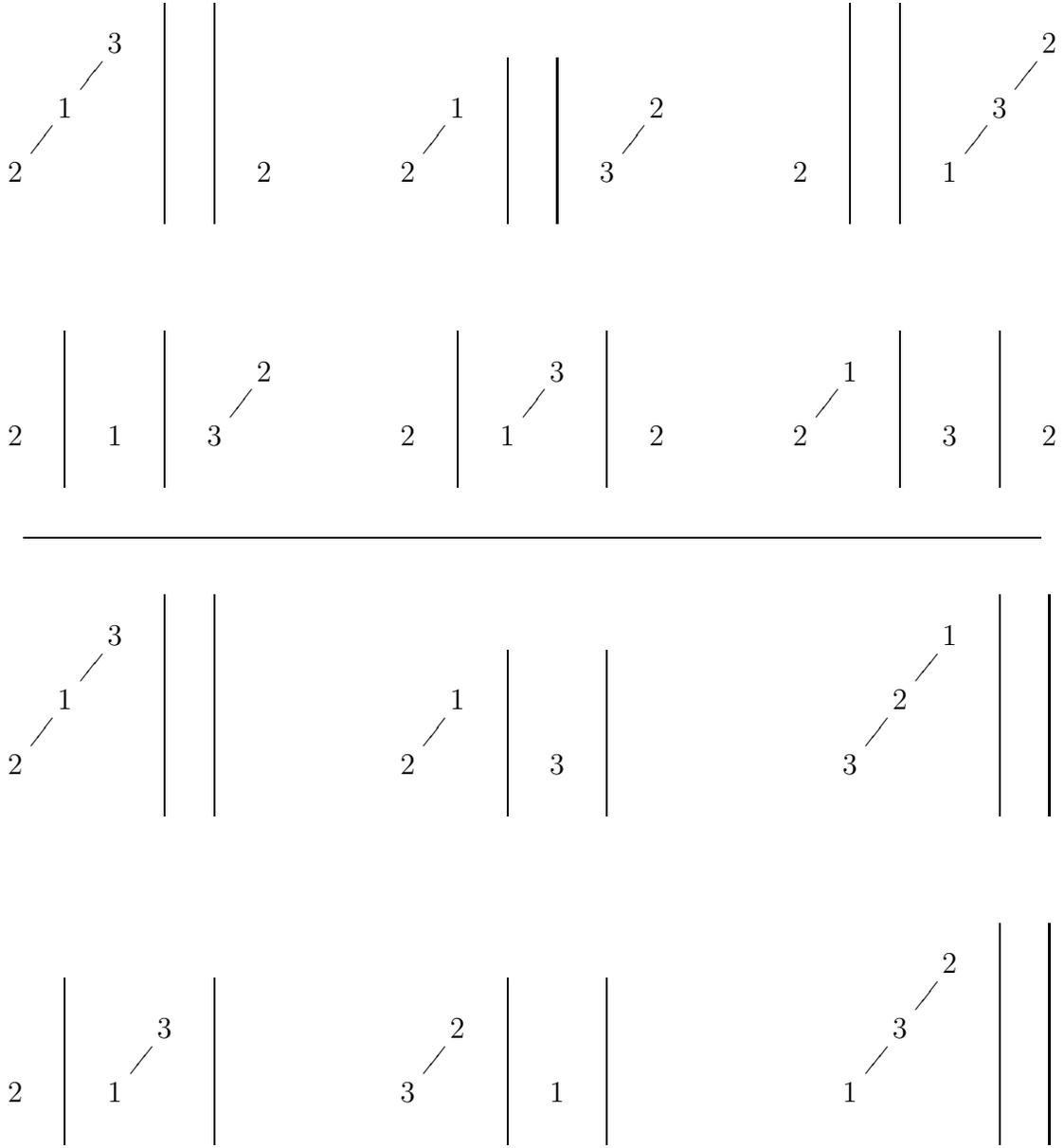

Define the cyclic structure polynomial $\psi(x)$ in the group algebra of $\symn$ by
\[
\psi(x) = \sum_{\pi \in \symn} \frac{x}{n}\binom{x+n-1-\cdes(\pi)}{n-1} \pi.
\]
Both Petersen \cite{Petersen2005} and Cellini \cite{Cellini1998} described orthogonal idempotents in the cyclic Eulerian descent algebra.
In particular, Petersen \cite{Petersen2005} showed that if we expand
\[
\psi(x)= \dsum_{\pi \in \symn} \frac{x}{n}\binom{x+n-1-\cdes(\pi)}{n-1} \pi = \dsum_{i=2}^n c_i x^i
\]
then the $c_i$ terms are orthogonal idempotents which span the cyclic Eulerian descent algebra.

\begin{thm}
As polynomials in $x$ and $y$ with coefficients in the group algebra of $\symn$,
\[
\phi(x)\psi(y) = \psi(xy).
\]
\end{thm}

\begin{proof}
By comparing coefficients of $s^jt^k$ on both sides of equation \eqref{eq: des cdes module}, we have
\[
(j+1)\binom{(j+1)k +n-1-\cdes(\pi)}{n-1} = \sum_{\sigma\tau = \pi}  \binom{j+n -\des(\sigma)}{n} \binom{k+n -1- \cdes(\tau)}{n-1}
\]
for all $j,k \geq 0$.
Multiplying both sides by $\frac{k}{n}\pi$ and summing over all $\pi \in \symn$ shows that $\phi(x)\psi(y) = \psi(xy)$ for all positive integers and thus the equation holds as polynomials in $x$ and $y$.
\end{proof}

The property $\phi(x)\psi(y) = \psi(xy)$ can be rewritten as
\[
\sum_{i=1}^n\sum_{j = 2}^n e_i c_j x^i y^j = \sum_{i=2}^n c_i (xy)^i.
\]
From this we conclude that $e_ic_i = c_i$ and $e_i c_j = 0$ if $i \neq j$.
Thus the cyclic Eulerian descent algebra is a left module over the Eulerian descent algebra.
However, it is not true that $\psi(x)\phi(y) = \psi(xy)$ nor is it true that $\psi(x)\phi(y) = \phi(xy)$.
The cyclic Eulerian descent algebra is not a right module over the Eulerian descent algebra.

\begin{ex}
Let $n=4$.
Then $C_1 = 1234+2341+3412+4123$ is the formal sum of all permutations in $\mathfrak{S}_4$ with one cyclic descent and is a basis element for the cyclic Eulerian descent algebra.
Also,
\[
E_1 = 1243+1324+1342+1423+2134+2314+2341+2413+3124+3412+4123
\]
is the formal sum of all permutations in $\mathfrak{S}_4$ with one descent and is a basis element for the Eulerian descent algebra.
We compute 
\begin{align*}
C_1 E_1 = & 3\cdot1234 + 2\cdot1243 + 1324 + 3\cdot1342 + 2\cdot1423 + 2\cdot2134 + 2\cdot2314 \\
&+ 3\cdot2341 + 3\cdot2413 + 2431 + 3\cdot3124 + 3142 + 2\cdot3241 + 3\cdot3412 \\
&+ 2\cdot3421 + 3\cdot4123 + 2\cdot4132 + 4213 + 3\cdot4231 + 2\cdot4312
\end{align*}
where the notation $j\cdot\pi$ means that the permutation $\pi$ has a coefficient of $j$.
Inspection reveals that both $1243$ and $1324$ have two cyclic descents  and both appear in $E_1$.
However, they have different coefficients in the product $C_1 E_1$ and so $C_1 E_1$ cannot be written as a linear combination of basis elements from either the Eulerian descent algebra or from the cyclic Eulerian descent algebra.
This result is somewhat surprising given that the cyclic Eulerian descent algebra on $\symn$ is isomorphic to the Eulerian descent algebra on $\mathfrak{S}_{n-1}$.
\end{ex}


\section{Counting permutations and inverse permutations by descent number and major index}\label{sec: des ides maj imaj}

It would be nice to find a formula for
\[
A_{\pi}(t_1,t_2,q_1,q_2) := \sum_{\sigma \tau = \pi} t_1^{\des(\sigma)} t_2^{\des(\tau)} q_1^{\maj(\sigma)} q_2^{\maj(\tau)}.
\]
However, descent number and major index first fail to induce an algebra when ${n=5}$ and so $A_{\pi} (t_1,t_2,q_1,q_2)$ will not, in general, be determined solely by $\des(\pi)$ and $\maj(\pi)$.
Instead, we restrict our attention to the special case where $\pi$ is the identity permutation in $\symn$.
For every $\pi \in \symn$, let $\ides(\pi)$ denote $\des(\pi^{-1})$ and $\imaj(\pi)$ denote $\maj(\pi^{-1})$.
Then define
\[
A_n (t_1,t_2,q_1,q_2) = \sum_{\pi \in \symn} t_1^{\des(\pi)} t_2^{\ides(\pi)} q_1^{\maj(\pi)} q_2^{\imaj(\pi)}.
\]
Garsia and Gessel \cite{GarsiaGessel1979} originally found a generating function for these multivariate polynomials and the purpose of this section is to show how we can use barred posets to obtain the same result.
However, since we choose to work with order-preserving $P$-partitions, we first define $\icomaj(\pi) = \comaj(\pi^{-1})$ and compute
\[
A_n^{\co} (t_1,t_2,q_1,q_2) = \sum_{\pi \in \symn} t_1^{\des(\pi)} t_2^{\ides(\pi)} q_1^{\comaj(\pi)} q_2^{\icomaj(\pi)}.
\]

\begin{thm}
\[
\sum_{i,j \geq 0} t_1^i t_2^j \prod_{k = 0}^i \prod_{l = 0}^j \frac{1}{(1-zq_1^kq_2^l)} = \sum_{n\geq 0} z^n\,  \frac{A_n^{\co}(t_1,t_2,q_1,q_2)}{(t_1;q_1)_{n+1}(t_2;q_2)_{n+1}}
\]
\end{thm}

\begin{proof}
Recall that if we define the order $q$-polynomial $\Omega_P(j,q)$ to be the sum of the weights of all $P$-partitions $f$ with parts less than or equal to $j$ where $f$ is weighted by $q$ to the sum of its parts then
\[
\sum_{j\geq 0} \Omega_P(j,q)t^j = \frac{\dsum_{\pi \in \Lin(P)}t^{\des(\pi)}q^{\comaj(\pi)}}{(t;q)_{|P|+1}}.
\]
Following the proof of the existence of the Eulerian descent algebra, we associate to each possible descent set $I \subseteq [n-1]$ the zig-zag poset $Z(I,12\cdots n)$ and the chain poset $C(I,12\cdots n)$.
For this section we denote them by $Z(I)$ and $C(I)$ respectively.
As before, we consider barred zig-zag posets and barred chain posets with the rules for bar placement the same as in Section \ref{sec: Eulerian descent algebra}.
Define $\Omega_{Z(I)}(i,j,q_1,q_2)$ to be the sum of the product of the weights of all ordered pairs $(f,P)$ where $P$ is a barred $Z(I)$ poset with $j$ bars and $f$ is a $Z(I)$-partition with parts less than or equal to $i$.
Given a barred $Z(I)$ poset, weight each bar in position $l$ by $q_2^{n-l}$ and define the weight of a barred $Z(I)$ poset to be the product of the weight of the bars.
Weight each $Z(I)$-partition $f$ by $q_1^{f(1)+\cdots+f(n)}$.
Then
\[
\Omega_{Z(I)}(i,j,q_1,q_2) = \sum_{\pi \in \Lin(Z(I))} q_1^{\comaj(\pi)} \qbinomplain{i+n-\des(\pi)}{n}_{q_1} q_2^{\icomaj(\pi)} \qbinomplain{j+n-\ides(\pi)}{n}_{q_2}
\]
and so we have
\[
\frac{A_n^{\co} (t_1,t_2,q_1,q_2)}{(t_1;q_1)_{n+1}(t_2;q_2)_{n+1}} = \sum_{i,j \geq 0} \sum_{I \subseteq [n-1]} \Omega_{Z(I)}(i,j,q_1,q_2) t_1^i t_2^j.
\]

Next we note that the bijection used in Lemma~\ref{Barred Poset Extensions} to move between zig-zag posets and chain posets shows that
\[
\sum_{I \subseteq [n-1]} \Omega_{Z(I)}(i,j,q_1,q_2) =  \sum_{I \subseteq [n-1]} \Omega_{C(I)}(i,j,q_1,q_2)
\]
where $\Omega_{C(I)}(i,j,q_1,q_2)$ is the sum of the products of the weights of all ordered pairs $(f,P)$ where $P$ is a barred $C(I)$ poset with $j$ bars and $f$ is a $C(I)$-partition with parts less than or equal to $i$.
Thus our goal is to compute
\[
\sum_{n \geq 0} z^n \sum_{i,j \geq 0} \sum_{I \subseteq [n-1]} \Omega_{C(I)}(i,j,q_1,q_2) t_1^i t_2^j.
\]

It turns out that we can find an elegant formula if we consider all $n \geq 0$ at the same time.
Assume there are $j$ bars creating $j+1$ compartments labeled $0,1,\ldots, j$ from left to right.
We place naturally labeled chains of arbitrary length in each compartment.
If we want the order $q$-polynomial for the chain in compartment $l$ we simply choose how many letters we want to map to each $k$ with $0 \leq k \leq i$.
We weight each letter by $z$ to keep track of the length of the chain and we also weight each letter in compartment $l$ by $q_2^l$.
As before, weighting bars by the number of letters to the right ($n-i$ for space $i$) is the same as weighting letters by the number of bars to the left (compartment number).
The sum of the weights of all such chains in compartment $l$ is
\[
\prod_{k = 0}^i \frac{1}{(1-zq_1^k q_2^l)}.
\]
We then take the product over all $l$ with $0\leq l \leq j$.
Finally, summing over $i$ and $j$ completes the proof.
\end{proof}

The last step is to convert from comajor index to major index. Lemma~\ref{lem: maj comaj swap} isn't immediately helpful since we are summing over $n$ but the bijection described in Theorem~\ref{thm: maj comaj bijection} has nice properties for $\ides$ and $\imaj$ as well.
Thus we arrive at the following theorem of Garsia and Gessel~\cite{GarsiaGessel1979}.

\begin{thm}[Garsia and Gessel]
\[
\sum_{i,j \geq 0} t_1^i t_2^j \prod_{k = 0}^i \prod_{l = 0}^j \frac{1}{(1-zq_1^k q_2^l)} = \sum_{n\geq 0} z^n \, \frac{A_n (t_1,t_2,q_1,q_2)}{(t_1;q_1)_{n+1}(t_2;q_2)_{n+1}}
\]
\end{thm}

\begin{proof}
We prove the theorem by showing that for every $n\geq 0$ we have
\[
A_n (t_1,t_2,q_1,q_2) = A_n^{\co} (t_1,t_2,q_1,q_2).
\]
Consider the bijection from $\symn$ to itself mapping $\pi$ to $\sigma$ where $\sigma(i) = n + 1 - \pi(n+1-i)$ for $i = 1,\ldots , n$.
This amounts to writing $\pi$ backwards and then swapping $j$ with $n+1-j$ for each letter.
We proved in Theorem \ref{thm: maj comaj bijection} that this bijection preserves descent number and exchanges major index and comajor index.
It also preserves the number of inverse descents as well as exchanges $\imaj$ with $\icomaj$.
To see this, assume that $n-i \in \Des(\pi^{-1})$.
Then $n+1-i$ comes before $n-i$ in $\pi$.
If we first reverse $\pi$ then we have $n-i$ before $n+1-i$.
But then this implies that $n+1-(n-i) = i+1$ comes before $n+1-(n+1-i) = i$ in $\sigma$.
Thus $i \in \Des(\sigma^{-1})$ and so $\ides(\pi) = \ides(\sigma)$ and $\icomaj(\pi) = \imaj(\sigma)$.
\end{proof}


\chapter{The Hyperoctahedral Group}

The goal of this chapter is to extend the work of Chapter~\ref{chpt: Symmetric Group} to the hyperoctahedral group which is also the Coxeter group of type~$B_n$.
This group has a combinatorial interpretation as the group of ``signed'' permutations of length~$n$ and contains a type~$B$ analogue of the Eulerian descent algebra.
However, by altering the definition of the descent set of a signed permutation, we produce multiple algebras.
The existence of these algebras is not new but we include existence proofs using barred versions of ``signed'' posets because we later examine the relationships between the various algebras by combining techniques from these individual proofs.
The main results of this chapter are the existence of an algebra induced by Adin, Brenti, and Roichman's flag descent number and the study of the relationships between the different algebras.
The chapter is organized as follows.

First, signed permutations and signed posets are defined in Section~\ref{sec: signed permutations and signed posets}.
In Section~\ref{sec: signed multiset permutations}, we count a signed generalization of multiset permutations by descents and major index using a variation of the barred permutation technique.
These formulas allow us to derive results of Chow and Gessel~\cite{ChowGessel2007} concerning the flag major index statistic defined by Adin and Roichman~\cite{AdinRoichman2001}.
In Sections~\ref{sec: type B Eulerian descent algebra}, \ref{sec: type A Eulerian descent algebra}, and \ref{sec: augmented descent algebra}, we use different versions of ``signed'' $P$-partitions to prove that several descent statistics induce algebras.
These include Chow's $P$-partitions of type~$B$ defined in~\cite{ChowThesis2001} and Petersen's augmented $P$-partitions defined in~\cite{Petersen2005}.
Adin, Brenti, and Roichman~\cite{AdinBrentiRoichman2001} introduced the flag descent number and we prove in Section~\ref{sec: flag descent algebra} that the flag descent number also induces an algebra.
Finally, in Section~\ref{sec: ideals}, we prove that some of the algebras of Sections~\ref{sec: type B Eulerian descent algebra}, \ref{sec: type A Eulerian descent algebra}, and \ref{sec: augmented descent algebra} are ideals in larger algebras, expanding on a result of Petersen~\cite{Petersen2005}.


\section{Signed permutations and signed posets}\label{sec: signed permutations and signed posets}

Let $[0,n]$ denote the set $\{0,\ldots,n\}$ and let $\pm[n]$ denote the set
\[
\{-n,\ldots,-1,0,1,\ldots,n\}.
\]
The \emph{hyperoctahedral group} $\hypn$ is the group of all bijections $\pi:\pm[n] \rightarrow \pm[n]$ such that $\pi(-i) = -\pi(i)$ for $i=0,1,\ldots,n$.
Thus $\pi(0)=0$ and $\pi$ is determined by the images of the letters in $[n]$.
We call these bijections \emph{signed permutations} and write them in one-line notation as $\pi = \pi(1)\pi(2)\cdots\pi(n)$.
For readability, we write $\bar{a}$ for $-a$.

There are multiple ways to extend the notion of a descent set to signed permutations.
The first approach is to define the \emph{type $A$ descent set} $\Desa(\pi)$ of a signed permutation $\pi$ to be the set of all $i \in [n-1]$ such that $\pi(i) > \pi(i+1)$.
We let $\desa(\pi) = |\Desa(\pi)|$ denote the number of type $A$ descents of $\pi$.
In practice, this often produces uninteresting results since it fails to capture the additional signed structure of the letters.
It is included here for its use in Sections \ref{sec: type A Eulerian descent algebra} and \ref{sec: ideals}.

The much more frequently used definition is the \emph{type $B$ descent set} $\Desb(\pi)$ of a signed permutation $\pi$ which is defined to be the set of all $i \in [0,n-1]$ such that $\pi(i) > \pi(i+1)$.
Thus $0 \in \Desb(\pi)$ if and only if $\pi(1) < 0$.
This definition is consistent with the definition of descents for general Coxeter groups.
We let $\desb(\pi) = |\Desb(\pi) |$ denote the number of type $B$ descents of $\pi$.

The third descent set is the \emph{augmented descent set} $\aDes(\pi)$ of a signed permutation~$\pi$, defined to be the set of all $i\in [0,n]$ such that $\pi(i) > \pi(i+1)$ with $\pi(n+1) := 0$.
Thus $n\in \aDes(\pi)$ if and only if $\pi(n)>0$.
We let $\ades(\pi) = |\aDes(\pi)|$ denote the number of augmented descents of $\pi$.
Augmented descents are a natural extension of cyclic descents to $\hypn$ (see \cite{CelliniI1995}).

Finally, we define the major index of a signed permutation.
Since a descent in position $0$ does not count towards the sum of the positions of the descents of a signed permutation, we set
\[
\maj(\pi) = \sum_{i\in \Desa(\pi)} i = \sum_{i\in \Desb(\pi)} i
\]
and
\[
\amaj(\pi) = \sum_{i\in \aDes(\pi)} i.
\]

\begin{ex}
The signed permutation $\pi = \bar{5}14\bar{2}3$ is in $\mathfrak{B}_5$ and $\Desa(\pi) = \{3\}$, $\Desb(\pi) = \{0,3\}$, and $\aDes(\pi) = \{0,3,5\}$.
Thus $\desa(\pi) = 1$, $\desb(\pi)=2$, and $\ades(\pi) = 3$.
Lastly, $\maj(\pi) = 3$ and $\amaj(\pi) = 8$.
\end{ex}

In order to extend our previous results, we must extend the definition of \linebreak $P$-partitions to account for the addition of signs to the permutations.
The following definition and theory are due to Chow \cite{ChowThesis2001}, though Chow's results derive from Reiner's earlier work in \cite{Reiner1992, ReinerSP1993}.
Define a \emph{signed poset} to be a partial order $P$ on the set $\pm[n]$ such that if $x <_P y$ in $P$ then $\bar{x} >_P \bar{y}$ in $P$.
Every signed permutation $\pi$ can be represented by the signed poset
\[
\pi(\bar{n}) <_P \cdots  <_P \pi(\bar{1}) <_P 0 <_P \pi(1) <_P \cdots <_P \pi(n)
\]
and we denote by $P(\pi)$ the signed poset
\[
\pi(1) <_P \cdots <_P \pi(n)
\]
with no relation between $0$ and $\pi(i)$ for any $i \in [n]$.

We define the set $\Lin(P)$ to be the set of all signed permutations in $\hypn$ which are linear extensions of $P$.
A signed permutation $\pi$ is in $\Lin(P)$ if and only if $i <_P j$ implies that $\pi^{-1}(i) < \pi^{-1}(j)$.
See Figure \ref{fig: signed poset extensions} for an example of a signed poset $P$ and the linear extensions of $P$.

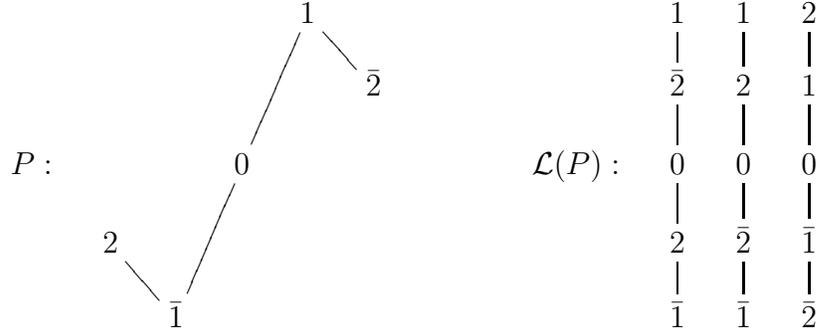
\begin{figure}[htbp]
\[\xymatrix @R=1pc @C=1pc  {& & & & 1\ar@{-}[dr]  \ar@{-}[ddl] & & & & & 1\ar@{-}[d] & 1\ar@{-}[d] & 2\ar@{-}[d]  \\   & & & & & \bar2 & & & & \bar2\ar@{-}[d] & 2\ar@{-}[d] & 1\ar@{-}[d] \\ P: & & & 0  \ar@{-}[ddl]  & & & & & \Lin(P): &  0\ar@{-}[d] & 0\ar@{-}[d] & 0\ar@{-}[d] \\  & 2\ar@{-}[dr] & & & & & & & & 2\ar@{-}[d] & \bar2\ar@{-}[d] & \bar1\ar@{-}[d]  \\ &  & \bar1 & &  & & & & & \bar1 & \bar1 & \bar2}\] 
\caption{\; Linear extensions of a signed poset $P$.}
\label{fig: signed poset extensions}
\end{figure}

We refrain from discussing ``signed'' $P$-partitions because we will need a different definition for each type of descent set.
However, we can use the barred permutation technique to count signed permutations by descent number and major index.


\section{Multiset enumeration}\label{sec: signed multiset permutations}

There are two natural ways to extend the notion of a multiset permutation to a signed multiset permutation, both of which are seen in \cite{HaglundLoehrRemmel2005}.
One way is to consider all the different ways to arrange the elements of a multiset in which the multiplicities of $s$ and $\bar{s}$ are chosen independently.
In practice, this often produces uninteresting results since we lose the connection between $s$ and $\bar{s}$.
We adopt the second approach and define the set of signed multiset permutations as follows.
Given $\multn$ as in the unsigned case, define $\hypmultn$ to be the set of all signed words from $\multn$, that is, the set consisting of all words from $\multn$ together with all possible choices of signs for the letters of each word in $\multn$.
The goal of this section is to count words in $\hypmultn$ by descent number and major index for extensions of the different descent sets from the previous section.

The type $A$ descent set $\Desa(\pi)$ of a word $\pi \in \hypmultn$ is the set of all $i \in [n-1]$ such that $\pi(i) > \pi(i+1)$.
As before, $\desa(\pi) = |\Desa(\pi)|$.
We then define
\[
\desb(\pi)  = \left\{\begin{array}{ll} \desa(\pi), & \text{if } \pi(1) > 0 \\  \desa(\pi)+1, & \text{if } \pi(1) < 0  \end{array}\right.
\]
and
\[
\ades(\pi) = \left\{\begin{array}{ll} \desb(\pi), & \text{if } \pi(n) < 0 \\  \desb(\pi)+1, & \text{if } \pi(n) > 0  \end{array}\right. .
\]
We define $\maj(\pi)$ to be the sum of the positions of all type $A$ descents (which is also the sum of the positions of all type $B$ descents).
Similarly, we define $\amaj(\pi)$ to be the sum of the positions of all augmented descents.

In order to count permutations in $\hypmultn$ by descent number and major index, define a \emph{barred signed multiset permutation} to be a shuffle of a signed multiset permutation $\pi \in \hypmultn$ with a sequence of bars such that the letters of $\pi$ are weakly increasing between bars.
For example,
\[
24\mid1\mid\bar{3}\mid66\mid\mid \bar{5} 3
\]
is a barred signed multiset permutation with $5$ bars.
First we compute
\[
B_\alpha^A (t,q) := \sum_{\pi \in \hypmultn} t^{\desa(\pi)}q^{\maj(\pi)}.
\]
The technique is similar to the unsigned case and we count barred signed multiset permutations in two different ways.
To create a barred signed multiset permutation given a signed multiset permutation $\pi \in \hypmultn$, we must first place a bar in each descent of $\pi$ to ensure that the letters are weakly increasing between bars.
After that, we are free to place any number of bars in any of the spaces between the letters of $\pi$ and on the ends.
These spaces are labeled $0,1,\ldots,n$ from left to right with space $i$ between $\pi(i)$ and $\pi(i+1)$.
The label of a space counts the number of letters of $\pi$ to the left of the space. 

In order to count signed multiset permutations by type A descent number and major index, we weight each bar in space $i$ by $tq^i$ and define the weight of a barred signed multiset permutation to be the product of the weights of the bars.
Thus the initial bars placed in the descents of $\pi$ contribute a factor of $t^{\desa(\pi)}q^{\maj(\pi)}$.
The sum of the weights of all barred signed multiset permutations with underlying signed multiset permutation $\pi$ is
\[
t^{\desa(\pi)}q^{\maj(\pi)}\prod_{i=0}^{n}(1+tq^i+(tq^i)^2+\cdots) = \frac{t^{\desa(\pi)}q^{\maj(\pi)}}{(t;q)_{n+1}}
\]
and a choice of $(tq^i)^k$ in the expansion represents placing $k$ additional bars in space~$i$.
Summing over all $\pi \in \hypmultn$ yields
\[
\frac{B_\alpha^A (t,q)}{(t;q)_{n+1}}.
\]

We now count barred signed multiset permutations a different way and reinterpret the weighting of the bars.
Rather than beginning with a signed multiset permutation, we could instead begin with a fixed number of bars and think of placing letters between the bars.
Given $j$ bars there are $j+1$ compartments labeled $0,1, \ldots, j$ from right to~left.

Before, each barred signed multiset permutation was weighted by the product of the weights of the bars and each bar in space $i$ was weighted by $tq^i$.
Thus the exponent of $t$ counts the number of bars and if a bar is in position $i$ then there are $i$ letters of the signed multiset permutation to the left of the bar.
However, we obtain the same weighting if we instead weight each bar by a factor of $t$ and weight each letter in compartment $i$ by $q^i$.
Given a barred signed multiset permutation, we define $w(\pi)$ to be the product of the weights of the letters of $\pi$ and define the order $q$-polynomial $\OmegaA_\alpha (j,q):= \sum w(\pi)$ where the sum is over all barred signed multiset permutations containing $j$ bars.
Then
\[
\sum_{j\geq 0} \Omega_\alpha^A (j,q)t^j  = \frac{B_\alpha^A (t,q)}{(t;q)_{n+1}}.
\]
The only step remaining is to compute $\OmegaA_\alpha (j,q)$. 

If we place a collection of letters (possibly positive and negative) into a compartment then there is a unique way to place those letters in weakly increasing order (up to permutation of identical letters).
Thus each placement corresponds to a unique barred signed multiset permutation and this implies that the order $q$-polynomial $\OmegaA_\alpha (j,q)$ factors as
\[
\OmegaA_\alpha (j,q) = \OmegaA_{(\alpha_1)} (j,q) \OmegaA_{(\alpha_2)} (j,q) \cdots \OmegaA_{(\alpha_m)} (j,q).
\]
For $i=1,\ldots,m$, we must choose a sign and a compartment for each of the $\alpha_i$ $i$'s. If we choose $k$ letters to be negative then the sum of the products of their weights is $\qbinom{j+k}{k}$.
This would leave $\alpha_i-k$ letters positive and the sum of the products of their weights is $\qbinom{j+\alpha_i -k}{\alpha_i-k}$.
Summing over $k$, we see that
\begin{equation}\label{eq: type A omega sum}
\OmegaA_{(\alpha_i)} (j,q) = \sum_{k=0}^{\alpha_i}  \qbinom{j+k}{ k} \qbinom{j+\alpha_i-k }{ \alpha_i-k}
\end{equation}
and we have the following theorem.

\begin{thm}
\[
\sum_{j\geq 0} t^j \prod_{i=1}^m \left(\sum_{k=0}^{\alpha_i}  \qbinom{j+k}{ k} \qbinom{j+\alpha_i-k }{ \alpha_i-k}\right) = \frac{B_\alpha^A (t,q)}{(t;q)_{n+1}}
\]
\end{thm}

Next we compute
\[
B_\alpha^B (t,q) := \sum_{\pi \in \hypmultn} t^{\desb(\pi)}q^{\maj(\pi)}.
\]
We first begin with $\pi \in \hypmultn$ and place a bar in every type $B$ descent, weighting each bar in position $i$ by $tq^i$.
Next we are free to place bars in any of the positions $0,1,\ldots,n$ and thus the sum of the weights of all barred signed multiset permutations with underlying signed multiset permutation $\pi$ is $t^{\desb(\pi)}q^{\maj(\pi)}/(t;q)_{n+1}$.

Next we begin with $j$ bars which form $j+1$ labeled compartments.
Note that if we are counting barred signed multiset permutations by type $B$ descents then no negative letters can go in the leftmost compartment.
Thus we define $\OmegaB_\alpha (t,q)$ to be the sum of the products of the weights of the letters over all barred signed multiset permutations with $j$ bars such that there are no negative letters in the leftmost compartment.
As before, the order $q$-polynomial factors as
\[
\OmegaB_\alpha (j,q) = \OmegaB_{(\alpha_1)} (j,q) \OmegaB_{(\alpha_2)} (j,q) \cdots \OmegaB_{(\alpha_m)} (j,q).
\]
By considering the number of negative letters, we have
\begin{equation}\label{eq: type B omega sum}
\OmegaB_{(\alpha_i)} (j,q) = \sum_{k=0}^{\alpha_i}  \qbinom{j+k-1}{ k} \qbinom{j+\alpha_i-k }{ \alpha_i-k}
\end{equation}
and we have the following theorem.

\begin{thm}
\[
\sum_{j\geq 0} t^j \prod_{i=1}^m \left(\sum_{k=0}^{\alpha_i}  \qbinom{j+k-1}{ k} \qbinom{j+\alpha_i-k }{ \alpha_i-k}\right) = \frac{B_\alpha^B (t,q)}{(t;q)_{n+1}}
\]
\end{thm}

Lastly, we consider the augmented descent number and compute
\[
B^{(a)}_\alpha (t,q) := \sum_{\pi \in \hypmultn} t^{\ades(\pi)} q^{\amaj(\pi)}.
\]
We first begin with $\pi \in \hypmultn$ and place a bar in every augmented descent, weighting each bar in position $i$ by $tq^i$.
Next we are free to place bars in any of the positions $0,1,\ldots,n$ and thus the sum of the weights of all barred signed multiset permutations with underlying signed multiset permutation $\pi$ is $t^{\ades(\pi)}q^{\amaj(\pi)}/(t;q)_{n+1}$.

Next we begin with $j$ bars which form $j+1$ labeled compartments.
Note that if we are counting barred signed multiset permutations with augmented descents then no negative letters can go in the leftmost compartment and no positive letters can go in the rightmost compartment.
Thus we define $\Omegaaug_\alpha (t,q)$ to be the sum of the products of the weights of the letters over all barred signed multiset permutations with $j$ bars such that there are no negative letters in the leftmost compartment and no positive letters in the rightmost compartment.
As before, the order $q$-polynomial factors as
\[
\Omegaaug_\alpha (j,q) = \Omegaaug_{(\alpha_1)} (j,q) \Omegaaug_{(\alpha_2)} (j,q) \cdots \Omegaaug_{(\alpha_m)} (j,q).
\]
By considering the number of negative letters, we have
\begin{equation}\label{eq: aug type omega sum}
\Omegaaug_{(\alpha_i)}(j,q) = \sum_{k=0}^{\alpha_i}  \qbinom{j+k-1}{ k} \qbinom{j+\alpha_i-k -1 }{ \alpha_i-k} q^{\alpha_i - k}
\end{equation}
where the additional factor of $q^{\alpha_i-k}$ accounts for the fact that positive letters are only allowed in compartments $1,2,\ldots, j$ instead of $0,1,\ldots,j-1$.

\begin{thm}
\[
\sum_{j\geq 0} t^j \prod_{i=1}^m \left(\sum_{k=0}^{\alpha_i}  \qbinom{j+k-1}{ k} \qbinom{j+\alpha_i-k -1}{ \alpha_i-k} q^{\alpha_i - k} \right) = \frac{\augversion{B}_\alpha (t,q)}{(t;q)_{n+1}}
\]
\end{thm}

The previous theorems, while true, are not very elegant due the the fact that $\OmegaA_{(\alpha_i)} (j,q), \OmegaB_{(\alpha_i)}(j,q)$, and $\Omegaaug_{(\alpha_i)} (j,q)$ are all written as sums.
Is there a variation of these formulas in which, possibly after a substitution of variables, the sums can be evaluated?
Towards that end we choose to also keep track of the number of negative terms, weighting them by $z$.
We have
\begin{align}\label{eq: type A z weighted omega}
\OmegaA_{(\alpha_i)} (j,q,z) &= \sum_{k=0}^{\alpha_i}  \qbinom{j+k}{ k} \qbinom{j+\alpha_i-k }{ \alpha_i-k} z^k \\
\OmegaB_{(\alpha_i)} (j,q,z) &= \sum_{k=0}^{\alpha_i}  \qbinom{j+k-1}{ k} \qbinom{j+\alpha_i-k }{ \alpha_i-k} z^k \\
\Omegaaug_{(\alpha_i)} (j,q,z) &= \sum_{k=0}^{\alpha_i}  \qbinom{j+k-1}{ k} \qbinom{j+\alpha_i-k -1 }{ \alpha_i-k} q^{\alpha_i - k} z^k .
\end{align}

First consider equation \eqref{eq: type A z weighted omega}.
If we multiply both sides by $t^{\alpha_i}$ and sum over all $\alpha_i \geq 0$ then we have
\begin{equation}\label{eq: type A z weighted omega sum}
\sum_{\alpha_i \geq 0} \OmegaA_{(\alpha_i)} (j,q,z) t^{\alpha_i} = \frac{1}{(t;q)_{j+1} (tz;q)_{j+1}}.
\end{equation}
If we substitute $q \leftarrow q^2$ and $z \leftarrow q$ into equation \eqref{eq: type A z weighted omega sum} then the terms on the right side of the equation interleave.
Hence we have
\[
\sum_{\alpha_i \geq 0} \OmegaA_{(\alpha_i)} (j,q^2,q) t^{\alpha_i} = \frac{1}{(t;q)_{2j+2}}
\]
and we arrive at the following lemma for which we also include a combinatorial proof.

\begin{lem}
\[
\OmegaA_{(\alpha_i)} (j,q^2,q) =  \qbinom{2j+1+\alpha_i}{ \alpha_i}
\]
\end{lem}

\begin{proof}
Rather than choosing the number of negative letters when adding them to compartments, we instead think of splitting each compartment in two.
Placing a term in the left subcompartment will indicate that it is negative and terms in the right subcompartment are positive.
Thus positive letters are in even compartments and so are weighted by $q^{2i}$.
Negative letters are in odd compartments and so are weighted by $q^{2i+1}$.
Thus each negative term needs to have the weight corrected.
The substitution $z=q$ accounts for the new bar in the middle of the compartment.
Hence we can think of beginning with $2j+2$ compartments and choosing $\alpha_i$ of them with repetition.
\end{proof}

The substitution $q \leftarrow q^2$ and $z \leftarrow q$ defines the flag major index of Adin and Roichman \cite{AdinRoichman2001}, denoted by $\fmaj(\pi)$, where
\[
\fmaj(\pi) = 2\maj(\pi) + N(\pi).
\]
Using the previous lemma we arrive at the following theorem.

\begin{thm}\label{thm: mult fmaja}
\[
\sum_{j \geq 0} t^j \prod_{i=1}^m \qbinom{2j+1+\alpha_i }{ \alpha_i} = \frac{\dsum_{\pi \in \hypmultn} t^{\desa(\pi)}q^{\fmaj(\pi)}}{(t;q^2)_{n+1}}
\]
\end{thm}

We analogously define
\[
\famaj(\pi) = 2\amaj(\pi) + N(\pi).
\]
The following theorem allows us to count signed multiset permutations by the joint distributions $(\desb,\fmaj)$ and $(\ades,\famaj)$.

\begin{thm}\label{thm: mult fmajb famaj}
\begin{align}
\sum_{j \geq 0} t^j \prod_{i=1}^m \qbinom{2j+\alpha_i}{ \alpha_i} &= \frac{\dsum_{\pi \in \hypmultn} t^{\desb(\pi)}q^{\fmaj(\pi)}}{(t;q^2)_{n+1}} \\
\sum_{j \geq 0} t^j \prod_{i=1}^m \qbinom{2j-1+\alpha_i }{ \alpha_i} &= \frac{\dsum_{\pi \in \hypmultn} t^{\ades(\pi)}q^{\famaj(\pi)}}{(t;q^2)_{n+1}} 
\end{align}
\end{thm}

\begin{proof}
Following the case for type $A$ descents, we see that $\OmegaB_{(\alpha_i)} (j,q,z)$ is equal to the coefficient of $t^{\alpha_i}$ in
\[
\frac{1}{(tz;q)_j (t;q)_{j+1}}
\]
and $\Omegaaug_{(\alpha_i)} (j,q,z)$ is equal to the coefficient of $t^{\alpha_i}$ in
\[
\frac{1}{(tz;q)_j (t;q)_j}.
\]
Substituting $q \leftarrow q^2$ and $z \leftarrow q$, we have $1/(t;q)_{2j+1}$ and $1/(t;q)_{2j}$ respectively.
\end{proof}

If we assume that $\alpha = (1,\ldots,1)$ with $|\alpha| = n$ then Theorem \ref{thm: mult fmaja} and Theorem~\ref{thm: mult fmajb famaj} can be combined into the following theorem for $\hypn$.

\begin{thm}
\begin{align}
\label{eq: desa fmaj} \sum_{j \geq 0}   [2j+2]^n_q t^j &= \frac{\dsum_{\pi \in \hypn} t^{\desa(\pi)}q^{\fmaj(\pi)}}{(t;q^2)_{n+1}} \\ 
\label{eq: desb fmaj} \sum_{j \geq 0}   [2j+1]^n_q t^j &= \frac{\dsum_{\pi \in \hypn} t^{\desb(\pi)}q^{\fmaj(\pi)}}{(t;q^2)_{n+1}}  \\
\sum_{j \geq 0}   [2j]^n_q t^j &= \frac{\dsum_{\pi \in \hypn} t^{\ades(\pi)}q^{\famaj(\pi)}}{(t;q^2)_{n+1}} 
\end{align}
\end{thm}

Equation \eqref{eq: desb fmaj} was originally proven by Chow and Gessel in \cite[Theorem 3.7]{ChowGessel2007} and equation \eqref{eq: desa fmaj} can be derived from their Lemma 4.3.


\section{The type $B$ Eulerian descent algebra}\label{sec: type B Eulerian descent algebra}

The group algebra of the hyperoctahedral group $\hypn$ contains a ``descent'' algebra induced by the type $B$ descent set (see \cite{BergeronBergeronI1992, BergeronII1992} for a detailed study of this algebra).
Just as in the unsigned case, there also exists an algebra induced by the type~$B$ descent number, called the \emph{type $B$ Eulerian descent algebra}, whose existence was first proven by Bergeron and Bergeron \cite{BergeronBergeron1992}.
The purpose of this section is to show how our method from Chapter~\ref{chpt: Symmetric Group} generalizes to the signed case.
We define
\[
B_i = \sum_{\desb(\pi) = i} \pi
\]
for $i=0,\ldots,n$ and will show that together the $B_i$ form a basis for an algebra.
To prove that such an algebra exists, we compute $B_\pi(s,t) := \sum_{\sigma \tau = \pi} s^{\desb(\sigma)}t^{\desb(\tau)}$ and show that $B_\pi(s,t)$ is determined by $\desb(\pi)$.

The following definition of $P$-partitions of type~$B$ and the resulting theorems are taken from Chow \cite{ChowThesis2001}.
For every signed poset $P$, a \emph{$P$-partition of type $B$} is a function $f:\pm[n] \rightarrow \Z$ such that:
\begin{enumerate}
\item[(i)] $f(i) \leq f(j)$ if $i <_P j$
\item[(ii)] $f(i) < f(j)$ if $i <_P j$ and $i > j$ in $\Z$
\item[(iii)] $f(-i) = -f(i)$.
\end{enumerate}
Note that condition $(\mathrm{iii})$ implies that $f(0)=0$.
Let $\bversion{\A}(P)$ denote the set of all $P$-partitions of type~$B$.
We define the type~$B$ order polynomial, denoted by $\OmegaB_P(j)$, to be the number of type~$B$ $P$-partitions $f$ such that $f(i) \in \{-j, \ldots, 0, \ldots, j\}$ for every $i \in \pm[n]$.
We say that such $P$-partitions have parts less than or equal to $j$.

Every signed permutation $\pi \in \hypn$ can be represented by the signed poset
\[
\pi(\bar{n}) <_P \cdots <_P \pi(\bar{1}) <_P 0 <_P \pi(1) <_P \cdots <_P \pi(n).
\]
and so we denote by $\bversion{\A}(\pi)$ the set of all functions $f:\pm[n] \rightarrow \Z$ with $f(-i) = -f(i)$ for all $i \in [0,n]$ such that
\[
0 = f(\pi(0)) \leq f(\pi(1)) \leq \cdots \leq f(\pi(n))
\]
with $f(\pi(i)) < f(\pi(i+1))$ if $i \in \Desb(\pi)$.
Thus $\OmegaB_\pi(j)$ is equal to the number of integer solutions to the set of inequalities
\[
0 = i_0 \leq i_1 \leq \cdots \leq i_n \leq j  \qquad \text{with} \qquad i_k < i_{k+1} \text{ if } k \in \Desb(\pi)
\]
and so
\begin{equation}\label{eq: type b order poly formula}
\OmegaB_\pi(j) = \mchoose{j+1-\desb(\pi)}{n} = \binom{j+n-\desb(\pi)}{n}.
\end{equation}
We have the following Fundamental Theorem of $P$-partitions of Type~$B$.

\begin{thm}[FTPPB]
The set of all type $B$ $P$-partitions of a signed poset $P$ is the disjoint union of the set of all type $B$ $\pi$-partitions over all linear extensions~$\pi$ of $P$:
\[
\bversion{\A}(P) = \coprod_{\pi \in \Lin(P)} \bversion{\A}(\pi).
\]
\end{thm}

\begin{cor}\label{cor: type B order poly linear extensions}
\[
\OmegaB_P(j) = \sum_{\pi \in \Lin(P)} \OmegaB_\pi(j)
\]
\end{cor}

Next we consider the union of two disjoint posets.
We have the following theorem and corollary.

\begin{thm}
If $P_1$ and $P_2$ are two disjoint signed posets then there exists a bijection between the set of all $(P_1\sqcup P_2)$-partitions of type $B$ and the set ${\bversion{\A}(P_1) \times \bversion{\A}(P_2)}$:
\[
\bversion{\A}(P_1\sqcup P_2) \leftrightarrow \bversion{\A}(P_1) \times \bversion{\A}(P_2).
\]
\end{thm}

\begin{proof}
Let $f$ be a $(P_1 \sqcup P_2)$-partition of type $B$.
The map that sends $f$ to the ordered pair $(g,h)$ where $g = f|_{P_1}$ and $h = f|_{P_2}$ is a bijection between $\bversion{\A}(P_1\sqcup P_2)$ and $\bversion{\A}(P_1) \times \bversion{\A}(P_2)$.
\end{proof}

\begin{cor}\label{cor: type B order poly product}
\[
\OmegaB_{P_1 \sqcup P_2}(j) = \OmegaB_{P_1}(j) \OmegaB_{P_2}(j)
\]
\end{cor}

For every $I \subseteq [0,n]$ and $\pi \in \hypn$, define the \emph{type $B$ zig-zag poset} $\zigbipi$ by setting $\pi(i) <_Z \pi(i+1)$ if $i \notin I$ and $\pi(i) >_Z \pi(i+1)$ if $i \in I$ with $\pi(0) = 0$.
Just as in the unsigned case, we have the following lemma about linear extensions of type $B$ zig-zag posets.

\begin{lem}\label{lem: type B zig-zag extensions}
If $\sigma, \pi \in \hypn$ and $I \subseteq [0,n]$ then $\sigma \in \Lin(\zigbipi)$ if and only if $\Desb(\sigma^{-1}\pi) = I$.
\end{lem}

\begin{ex}
If $\pi = 2\bar{3}1$, and $I = \{0\}$ then $\zigbipi$ is the poset with ordering $0 >_Z 2<_Z \bar{3} <_Z 1$.
We have $\Lin(\zigbipi) = \{\bar{3}\bar{2}1, \bar{3}1\bar{2}, \bar{2}\bar{3}1, \bar{1}3\bar{2}, 13\bar{2}, 3\bar{2}1, 31\bar{2} \}$ and $\{\sigma^{-1}\pi \, | \, \sigma \in \Lin(\zigbipi)\} = \{\bar{3}\bar{2}\bar{1}, \bar{3}\bar{2}1, \bar{3}\bar{1}2, \bar{3}12, \bar{2}\bar{1}3, \bar{2}13, \bar{1}23 \}$.
Hence $\Desb(\sigma^{-1}\pi) = \{0\}$ for all $\sigma \in \Lin(\zigbipi)$.
\end{ex}

Next we consider a second collection of posets.
For every $I \subseteq [0,n]$ and $\pi \in \hypn$, define the \emph{type $B$ chain poset} $\chainbipi$ by setting $\pi(i) <_C \pi(i+1)$ if $i \notin I$ with $\pi(0) = 0$.

\begin{lem}\label{lem: type B chain extensions}
If $\sigma,\pi \in \hypn$ and $I \subseteq [0,n]$ then $\sigma \in \Lin(\chainbipi)$ if and only if $\Desb(\sigma^{-1}\pi) \subseteq I$.
\end{lem}

Next we define barred versions of both type $B$ zig-zag posets and type $B$ chain posets.
First, a \emph{barred type $B$ zig-zag poset} is defined to be a type $B$ zig-zag poset $\zigbipi$ with an arbitrary number of bars placed in each of the $n+1$ spaces to the right of the $0$ element such that between any two bars the elements of the poset (not necessarily their labels) are increasing.
No bars are allowed to the left of the $0$ element.
Figure \ref{fig: barred type B zig-zag poset} gives an example of a barred $\zigbipi$ poset with $I = \{0,2\}$ and $\pi = 2\bar{3}1$.

\begin{figure}[htbp]
\[\xymatrix @!R @!C @C=10pt{ & \ar@{-}[ddd] &&&&& \ar@{-}[ddd]  & \ar@{-}[ddd]  &  & \ar@{-}[ddd]  \\   0 \ar@{-}[drrr]  &  & & & & \bar{3}\ar@{-}[drrr] & & & &   \\   & & & 2 \ar@{-}[urr]  & & & & & 1 &  \\  &&&&&&&&&   }\] 
\caption{\; A barred $\zigbipi$ poset for $I=\{0,2\}$ and $\pi = 2\bar{3}1$.}
\label{fig: barred type B zig-zag poset}
\end{figure}

To create a barred type $B$ zig-zag poset, we begin with a type $B$ zig-zag poset $\zigbipi$ and must first place a bar in space $i$ for each $i \in I$.
From there we are free to place any number of bars in any of the $n+1$ spaces.
Define $\OmegaB_{\zigbipi}(j,k)$ to be the number of ordered pairs $(f,P)$ where $P$ is a barred $\zigbipi$ poset with $k$ bars and $f$ is a $\zigbipi$-partition of type~$B$ with parts less than or equal to $j$.
Recall that $\sigma \in \Lin(\zigbipi)$ if and only if $\Desb(\sigma^{-1}\pi) = I$.
If we begin with the type $B$ zig-zag poset $\zigbipi$ then there is a unique way to place the first $|I|$ bars (place one bar in space $i$ for each $i\in I$).
Next there are
\[
\mchoose{n+1}{k-|I|} = \binom{k + n - |I|}{n}
\]
ways to place the remaining $k-|I|$ bars in the $n+1$ spaces and hence there are $\binom{k+n-|I|}{n}$ barred $\zigbipi$ posets with $k$ bars.
Thus
\begin{align*}
\OmegaB_{\zigbipi}(j,k) &= \sum_{\sigma \in \Lin(\zigbipi)} \OmegaB_{\sigma}(j) \binom{k+n-\desb(\sigma^{-1}\pi)}{n} \\
&= \sum_{\sigma \in \Lin(\zigbipi)} \OmegaB_{\sigma}(j) \OmegaB_{\sigma^{-1}\pi}(k).
\end{align*}
If we set $\tau = \sigma^{-1}\pi$ and sum over all $I \subseteq [0,n]$ we have
\begin{equation}\label{eq: barred type b zig-zag order polys}
\sum_{I \subseteq [0,n]} \OmegaB_{\zigbipi}(j,k) = \sum_{\sigma \tau = \pi} \OmegaB_{\sigma}(j) \OmegaB_{\tau}(k).
\end{equation}

Next we define a \emph{barred type $B$ chain poset} to be a type $B$ chain poset $\chainbipi$ with at least one bar in space $i$ for each $i \in I$ and with an arbitrary number of bars placed on the right end.
No bars are allowed in space $i$ for $i\in [0,n-1]\setminus I$ and no bars are allowed to the left of the $0$ element.
Thus we place at least one bar between each chain of the type $B$ chain poset and allow for bars on the right end.
Figure \ref{fig: barred type B chain poset} gives an example of a barred $\chainbipi$ poset with $I = \{0,2\}$ and $\pi = 2\bar{3}1$.

\begin{figure}[htbp]
\[\xymatrix @!R @!C @C=10pt{ & \ar@{-}[ddd] &&&& \ar@{-}[ddd]  & \ar@{-}[ddd]  &  & \ar@{-}[ddd]  \\     &  & & & \bar{3} & & & &   \\  0   & & 2 \ar@{-}[urr]  & & & & & 1 &  \\  &&&&&&&&   }\] 
\caption{\; A barred $\chainbipi$ poset for $I = \{0,2\}$ and $\pi = 2\bar{3}1$.}
\label{fig: barred type B chain poset}
\end{figure}

To create a barred type $B$ chain poset, we begin with a type $B$ chain poset $\chainbipi$ and must first place a bar in space $i$ for each $i \in I$.
From there we are free to place any number of bars in space $i$ for any $i \in I$ and can also place any number of bars on the right end.
This creates a collection of bars with each compartment containing at most one nonempty chain.
Define $\OmegaB_{\chainbipi} (j,k)$ to be the number of ordered pairs $(f,P)$ where $P$ is a barred $\chainbipi$ poset with $k$ bars and $f$ is a $\chainbipi$-partition of type $B$ with parts less than or equal to $j$.
Recall that $\sigma \in \Lin(\chainbipi)$ if and only if $\Desb(\sigma^{-1}\pi) \subseteq I$.
If we begin with the type $B$ chain poset $\chainbipi$ then there is a unique way to place the first $|I|$ bars (place one bar in space $i$ for each $i\in I$).
Next there are
\[
\mchoose{|I|+1}{k-|I|} = \binom{k}{k-|I|}
\]
ways to place the remaining $k-|I|$ bars in the $|I|+1$ allowable spaces and hence there are $\binom{k}{k-|I|}$ barred $\chainbipi$ posets with $k$ bars.
Thus
\[
\OmegaB_{\chainbipi}(j,k) = \sum_{\sigma \in \Lin(\chainbipi)} \OmegaB_{\sigma}(j) \binom{k}{k-|I|}.
\]

The following lemma shows that if $\sigma \in \hypn$ then the number of barred type $B$ zig-zag posets with $k$ bars such that $\sigma$ is a linear extension of the underlying type $B$ zig-zag poset is the same as the number of barred type $B$ chain posets with $k$ bars such that $\sigma$ is a linear extension of the underlying type $B$ chain poset.

\begin{lem}\label{lem: barred type B zig chain equiv}
For every $\pi \in \hypn$ and $j,k\geq0$,
\[
\sum_{I \subseteq [0,n]}  \OmegaB_{\zigbipi}(j,k)= \sum_{I \subseteq [0,n]} \OmegaB_{\chainbipi}(j,k).
\]
\end{lem}

\begin{proof}
We must find a bijection between the two types of barred signed posets that respects the number of bars.
Recall that $\sigma \in \Lin(\zigbipi)$ for a unique $I \subseteq [0,n]$ and that $\sigma \in \Lin(C^B(J,\pi))$ if and only if $J \supseteq I$.
The desired bijection is simply the map that sends each barred type $B$ zig-zag poset with underlying poset $\zigbipi$ to the barred type $B$ chain poset obtained from $\zigbipi$ by removing the relation between $\pi(i)$ and $\pi(i+1)$ for every space $i$ containing at least one bar.
\end{proof}

The bijection in the previous lemma maps Figure \ref{fig: barred type B zig-zag poset} to Figure \ref{fig: barred type B chain poset}.
Now that all the pieces are in place, we are ready to prove the existence of the type $B$ Eulerian descent algebra by computing $B_\pi (s,t)$.

\begin{thm}\label{thm: type B algebra}
For every $\pi \in \hypn$,
\begin{equation}\label{eq: type B algebra}
\sum_{j,k \geq 0} \binom{2jk+j+k +n-\desb(\pi)}{n} s^j t^k = \frac{B_\pi (s,t)}{(1-s)^{n+1}(1-t)^{n+1}}.
\end{equation}
\end{thm}

\begin{proof}
First we see that
\begin{align*}
\frac{B_\pi(s,t)}{(1-s)^{n+1}(1-t)^{n+1}} &=  \sum_{j,k \geq 0} \sum_{\sigma \tau = \pi} \binom{j+n}{n}\binom{k+n}{n}s^{j+\desb(\sigma)} t^{k+\desb(\tau)}\\
&= \sum_{j,k \geq 0} \sum_{\sigma \tau = \pi} \binom{j+n-\desb(\sigma)}{n}\binom{k+n-\desb(\tau)}{n}s^j t^k \\ 
&= \sum_{j,k \geq 0} \sum_{I \subseteq [0,n]} \OmegaB_{\zigbipi}(j,k) s^j t^k
\end{align*}
where the last equality follows from equations \eqref{eq: type b order poly formula} and \eqref{eq: barred type b zig-zag order polys}.
The bijection in Lemma~\ref{lem: barred type B zig chain equiv} allows us to shift our focus from barred type $B$ zig-zag posets to barred type $B$ chain posets and shows that
\[
\frac{B_\pi(s,t)}{(1-s)^{n+1}(1-t)^{n+1}} = \sum_{j,k \geq 0} \sum_{I \subseteq [0,n]} \OmegaB_{\chainbipi}(j,k) s^j t^k.
\]
The only remaining step is to prove that
\[
\sum_{I \subseteq [0,n]} \OmegaB_{\chainbipi}(j,k) =  \binom{2jk+j+k+n-\desb(\pi)}{n}.
\]

First we note that $\sum_{I \subseteq [0,n]} \OmegaB_{\chainbipi}(j,k)$ counts ordered pairs $(f,P)$ where $P$ is a barred $\chainbipi$ poset for some $I \subseteq [0,n]$ and $f$ is a $\chainbipi$-partition with parts less than or equal to $j$.
If we use the bars to define compartments labeled $0,\ldots,k$ from left to right then we have a bijection between ordered pairs $(f,P)$ and $(k+1)$-tuples of the form $(f_0,f_1,\ldots,f_k)$ where $f_i$ is simply $f$ restricted to the chain in compartment~$i$.
Given such a $(k+1)$-tuple, we define a type~$B$ $\pi$-partition $g$ with parts less than or equal to $2jk+j+k$ as follows.
On the interval $[0,j]$, we set $g= f_0$.
Then for $i = 1,\ldots,k$ we set $g = f_i + i(2j+1)$ on the interval $[-j+i(2j+1) , j+i(2j+1)]$.
This shows that if we sum over all $I \subseteq [0,n]$ then $\sum_{I \subseteq [0,n]} \OmegaB_{\chainbipi}(j,k)$ is equal to the number of $\pi$-partitions of type $B$ with parts less than or equal to $2jk+j+k$.
Hence we conclude that
\[
\sum_{I \subseteq [0,n]} \OmegaB_{\chainbipi}(j,k)  = \binom{2jk+j+k +n-\desb(\pi)}{n} .  \qedhere
\]
\end{proof}

Figure \ref{fig: every barred type B chain poset} depicts every barred $C^B(I,\bar{2}1)$ poset with $2$ bars and $I \subseteq [0,1]$.
Thus we can identify each barred $\chainbipi$ poset with a barred signed permutation with underlying signed permutation $\pi$ such that $0$ is in the leftmost compartment.
This provides a visual representation of the proof of the previous theorem.

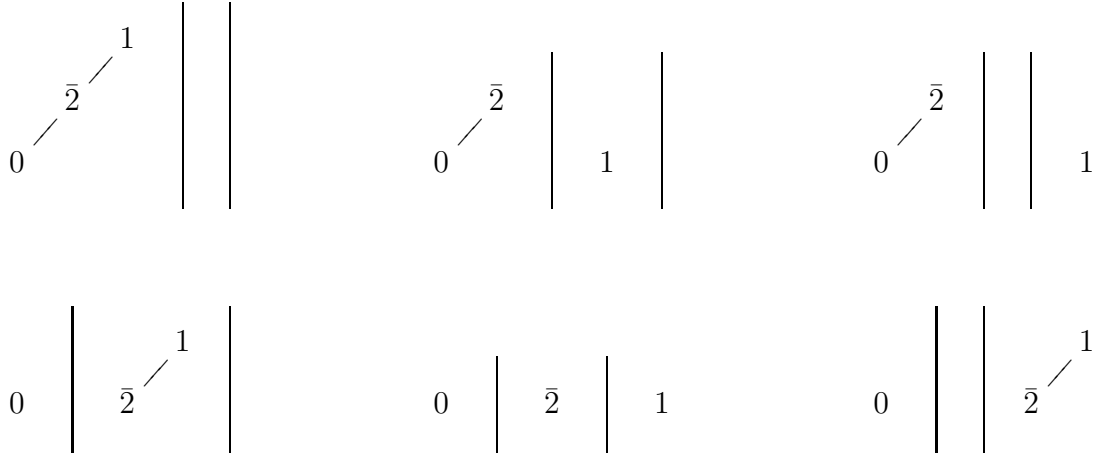
\begin{figure}[htbp]
\[\xymatrix  @R=8pt @C=8pt{ &&& \ar@{-}[dddd] & \ar@{-}[dddd] &&&&&&&&&& &  & \\   && 1 &&&&&& \ar@{-}[ddd]  && \ar@{-}[ddd]  &&&& \ar@{-}[ddd] & \ar@{-}[ddd] & \\  & \bar{2}\ar@{-}[ur] &&&&&& \bar{2} &&&&&& \bar{2} &&&  \\  0\ar@{-}[ur] &&&&& \qquad\qquad & 0\ar@{-}[ur] &&& 1 && \qquad\qquad & 0\ar@{-}[ur] &&&& 1 \\  &&&&&&&&&&&&&&&& \\ &&&&&&&&&&&&&&&& \\ & \ar@{-}[ddd] &&& \ar@{-}[ddd] &&&&&&&&& \ar@{-}[ddd] & \ar@{-}[ddd] && \\ &&& 1 &&&& \ar@{-}[dd] && \ar@{-}[dd] &&&&&&& 1 \\ 0 && \bar{2}\ar@{-}[ur] &&&& 0 && \bar{2} && 1 && 0 &&& \bar{2}\ar@{-}[ur] & \\ &&&&&&&&&&&&&&&& \\     }\] 
\caption{\; Barred $C^B(I,\bar{2}1)$ posets with $2$ bars and $I \subseteq [0,1]$.}
\label{fig: every barred type B chain poset}
\end{figure}

Define the type $B$ structure polynomial $\phi_B(x)$ in the group algebra of $\hypn$ by
\[
\phi_B(x) = \sum_{\pi \in \hypn} \binom{x+n-\desb(\pi)}{n} \pi.
\]
If we expand both sides of equation \eqref{eq: type B algebra} and compare the coefficients of $s^jt^k$, we have
\[
\binom{2jk+j+k +n-\desb(\pi)}{n} = \sum_{\sigma\tau = \pi}  \binom{j+n -\desb(\sigma)}{n} \binom{k+n - \desb(\tau)}{n}.
\]
This implies that $\phi_B(j)\phi_B(k) = \phi_B(2jk + j + k)$ for all $j,k \geq 0$ and so we have the following theorem, originally proven by Chow \cite{ChowThesis2001}.

\begin{thm}[Chow]\label{thm: desb order poly}
As polynomials in $x$ and $y$ with coefficients in the group algebra of $\hypn$,
\[
\phi_B(x)\phi_B(y) = \phi_B(2xy + x + y).
\]
\end{thm}

If we substitute $x \leftarrow (x-1)/2$ and $y \leftarrow (y-1)/2$ in the previous theorem, we see that $\phi_B((x-1)/2)\phi_B((y-1)/2) = \phi_B((xy-1)/2)$.
As before, if we expand
\[
\phi_B((x-1)/2) = \sum_{\pi \in \hypn} \binom{\frac{x-1}{2}+n-\desb(\pi)}{n} \pi = \sum_{i=0}^n b_i x^i
\]
then the $b_i$ form a set of orthogonal idempotents which span the type $B$ Eulerian descent algebra.
These orthogonal idempotents were originally described by Bergeron and Bergeron in \cite{BergeronBergeron1992}.


\section{The type $A$ Eulerian descent algebra}\label{sec: type A Eulerian descent algebra}

The group algebra of the hyperoctahedral group $\hypn$ contains a subalgebra induced by the type $A$ descent number called the \emph{type $A$ Eulerian descent algebra}.
We define
\[
A_i = \sum_{\desa(\pi) = i} \pi
\]
for $i=0,\ldots,n-1$ and will show that together the $A_i$ form a basis for an algebra.
To prove that such an algebra exists, we compute $A_\pi(s,t) := \sum_{\sigma \tau = \pi} s^{\desa(\sigma)}t^{\desa(\tau)}$ and show that $A_\pi(s,t)$ is determined by $\desa(\pi)$.

Our first step is to define a type $A$ version of Chow's $P$-partitions of type $B$.
We do this by ignoring the $0$ element in signed posets.
Let $\aversion{[n]}$ denote the set $\{-n,\ldots,-1,1,\ldots,n\}$ and let $\aversion{\Z}$ denote a modified version of the integers where we replace $0$ with both $\zeroplus$ and $\zerominus$ satisfying $-\zeroplus = \zerominus$.
For every signed poset $P$, a $P$-partition of type $A$ is a function $f:\aversion{[n]} \rightarrow \aversion{\Z}$ such that:
\begin{enumerate}
\item[(i)] $f(i) \leq f(j)$ if $i <_P j$
\item[(ii)] $f(i) < f(j)$ if $i <_P j$ and $i > j$ in $\Z$
\item[(iii)] $f(-i) = -f(i)$.
\end{enumerate}
Let $\aversion{\A}(P)$ denote the set of all $P$-partitions of type $A$.
The type $A$ order polynomial, denoted by $\OmegaA_P(j)$, is the number of type $A$ $P$-partitions with parts in the set $\{-j,\ldots,j\}$.

The reason we use $\aversion{\Z}$ instead of $\Z$ is illustrated by the following example in $\mathfrak{B}_1$.
Consider the two signed posets $P_1$ and $P_2$ where $P_1$ is given by $0 <_P 1$ and $P_2$ is given by $0 <_P \bar{1}$.
If we had chosen $\Z$ as the codomain in our definition then the set $\A(P_1)$ would include the map $f$ defined by $f(1) = 0$.
However, the set $\A(P_2)$ would not contain the map $f$ defined by $f(\bar{1})=0$.
This is problematic because we would like our definition of type~$A$ $P$-partitions to treat both posets as being the same.
The use of $\aversion{\Z}$ corrects for this.

Every signed permutation $\pi \in \hypn$ can be represented by the signed poset
\[
\pi(\bar{n}) <_P \cdots <_P \pi(\bar{1}) <_P 0 <_P \pi(1) <_P \cdots <_P \pi(n)
\]
and so we denote by $\aversion{\A}(\pi)$ the set of all functions $f:\aversion{[n]} \rightarrow \aversion{\Z}$ with $f(-i) = -f(i)$ for all $i \in [n]$ such that
\[
\zeroplus \leq f(\pi(1)) \leq \cdots \leq f(\pi(n))
\]
with $f(\pi(i)) < f(\pi(i+1))$ if $i \in \Desa(\pi)$.
Thus
\begin{equation}\label{eq: type A order poly for pi}
\OmegaA_\pi(j) = \mchoose{j+1-\desa(\pi)}{n} = \binom{j+n-\desa(\pi)}{n}.
\end{equation}
Just as in the type $B$ case, we have a Fundamental Theorem of $P$-partitions of Type~$A$.

\begin{thm}[FTPPA]
The set of all type $A$ $P$-partitions of a signed poset~$P$ is the disjoint union of the set of all type $A$ $\pi$-partitions over all linear extensions $\pi$ of $P$:
\[
\aversion{\A}(P) = \coprod_{\pi \in \Lin(P)} \aversion{\A}(\pi).
\]
\end{thm}

\begin{cor}
\[
\OmegaA_P(j) = \sum_{\pi \in \Lin(P)} \OmegaA_\pi(j)
\]
\end{cor}

Again we consider the union of two disjoint posets and have the following theorem and corollary.

\begin{thm}
\looseness+1 If $P_1$ and $P_2$ are two disjoint signed posets then there exists a bijection between the set of all $(P_1\sqcup P_2)$-partitions of type $A$ and the set $\aversion{\A}(P_1) \times \aversion{\A}(P_2)$.
\end{thm}

\begin{proof}
Let $f$ be a $(P_1 \sqcup P_2)$-partition of type $A$.
The map that sends $f$ to the ordered pair $(g,h)$ where $g = f|_{P_1}$ and $h = f|_{P_2}$ is a bijection between $\aversion{\A}(P_1\sqcup P_2)$ and $\aversion{\A}(P_1) \times \aversion{\A}(P_2)$.
\end{proof}

\begin{cor}
\[
\OmegaA_{P_1 \sqcup P_2}(j) = \OmegaA_{P_1}(j) \OmegaA_{P_2}(j)
\]
\end{cor}

For every $I \subseteq [n-1]$ and $\pi \in \hypn$, define the \emph{type $A$ zig-zag poset} $\zigaipi$ by setting $\pi(i) <_Z \pi(i+1)$ if $i \notin I$ and $\pi(i) >_Z \pi(i+1)$ if $i \in I$.
We also define for every $I \subseteq [n-1]$ and $\pi \in \hypn$ the \emph{type $A$ chain poset} $\chainaipi$ by setting $\pi(i) <_C \pi(i+1)$ if $i \notin I$.
Note that for both types of signed posets there is no relation between $0$ and $\pi(i)$ for any $i \in [n]$.
We have the following lemmas about linear extensions of these two classes of signed posets.

\begin{lem}\label{lem: type A zig-zag extensions}
If $\sigma, \pi \in \hypn$ and $I \subseteq [n-1]$ then $\sigma \in \Lin(\zigaipi)$ if and only if $\Desa(\sigma^{-1}\pi) = I$.
\end{lem}

\begin{lem}\label{lem: type A chain extensions}
If $\sigma,\pi \in \hypn$ and $I \subseteq [n-1]$ then $\sigma \in \Lin(\chainaipi)$ if and only if $\Desa(\sigma^{-1}\pi) \subseteq I$.
\end{lem}

Next we define barred versions of both type $A$ zig-zag posets and type $A$ chain posets.
First, a \emph{barred type $A$ zig-zag poset} is defined to be a type $A$ zig-zag poset $\zigaipi$ with an arbitrary number of bars placed in each of the $n+1$ spaces such that between any two bars the elements of the poset (not necessarily their labels) are increasing.
Figure \ref{fig: barred type A zig-zag poset} gives an example of a barred $\zigaipi$ poset with $I = \{2\}$ and $\pi = 2\bar{3}1$.

\begin{figure}[htbp]
\[\xymatrix @R=30pt @C=13pt{  \ar@{-}[ddd] &&&& \ar@{-}[ddd]  & \ar@{-}[ddd]  &  & \ar@{-}[ddd]  \\    & & & \bar{3}\ar@{-}[drrr] & & & &   \\   & 2 \ar@{-}[urr]  & & & & & 1 &  \\  &&&&&&&   }\] 
\caption{\; A barred $\zigaipi$ poset for $I=\{2\}$ and $\pi = 2\bar{3}1$.}
\label{fig: barred type A zig-zag poset}
\end{figure}

To create a barred type $A$ zig-zag poset, we begin with a type $A$ zig-zag poset $\zigaipi$ and must first place a bar in space $i$ for each $i \in I$.
From there we are free to place any number of bars in any of the $n+1$ spaces.
Define $\OmegaA_{\zigaipi}(j,k)$ to be the number of ordered pairs $(f,P)$ where $P$ is a barred $\zigaipi$ poset with $k$ bars and $f$ is a $\zigaipi$-partition of type~$A$ with parts less than or equal to $j$.
Recall that $\sigma \in \Lin(\zigaipi)$ if and only if $\Desa(\sigma^{-1}\pi) = I$.
If we begin with the type $A$ zig-zag poset $\zigaipi$ then there is a unique way to place the first $|I|$ bars (place one bar in space $i$ for each $i\in I$).
Next there are
\[
\mchoose{n+1}{k-|I|} = \binom{k + n - |I|}{n}
\]
ways to place the remaining $k-|I|$ bars in the $n+1$ allowable spaces and hence there are $\binom{k+n-|I|}{n}$ barred $\zigaipi$ posets with $k$ bars.
Thus
\begin{align*}
\OmegaA_{\zigaipi}(j,k) &= \sum_{\sigma \in \Lin(\zigaipi)} \OmegaA_{\sigma}(j) \binom{k+n-\desa(\sigma^{-1}\pi)}{n} \\
&= \sum_{\sigma \in \Lin(\zigaipi)} \OmegaA_{\sigma}(j) \OmegaA_{\sigma^{-1}\pi}(k).
\end{align*}
If we set $\tau = \sigma^{-1}\pi$ and sum over all $I \subseteq [n-1]$ then we have
\begin{equation}\label{eq: barred type A zig-zag order polys}
\sum_{I \subseteq [n-1]} \OmegaA_{\zigaipi}(j,k) = \sum_{\sigma \tau = \pi} \OmegaA_{\sigma}(j) \OmegaA_{\tau}(k).
\end{equation}

Next we define a \emph{barred type $A$ chain poset} to be a type $A$ chain poset $\chainaipi$ with at least one bar in space $i$ for each $i \in I$ and with an arbitrary number of bars placed on either end.
No bars are allowed in space $i$ for $i\in [n-1]\setminus I$.
Thus we place at least one bar between each chain of $\chainaipi$ and allow for bars on either end.
Figure \ref{fig: barred type A chain poset} gives an example of a barred $\chainaipi$ poset with $I = \{2\}$ and $\pi = 2\bar{3}1$.

\begin{figure}[htbp]
\[\xymatrix@C=1pc{ \ar@{-}[ddd] &&&& \ar@{-}[ddd]  & \ar@{-}[ddd]  &  & \ar@{-}[ddd]  \\     & & & \bar{3} & & & &   \\   & 2 \ar@{-}[urr]  & & & & & 1 &  \\ &&&&&&&   }\] 
\caption{\; A barred $\chainaipi$ poset for $I = \{2\}$ and $\pi = 2\bar{3}1$.}
\label{fig: barred type A chain poset}
\end{figure}

To create a barred type $A$ chain poset, we begin with a type $A$ chain poset $\chainaipi$ and must first place a bar in space $i$ for each $i \in I$.
From there we are free to place any number of bars in space $i$ for any $i \in I$ and place any number of bars on either end.
This creates a collection of bars with each compartment containing at most one nonempty chain labeled by a subword of $\pi$.
Define $\OmegaA_{\chainaipi} (j,k)$ to be the number of ordered pairs $(f,P)$ where $P$ is a barred $\chainaipi$ poset with $k$ bars and $f$ is a $\chainaipi$-partition of type $A$ with parts less than or equal to $j$.
Recall that $\sigma \in \Lin(\chainaipi)$ if and only if $\Desa(\sigma^{-1}\pi) \subseteq I$.
If we begin with the type $A$ chain poset $\chainaipi$ then there is a unique way to place the first $|I|$ bars (place one bar in space $i$ for each $i\in I$).
Next there are
\[
\mchoose{|I|+2}{k-|I|} = \binom{k+1}{k-|I|}
\]
ways to place the remaining $k-|I|$ bars in the $|I|+2$ allowable spaces and hence there are $\binom{k+1}{k-|I|}$ barred $\chainaipi$ posets with $k$ bars.
Thus
\[
\OmegaA_{\chainaipi}(j,k) = \sum_{\sigma \in \Lin(\chainaipi)} \OmegaA_{\sigma}(j) \binom{k+1}{k-|I|}.
\]

The following lemma shows that if $\sigma \in \hypn$ then the number of barred type $A$ zig-zag posets with $k$ bars such that $\sigma$ is a linear extension of the underlying type $A$ zig-zag poset is the same as the number of barred type $A$ chain posets with $k$ bars such that $\sigma$ is a linear extension of the underlying type $A$ chain poset.

\begin{lem}\label{lem: barred type A zig chain equiv}
Let $\pi \in \hypn$ and $j,k\geq0$.
Then
\[
\sum_{I \subseteq [n-1]}  \OmegaA_{\zigaipi}(j,k)= \sum_{I \subseteq [n-1]} \OmegaA_{\chainaipi}(j,k).
\]
\end{lem}

\begin{proof}
Recall that $\sigma \in \Lin(\zigaipi)$ for a unique $I \subseteq [n-1]$ and that \linebreak $\sigma \in \Lin(C^A(J,\pi))$ if and only if $J \supseteq I$.
The desired bijection is simply the map that sends each barred type $A$ zig-zag poset with underlying poset $\zigaipi$ to the barred type $A$ chain poset obtained from $\zigaipi$ by removing the relation between $\pi(i)$ and $\pi(i+1)$ for every space $i$ containing at least one bar.
\end{proof}

The bijection in the previous lemma maps Figure \ref{fig: barred type A zig-zag poset} to Figure \ref{fig: barred type A chain poset}.
Now that all the pieces are in place, we are ready to prove the existence of the type $A$ Eulerian descent algebra by computing $A_\pi (s,t)$.

\begin{thm}\label{thm: type A algebra}
For every $\pi \in \hypn$,
\begin{equation}\label{eq: type A algebra}
\sum_{j,k \geq 0} \binom{2(j+1)(k+1)+n-1-\desa(\pi)}{n} s^j t^k = \frac{A_\pi (s,t)}{(1-s)^{n+1}(1-t)^{n+1}}.
\end{equation}
\end{thm}

\begin{proof}
First we see that
\begin{align*}
\frac{A_\pi(s,t)}{(1-s)^{n+1}(1-t)^{n+1}} &=  \sum_{j,k \geq 0} \sum_{\sigma \tau = \pi} \binom{j+n}{n}\binom{k+n}{n}s^{j+\desa(\sigma)} t^{k+\desa(\tau)}\\
&= \sum_{j,k \geq 0} \sum_{\sigma \tau = \pi} \binom{j+n-\desa(\sigma)}{n}\binom{k+n-\desa(\tau)}{n}s^j t^k \\ 
&= \sum_{j,k \geq 0} \sum_{I \subseteq [n-1]} \OmegaA_{\zigaipi}(j,k) s^j t^k
\end{align*}
where the last equality follows from equations \eqref{eq: type A order poly for pi} and \eqref{eq: barred type A zig-zag order polys}.
The bijection in Lemma~\ref{lem: barred type A zig chain equiv} allows us to shift our focus from barred type $A$ zig-zag posets to barred type $A$ chain posets and shows that
\[
\frac{A_\pi(s,t)}{(1-s)^{n+1}(1-t)^{n+1}} = \sum_{j,k \geq 0} \sum_{I \subseteq [n-1]} \OmegaA_{\chainaipi}(j,k) s^j t^k.
\]
The only remaining step is to prove that
\[
\sum_{I \subseteq [n-1]} \OmegaA_{\chainaipi}(j,k) =  \binom{2(j+1)(k+1)+n-1-\desa(\pi)}{n}.
\]

First we note that $\sum_{I \subseteq [n-1]} \OmegaA_{\chainaipi}(j,k)$ counts ordered pairs $(f,P)$ where $P$ is a barred $\chainaipi$ poset with $k$ bars for some $I \subseteq [n-1]$ and $f$ is a $\chainaipi$-partition of type~$A$ with parts less than or equal to $j$.
Fix a barred $\chainaipi$ poset with $k$ bars and use the bars to define compartments labeled $0,\ldots,k$ from left to right.
Then define $\pi_i$ to be the (possibly empty) subword of $\pi$ in compartment $i$ and denote the length of $\pi_i$ by $L_i$.
Then
\[
\OmegaA_{\chainaipi}(j) = \prod_{i=0}^{k} \OmegaA_{P(\pi_i)}(j).
\]
For $i=0,\ldots,k$, we let $\OmegaA_{P(\pi_i)}(j)$ count solutions to the inequalities
\[
i(2j+2) \leq s_{i_1} \leq  \cdots \leq s_{i_{L_i}} \leq i(2j+2) + 2j+1
\]
with $s_{i_l} < s_{i_{l+1}}$ if $l \in \Desa(\pi_i)$.
By concatenating these inequalities, we see that if we sum over all $I \subseteq [n-1]$ and all barred $\chainaipi$ posets with $k$ bars then $\sum_{I \subseteq [n-1]} \OmegaA_{\chainaipi}(j,k)$ is equal to the number of solutions to the inequalities
\[
0 \leq s_1 \leq \cdots \leq s_n \leq 2jk+2j+2k+1
\]
with $s_i < s_{i+1}$ if $i \in \Desa(\pi)$.
Hence we conclude that
\[
\sum_{I \subseteq [n-1]} \OmegaA_{\chainaipi}(j,k)  = \binom{2(j+1)(k+1)+n-1-\desa(\pi)}{n} . \qedhere
\]
\end{proof}

Define the type $A$ structure polynomial $\phi_A(x)$ in the group algebra of $\hypn$ by
\[
\phi_A(x) = \sum_{\pi \in \hypn} \binom{x+n-1-\desa(\pi)}{n} \pi.
\]
If we expand both sides of equation \eqref{eq: type A algebra} and compare the coefficients of $s^jt^k$ we have
\[
\binom{2(j+1)(k+1)+n-1-\desa(\pi)}{n} = \sum_{\sigma\tau = \pi}  \binom{j+n -\desa(\sigma)}{n} \binom{k+n - \desa(\tau)}{n} .
\]
This implies that $\phi_A(j+1)\phi_A(k+1) = \phi_A(2(j+1)(k+1))$ for all $j,k \geq 0$ and so we have the following theorem.

\begin{thm}
As polynomials in $x$ and $y$ with coefficients in the group algebra of $\hypn$,
\[
\phi_A(x)\phi_A(y) = \phi_A(2xy).
\]
\end{thm}

If we substitute $x \leftarrow x/2$ and $y \leftarrow y/2$ in the previous theorem then we see that $\phi_A(x/2)\phi_A(y/2) = \phi_A(xy/2)$.
Thus if we expand
\[
\phi_A(x/2) = \sum_{\pi \in \hypn} \binom{\frac{x}{2}+n-1-\desa(\pi)}{n} \pi = \sum_{i=1}^n a_i x^i
\]
the $a_i$ form a set of orthogonal idempotents which span the type $A$ Eulerian descent algebra.


\section{The augmented Eulerian descent algebra}\label{sec: augmented descent algebra}

The group algebra of the hyperoctahedral group $\hypn$ contains a subalgebra induced by the augmented descent number called the \emph{augmented Eulerian descent algebra}.
Cellini \cite{CelliniII1995} first proved the existence of this algebra which is a natural analogue of the cyclic Eulerian descent algebra from Chapter \ref{chpt: Symmetric Group}.
We define
\[
B^{(a)}_i = \dsum_{\ades(\pi) = i} \pi
\]
for $i=1,\ldots,n$ and will show that together the $B^{(a)}_i$ form a basis for an algebra.
To prove that such an algebra exists, we compute $B^{(a)}_\pi(s,t) := \sum_{\sigma \tau = \pi} s^{\ades(\sigma)}t^{\ades(\tau)}$ and show that $B^{(a)}_\pi(s,t)$ is determined by $\ades(\pi)$.

The following definition of augmented $P$-partitions and the resulting theorems are due to Petersen \cite{Petersen2005}.
Let $X=\{x_0,x_1,\ldots, x_{\infty} \}$ be a countable totally ordered set with a maximal element and total order
\[
x_0 < x_1 < \cdots < x_{\infty} .
\]
Define $\pm X$ to be the set $\{ -x_{\infty} \ldots, -x_1, x_0, x_1, \ldots, x_{\infty} \}$ with total order
\[
-x_{\infty} < \cdots < -x_1 < x_0 < x_1 < \cdots < x_{\infty}.
\]
For any signed poset $P$, an augmented $P$-partition is a function $f:\pm[n] \rightarrow \pm X$ such that:
\begin{enumerate}
\item[(i)] $f(i) \leq f(j)$ if $i <_P j$
\item[(ii)] $f(i) < f(j)$ if $i <_P j$ and $i > j$ in $\Z$
\item[(iii)] $f(-i) = -f(i)$
\item[(iv)] if $0 < i$ in $\Z$, then $f(i) < x_{\infty}$.
\end{enumerate}
Let $\augversion{\A}(P)$ denote the set of all augmented $P$-partitions.
For the rest of this section we set $X=\{0,\ldots, j\}$ and define the augmented order polynomial, denoted by $\Omegaaug_P(j)$, to be the number of augmented $P$-partitions.

Every signed permutation $\pi \in \hypn$ can be represented by the signed poset
\[
\pi(\bar{n}) <_P \cdots <_P \pi(\bar{1}) <_P 0 <_P \pi(1) <_P \cdots <_P \pi(n).
\]
We denote by $\augversion{\A}(\pi)$ the set of all functions $f:\pm [n] \rightarrow \pm X$ with $f(-i) = -f(i)$ for all $i \in [n]$ such that
\[
x_0 = f(\pi(0)) \leq f(\pi(1)) \leq \cdots \leq f(\pi(n)) \leq f(\pi(n+1)) = x_{\infty}
\]
with $f(\pi(i)) < f(\pi(i+1))$ if $i \in \aDes(\pi)$.
Thus $\Omegaaug_\pi(j)$ is equal to the number of integer solutions to the set of inequalities
\[
0 = i_0 \leq i_1 \leq \cdots \leq i_n \leq i_{n+1} = j  \qquad \text{with} \qquad i_k < i_{k+1} \text{ if } k \in \aDes(\pi).
\]
Thus
\begin{equation}\label{eq: aug order poly for pi}
\Omegaaug_\pi(j) = \mchoose{j+1-\ades(\pi)}{n} = \binom{j+n-\ades(\pi)}{n}.
\end{equation}
We have a Fundamental Theorem of Augmented $P$-partitions.

\begin{thm}
The set of all augmented $P$-partitions for a signed poset $P$ is the disjoint union of the set of all augmented $\pi$-partitions over all linear extensions $\pi$ of $P$:
\[
\augversion{\A}(P) = \coprod_{\pi \in \Lin(P)} \augversion{\A}(\pi).
\]
\end{thm}

\begin{cor}
\[
\Omegaaug_P(j) = \sum_{\pi \in \Lin(P)} \Omegaaug_\pi(j)
\]
\end{cor}

Again we consider the union of two disjoint posets and have the following theorem and corollary.

\begin{thm}
If $P_1$ and $P_2$ are two disjoint signed posets then there exists a bijection between the set of all augmented $(P_1\sqcup P_2)$-partitions and the set ${\augversion{\A}(P_1) \times \augversion{\A}(P_2)}$:
\[
\augversion{\A}(P_1\sqcup P_2) \leftrightarrow \augversion{\A}(P_1) \times \augversion{\A}(P_2).
\]
\end{thm}

\begin{proof}
Let $f$ be an augmented $(P_1 \sqcup P_2)$-partition.
The map that sends $f$ to the ordered pair $(g,h)$ where $g = f|_{P_1}$ and $h = f|_{P_2}$ is a bijection between $\augversion{\A}(P_1\sqcup P_2)$ and $\augversion{\A}(P_1) \times \augversion{\A}(P_2)$.
\end{proof}

\begin{cor}
\[
\Omegaaug_{P_1 \sqcup P_2}(j) = \Omegaaug_{P_1}(j) \Omegaaug_{P_2}(j)
\]
\end{cor}

Let $\augzigsets$ denote the set of every nonempty proper subset of $[0,n]$.
For every $I \in \augzigsets$ and $\pi \in \hypn$, define the \emph{augmented zig-zag poset} $\augzigipi$ by setting $\pi(i) <_Z \pi(i+1)$ if $i \notin I$ and $\pi(i) >_Z \pi(i+1)$ if $i \in I$ with $\pi(0) = 0 = \pi(n+1)$.
Note that both $I=\varnothing$ and $I = [0,n]$ would yield posets with $0 <_Z 0$ and so are excluded.
We have the following lemma about linear extensions of augmented zig-zag posets.

\begin{lem}\label{lem: aug zig-zag extensions}
If $\sigma, \pi \in \hypn$ and $I \in \augzigsets$ then $\sigma \in \Lin(\augzigipi)$ if and only if $\aDes(\sigma^{-1}\pi) = I$.
\end{lem}

\begin{ex}
If $\pi = \bar{2}1$ and $I = \{0,2\}$ then $\augzigipi$ is the poset with ordering $0 >_Z \bar{2} <_Z 1 >_Z 0$.
We have $\Lin(\augzigipi) = \{12, 21 \}$ and $\{\sigma^{-1}\pi \, : \, \sigma \in \Lin(\augzigipi)\} \linebreak = \{ \bar{2}1, \bar{1}2 \}$.
Hence $\aDes(\sigma^{-1}\pi) = \{0,2\}$ for all $\sigma \in \Lin(\augzigipi)$.
\end{ex}

Next we consider a second collection of posets.
Let $\augchainsets$ denote the set of every nonempty subset of $[0,n]$.
For every $I \in \augchainsets$ and $\pi \in \hypn$, define the \emph{augmented chain poset} $\augchainipi$ by setting $\pi(i) <_C \pi(i+1)$ if $i \notin I$ with $\pi(0) = 0 = \pi(n+1)$.
We now allow for the case $I=[0,n]$ and $\augversion{C}([0,n],\pi)$ is an antichain.

\begin{lem}\label{lem: aug chain extensions}
If $\sigma,\pi \in \hypn$ and $I \in \augchainsets$ then $\sigma \in \Lin(\augchainipi)$ if and only if $\aDes(\sigma^{-1}\pi) \subseteq I$.
\end{lem}

Next we define barred versions of both augmented zig-zag posets and augmented chain posets.
First, a \emph{barred augmented zig-zag poset} is defined to be an augmented zig-zag poset $\augzigipi$ with an arbitrary number of bars placed in each of the $n+1$ spaces between the two $0$ elements such that between any two bars the elements of the poset (not necessarily their labels) are increasing.
No bars are allowed on either end.
Figure \ref{fig: barred aug zig-zag poset} gives an example of a barred $\augzigipi$ poset with $I = \{0,2\}$ and $\pi = 2\bar{3}1$.

\begin{figure}[htbp]
\[\xymatrix @!R @!C @R = 30pt @C=10pt{ & \ar@{-}[ddd] &&&&& \ar@{-}[ddd]  & \ar@{-}[ddd]  &  & \ar@{-}[ddd] & \\   0 \ar@{-}[drrr]  &  & & & & \bar{3}\ar@{-}[drrr] & & & && 0  \\   & & & 2 \ar@{-}[urr]  & & & & & 1\ar@{-}[urr] &&  \\  &&&&&&&&&&   }\] 
\caption{\; A barred $\augzigipi$ poset for $I=\{0,2\}$ and $\pi = 2\bar{3}1$.}
\label{fig: barred aug zig-zag poset}
\end{figure}

To create a barred augmented zig-zag poset, we begin with an augmented zig-zag poset $\augzigipi$ and must first place a bar in space $i$ for each $i \in I$.
From there we are free to place any number of bars in any of the $n+1$ spaces.
Define $\Omegaaug_{\augzigipi}(j,k)$ to be the number of ordered pairs $(f,P)$ where $P$ is a barred $\augzigipi$ poset with $k$ bars and $f$ is an augmented $\augzigipi$-partition with parts less than or equal to $j$.
Recall that $\sigma \in \Lin(\augzigipi)$ if and only if $\aDes(\sigma^{-1}\pi) = I$.
If we begin with the augmented zig-zag poset $\augzigipi$ then there is a unique way to place the first $|I|$~bars (place one bar in space $i$ for each $i\in I$).
Next there are
\[
\mchoose{n+1}{k-|I|} = \binom{k + n - |I|}{n}
\]
ways to place the remaining $k-|I|$ bars in the $n+1$ spaces and hence there are $\binom{k+n-|I|}{n}$ barred $\augzigipi$ posets with $k$ bars.
Thus
\begin{align*}
\Omegaaug_{\augzigipi}(j,k) &= \sum_{\sigma \in \Lin(\augzigipi)} \Omegaaug_{\sigma}(j) \binom{k+n-\ades(\sigma^{-1}\pi)}{n} \\
&= \sum_{\sigma \in \Lin(\augzigipi)} \Omegaaug_{\sigma}(j) \Omegaaug_{\sigma^{-1}\pi}(k).
\end{align*}
If we set $\tau = \sigma^{-1}\pi$ and sum over all $I \in \augzigsets$ then we have
\begin{equation}\label{eq: barred aug zig-zag order polys}
\sum_{I \in \augzigsets} \Omegaaug_{\augzigipi}(j,k) = \sum_{\sigma \tau = \pi} \Omegaaug_{\sigma}(j) \Omegaaug_{\tau}(k).
\end{equation}

Next we define a \emph{barred augmented chain poset} to be an augmented chain poset $\augchainipi$ with at least one bar in space $i$ for each $i \in I$.
No bars are allowed in space~$i$ for $i\in [0,n]\setminus I$ and no bars are allowed on either end.
Thus we place at~least one bar between each chain of the augmented chain poset.
Figure \ref{fig: barred aug chain poset} gives an example of a barred $\augchainipi$ poset with $I = \{0,2,3\}$ and $\pi = 2\bar{3}1$.

\begin{figure}[htbp]
\[\xymatrix @!R @!C @R=32pt @C=15pt{ & \ar@{-}[ddd] &&&& \ar@{-}[ddd]  & \ar@{-}[ddd]  &  & \ar@{-}[ddd]  & \\     &  & & & \bar{3} & & & &&   \\  0   & & 2 \ar@{-}[urr]  & & & & & 1 && 0  \\  &&&&&&&&&   }\] 
\caption{\; A barred $\augchainipi$ poset for $I = \{0,2,3\}$ and $\pi = 2\bar{3}1$.}
\label{fig: barred aug chain poset}
\end{figure}

To create a barred augmented chain poset, we begin with an augmented chain poset $\augchainipi$ and must first place a bar in space $i$ for each $i \in I$.
From there we are free to place any number of bars in space $i$ for any $i \in I$.
This creates a collection of bars with each compartment containing at most one nonempty chain.
Define $\Omegaaug_{\augchainipi} (j,k)$ to be the number of ordered pairs $(f,P)$ where $P$ is a barred $\augchainipi$ poset with $k$ bars and $f$ is an augmented $\augchainipi$-partition with parts less than or equal to $j$.
Recall that $\sigma \in \Lin(\augchainipi)$ if and only if $\aDes(\sigma^{-1}\pi) \subseteq I$.
If we begin with the augmented chain poset $\augchainipi$ then there is a unique way to place the first $|I|$ bars (place one bar in space $i$ for each $i\in I$).
Next there are
\[
\mchoose{|I|}{k-|I|} = \binom{k-1}{k-|I|}
\]
ways to place the remaining $k-|I|$ bars in the $|I|$ allowable spaces and hence there are $\binom{k-1}{k-|I|}$ barred $\augchainipi$ posets with $k$ bars.
Thus
\[
\Omegaaug_{\augchainipi}(j,k) = \sum_{\sigma \in \Lin(\augchainipi)} \Omegaaug_{\sigma}(j) \binom{k-1}{k-|I|}.
\]

The following lemma shows that if $\sigma \in \hypn$ then the number of barred augmented zig-zag posets with $k$ bars such that $\sigma$ is a linear extension of the underlying augmented zig-zag poset is the same as the number of barred augmented chain posets with $k$ bars such that $\sigma$ is a linear extension of the underlying augmented chain poset.

\begin{lem}\label{lem: barred aug zig chain equiv}
Let $\pi \in \hypn$ and $j,k\geq0$.
Then
\[
\sum_{I \in \augzigsets}  \Omegaaug_{\augzigipi}(j,k)= \sum_{I \in \augchainsets} \Omegaaug_{\augchainipi}(j,k).
\]
\end{lem}

\begin{proof}
We must find a bijection between the two types of barred posets that respects the number of bars.
Recall that $\sigma \in \Lin(\augzigipi)$ for a unique $I \in \augzigsets$ and that $\sigma \in \Lin(C^{(a)}(J,\pi))$ if and only if $J \supseteq I$.
The desired bijection is simply the map that sends each barred augmented zig-zag poset with underlying poset $\augzigipi$ to the barred augmented chain poset obtained from $\augzigipi$ by removing the relation between $\pi(i)$ and $\pi(i+1)$ for every space $i$ containing at least one bar.
\end{proof}

The bijection in the previous lemma maps Figure \ref{fig: barred aug zig-zag poset} to Figure \ref{fig: barred aug chain poset}.
Now that all the pieces are in place, we are ready to prove the existence of the augmented descent algebra by computing $B^{(a)}_\pi (s,t)$.

\begin{thm}\label{thm: augmented algebra}
For every $\pi \in \hypn$,
\begin{equation}\label{eq: augmented algebra}
\sum_{j,k \geq 0} \binom{2jk + n-\ades(\pi)}{n} s^j t^k = \frac{B^{(a)}_\pi (s,t)}{(1-s)^{n+1}(1-t)^{n+1}}.
\end{equation}
\end{thm}

\begin{proof}
First we see that
\begin{align*}
\frac{B^{(a)}_\pi(s,t)}{(1-s)^{n+1}(1-t)^{n+1}} &=  \sum_{j,k \geq 0} \sum_{\sigma \tau = \pi} \binom{j+n}{n}\binom{k+n}{n}s^{j+\ades(\sigma)} t^{k+\ades(\tau)}\\
&= \sum_{j,k \geq 0} \sum_{\sigma \tau = \pi} \binom{j+n-\ades(\sigma)}{n}\binom{k+n-\ades(\tau)}{n}s^j t^k \\ 
&= \sum_{j,k \geq 0} \sum_{I \in \augzigsets} \Omegaaug_{\augzigipi}(j,k) s^j t^k
\end{align*}
where the last equality follows from equations \eqref{eq: aug order poly for pi} and \eqref{eq: barred aug zig-zag order polys}.
The bijection in Lemma~\ref{lem: barred aug zig chain equiv} allows us to shift our focus from barred augmented zig-zag posets to barred augmented chain posets and shows that
\[
\frac{B^{(a)}_\pi(s,t)}{(1-s)^{n+1}(1-t)^{n+1}} = \sum_{j,k \geq 0} \sum_{I \in \augchainsets} \Omegaaug_{\augchainipi}(j,k) s^j t^k.
\]
The only remaining step is to prove that
\[
\sum_{I \in \augchainsets} \Omegaaug_{\augchainipi}(j,k) =  \binom{2jk+n-\ades(\pi)}{n}.
\]

First we note that $\sum_{I \in \augchainsets} \Omegaaug_{\augchainipi}(j,k)$ counts ordered pairs $(f,P)$ where $P$ is a barred $\augchainipi$ poset with $k$ bars for some $I \in \augchainsets$ and $f$ is an augmented $\augchainipi$-partition with parts less than or equal to $j$.
Fix a barred $\augchainipi$ poset with $k$ bars and use the bars to define compartments labeled $0,\ldots,k$ from left to right.
Then define $\pi_i$ to be the (possibly empty) subword of $\pi$ in compartment $i$ and denote the length of $\pi_i$ by $L_i$.
Then
\[
\Omegaaug_{\augchainipi}(j) = \Omegaaug_{\pi_0}(j)\Omegaaug_{\overleftarrow{\pi_k}}(j)\prod_{i=1}^{k-1} \Omegaaug_{P(\pi_i)}(j) 
\]
where $\overleftarrow{\pi_{k}}$ is the word obtained by writing $\pi_{k}$ in reverse and negating each letter.
First we let $\Omegaaug_{\pi_0}(j)$ count solutions to the inequalities
\[
0 = s_{0_0} \leq s_{0_1} \leq  \cdots \leq s_{0_{L_0}} \leq s_{0_{L_0+1}} = j
\]
with $s_{0_l} < s_{0_{l+1}}$ if $l \in \aDes(\pi_0)$.
For $i=1,\ldots,k-1$, we let $\Omegaaug_{P(\pi_i)}(j)$ count solutions to the inequalities
\[
j+1+(i-1)(2j+1) = s_{i_0} \leq s_{i_1} \leq  \cdots \leq s_{i_{L_i}} \leq s_{i_{L_i+1}} = j+1+ (i-1)(2j+1) + 2j
\]
with $s_{i_l} < s_{i_{l+1}}$ if $l \in \aDes(\pi_i)$.
Finally we let $\Omegaaug_{\overleftarrow{\pi_k}}(j)$ count solutions to the inequalities
\[
j+1+(k-1)(2j+1) = s_{k_0} \leq s_{k_1} \leq  \cdots \leq s_{k_{L_k}} \leq s_{k_{L_k+1}} = 2jk+k
\]
with $s_{k_l} < s_{k_{l+1}}$ if $l \in \aDes(\pi_k)$.
We would like to concatenate these inequalities as before but we must make one correction to account for what happens when $\pi(l)$ crosses from one compartment to the next.

If $\pi(l) > 0$ then when $\pi(l)$ is the last letter of $\pi_i$ we have $s_{i_{L_i}} \neq 2ij+i+j$ but when $\pi(l)$ is the first letter of $\pi_{i+1}$ we do allow $s_{(i+1)_1} = 2ij+i+j+1$.
Similarly, if $\pi(l) < 0$ then when $\pi(l)$ is the last letter of $\pi_i$ we allow $s_{i_{L_i}} = 2ij+i+j$ but when $\pi(l)$ is the first letter of $\pi_{i+1}$ we do not allow $s_{(i+1)_1} = 2ij+i+j+1$.
Thus we follow Petersen's augmented lexicographic ordering \cite{Petersen2005} and form an equivalence class of points by declaring that $(2ij+i+j) \sim (2ij+i+j+1)$ for $i=0,\ldots,k-1$.
The equivalence relation returns a set of size $2jk$.
We then see that if we sum over all $I \in \augchainsets$ and all barred $\augchainipi$ posets with $k$ bars then $\sum_{I \in \augchainsets} \Omegaaug_{\augchainipi}(j,k)$ is equal to the number of solutions to the inequalities
\[
0 = s_0 \leq s_1 \leq \cdots \leq s_n \leq s_{n+1} = 2jk
\]
with $s_i < s_{i+1}$ if $i \in \aDes(\pi)$.
Hence
\[
\sum_{I \in \augchainsets} \Omegaaug_{\augchainipi}(j,k)  = \binom{2jk+n-\ades(\pi)}{n} .  \qedhere
\]
\end{proof}

Define the augmented structure polynomial $\psi(x)$ in the group algebra of $\hypn$ by
\[
\psi(x) = \sum_{\pi \in \hypn} \binom{x+n-\ades(\pi)}{n} \pi.
\]
If we expand both sides of equation \eqref{eq: augmented algebra} and compare the coefficients of $s^jt^k$ we have
\[
\binom{2jk+n-\ades(\pi)}{n} = \sum_{\sigma\tau = \pi}  \binom{j+n -\ades(\sigma)}{n} \binom{k+n - \ades(\tau)}{n} .
\]
This implies that $\psi(j)\psi(k) = \psi(2jk)$ for all $j,k \geq 0$ and so we have the following theorem, originally proven by Petersen \cite{Petersen2005}.

\begin{thm}[Petersen]\label{thm: ades order poly}
As polynomials in $x$ and $y$ with coefficients in the group algebra of $\hypn$,
\[
\psi(x)\psi(y) = \psi(2xy).
\]
\end{thm}

If we substitute $x \leftarrow x/2$ and $y \leftarrow y/2$ in the previous theorem then we see that $\psi(x/2)\psi(y/2) = \psi(xy/2)$.
As before, if we expand
\[
\psi(x/2) = \sum_{\pi \in \hypn} \binom{\frac{x}{2}+n-\ades(\pi)}{n} \pi = \sum_{i=1}^n \augversion{b}_i x^i
\]
then the $\augversion{b}_i$ form a set of orthogonal idempotents which span the augmented Eulerian descent algebra.
These orthogonal idempotents were originally described by Petersen in \cite{Petersen2005}.

At this point it is natural to ask whether or not the augmented descent set (as opposed to number) induces an algebra.
In $\mathfrak{B}_3$, the two signed permutations with augmented descent set $\{0,1\}$ are $\bar{1}\bar{3}\bar{2}$ and $\bar{2}\bar{3}\bar{1}$.
The only signed permutation with augmented descent set $\{0\}$ is $\bar{3}\bar{2}\bar{1}$.
However,
\[
(\bar{3}\bar{2}\bar{1})(\bar{1}\bar{3}\bar{2}+\bar{2}\bar{3}\bar{1}) = 312 + 213
\]
and both of those signed permutations have augmented descent set $\{1,3\}$.
The signed permutation $2\bar{1}3$ has the same augmented descent set.
Hence this product cannot be written as a linear combination of basis elements which implies that the augmented descent set does not induce an algebra.
This is somewhat unexpected since such an algebra would fit nicely between the Mantaci-Reutenauer algebra described in Chapter \ref{chpt: Colored Permutation Groups} and the augmented Eulerian descent algebra.


\section{The flag descent algebra}\label{sec: flag descent algebra}

Adin, Brenti, and Roichman \cite{AdinBrentiRoichman2001} defined a permutation statistic called the flag descent number during their efforts extending the Carlitz identity to $\hypn$.
The flag descent number $\fdes(\pi)$ of a signed permutation $\pi$ is
\[
\fdes(\pi) = \desa(\pi) + \desb(\pi) .
\]
The flag descent statistic thus counts a descent in position $0$ once and all other descents twice.
We can tell whether or not a permutation begins with a negative letter by analyzing the parity of $\fdes(\pi)$.

The goal of this section is to show that the group algebra of the hyperoctahedral group $\hypn$ contains a subalgebra induced by flag descent number called the \emph{flag descent algebra}.
We define
\[
F_i = \dsum_{\fdes(\pi) = i} \pi
\]
for $i=0,\ldots,2n-1$ and will show that together the $F_i$ form a basis for an algebra.
To prove that such an algebra exists, we compute $F_\pi(s,t) := \sum_{\sigma \tau = \pi} s^{\fdes(\sigma)}t^{\fdes(\tau)}$ and show that $F_\pi(s,t)$ is determined by $\fdes(\pi)$.

The following definition is due to Gessel \cite{GesselPersonal}.
For any signed poset $P$, the \emph{flag order quasi-polynomial}, denoted by $\Omegaflag_P(j)$, is defined to be the number of $P$-partitions of type $B$ with parts less than or equal to $j$ such that $f(i) \equiv j \mod{2}$ for $i=1,\ldots,n$.
We call such $P$-partitions \emph{flag $P$-partitions}.

Every signed permutation $\pi \in \hypn$ can be represented by the signed poset
\[
\pi(\bar{n}) <_P \cdots <_P \pi(\bar{1}) <_P 0 <_P \pi(1) <_P \cdots <_P \pi(n)
\]
and so $\Omegaflag_\pi(j)$ is equal to the number of integer solutions to the set of inequalities
\[
0 = i_0 \leq i_1 \leq i_2 \leq \cdots \leq i_n \leq j
\]
with $i_k < i_{k+1}$ if $k \in \Desb(\pi)$ and $i_k \equiv j \mod{2}$ for $k=1,\ldots,n$.
As before, we would like to reduce strict inequalities to weak inequalities but we do so in two steps.
First, if $i_0 < i_1$ then we subtract $1$ from $i_k$ for $k \geq 1$.
Next, if $i_l < i_{l+1}$ for $l>0$ then we subtract $2$ from $i_k$ for $k > l$.
Thus, after reducing strict inequalities to weak inequalities, we see that $\Omegaflag_\pi(j)$ is equal to the number of integer solutions to the set of inequalities
\[
0 \leq i_1 \leq i_2 \leq \cdots \leq i_n \leq j-\fdes(\pi)
\]
with $i_k \equiv j-\fdes(\pi) \mod{2}$ for $k=1,\ldots,n$.
Hence
\begin{equation}\label{eq: flag order poly formula}
\Omegaflag_\pi(j) = \mchoose{\ceiling{\frac{1}{2}(j+1-\fdes(\pi))}}{n} = \binom{\ceiling{\frac{1}{2}(j-1-\fdes(\pi))}+n}{n}.
\end{equation}
Since $\Omegaflag_P(j)$ counts certain $P$-partitions of type $B$, Corollaries \ref{cor: type B order poly linear extensions} and \ref{cor: type B order poly product} give us the following theorems.

\begin{thm}
\[
\Omegaflag_P(j) = \sum_{\pi \in \Lin(P)} \Omegaflag_\pi(j)
\]
\end{thm}

\begin{thm}
If $P_1$ and $P_2$ are two disjoint signed posets then
\[
\Omegaflag_{P_1 \sqcup P_2}(j) = \Omegaflag_{P_1}(j) \Omegaflag_{P_2}(j).
\]
\end{thm}

Next we define barred versions of both type $B$ zig-zag posets and type $B$ chain posets in a way that corresponds to the flag descent statistic.
First, a \emph{flag-barred zig-zag poset} is defined to be a type $B$ zig-zag poset $\zigbipi$ with an arbitrary number of bars placed in space $0$ and an even number of bars placed in each of the $n$ other spaces such that between any two bars the elements of the poset (not necessarily their labels) are increasing.
No bars are allowed to the left of the $0$ element.
Figure \ref{fig: flag-barred zig-zag poset} gives an example of a flag-barred $\zigbipi$ poset with $I = \{0,3\}$ and $\pi = 2\bar{3}14$.

\begin{figure}[htbp]
\[\xymatrix @!R @!C @R=16pt @C=5pt{  & \ar@{-}[dddd] & \ar@{-}[dddd] & \ar@{-}[dddd] & & & & & & \ar@{-}[dddd] & \ar@{-}[dddd] & & \ar@{-}[dddd] & \ar@{-}[dddd] & \ar@{-}[dddd] & \ar@{-}[dddd] & & \ar@{-}[dddd] & \ar@{-}[dddd]  \\  & & & & & & & & & & & 1\ar@{-}[drrrrr] & & & & & & &  \\  0\ar@{-}[drrrrr] & & & & & & & & \bar{3}\ar@{-}[urrr] & & & & & & & & 4 & &  \\  & & & & & 2\ar@{-}[urrr] & & & & & & & & & & & & & \\  &&&&&&&&&&&&&&&&&& \\ }\] 
\caption{\; A flag-barred $\zigbipi$ poset with $I = \{0,3\}$ and $\pi = 2\bar{3}14$.}
\label{fig: flag-barred zig-zag poset}
\end{figure}

To create a flag-barred zig-zag poset, we begin with a type $B$ zig-zag poset $\zigbipi$ and must first place one bar in space $0$ if $0 \in I$ and two bars in space $i$ if $i \in I$ and $i \neq 0$.
From there we are free to place any number of bars in space $0$ and any even number of bars in any of the other $n$ spaces.
Define $\Omegaflag_{\zigbipi} (j,k)$ to be the number of ordered pairs $(f,P)$ where $P$ is a flag-barred $\zigbipi$ poset with $k$ bars and $f$ is a flag $\zigbipi$-partition with parts less than or equal to $j$.
Recall that $\sigma \in \Lin(\zigbipi)$ if and only if $\Desb(\sigma^{-1}\pi) = I$.
For notational ease, we define the flag descent number of a set $I \subseteq [0,n-1]$ by
\[
\fdes(I) = \left\{\begin{array}{cl} 2 |I|, & \text{if } 0 \notin I \\ 2|I|-1, & \text{if } 0 \in I \end{array}\right. .
\]
If we begin with the poset $\zigbipi$ then there is a unique way to place the first $\fdes(I)$ bars.
Next we must place the remaining $k-\fdes(I)$ bars and there are two cases.

If $k-\fdes(I)$ is even then an even number of bars must go in space $0$.
Let $a = (k-\fdes(I))/2$.
Next, by placing $2l$ bars in space $0$ and summing on $l$, we see that there are
\begin{align*}
\sum_{l=0}^{a} \mchoose{n}{a-l} &= \sum_{l=0}^{a} \binom{n +a-l-1}{n-1} \\
&= \binom{a+n}{n} \\
&= \binom{\ceiling{\frac{1}{2}(k-1-\fdes(I))} + n}{n}
\end{align*}
ways to place the remaining $k-|I|$ bars.

If $k-\fdes(I)$ is odd then an odd number of bars must go in space $0$.
Let $a = (k-\fdes(I)-1)/2$.
Next, by placing $2l+1$ bars in space $0$ and summing on $l$, we see that there are
\begin{align*}
\sum_{l=0}^{a} \mchoose{n}{a-l} &= \sum_{l=0}^{a} \binom{n +a-l-1}{n-1} \\
&= \binom{a+n}{n} \\
&= \binom{\ceiling{\frac{1}{2}(k-1-\fdes(I))} + n}{n}
\end{align*}
ways to place the remaining $k-|I|$ bars.
Hence we see that
\begin{align*}
\Omegaflag_{\zigbipi}(j,k) &= \sum_{\sigma \in \Lin(\zigbipi)} \Omegaflag_{\sigma}(j) \binom{\ceiling{\frac{1}{2}(k-1-\fdes(\sigma^{-1}\pi))} +n}{n} \\
&= \sum_{\sigma \in \Lin(\zigbipi)} \Omegaflag_{\sigma}(j) \Omegaflag_{\sigma^{-1}\pi}(k).
\end{align*}
If we set $\tau = \sigma^{-1}\pi$ and sum over all $I \subseteq [0,n-1]$ then
\begin{equation}\label{eq: flag-barred zig-zag order polys}
\sum_{I \subseteq [0,n-1]} \Omegaflag_{\zigbipi}(j,k) = \sum_{\sigma \tau = \pi} \Omegaflag_{\sigma}(j) \Omegaflag_{\tau}(k).
\end{equation}

Next we define a \emph{flag-barred chain poset} to be a poset $\chainbipi$ with at least one bar in space $0$ if $0 \in I$ and an even positive number of bars in space $i$ for $i\in I$, $i \neq 0$.
We also allow an even number of bars placed on the right end.
No bars are allowed in space $i$ for $i\in [0,n-1]\setminus I$ and no bars are allowed to the left of the $0$ element.
Figure \ref{fig: flag-barred chain poset} gives an example of a flag-barred $\chainbipi$ poset with $I = \{0,2,3\}$ and $\pi = 2\bar{3}14$.

\begin{figure}[htbp]
\[\xymatrix @!R @!C @C=8pt {  & \ar@{-}[ddd] & \ar@{-}[ddd] & \ar@{-}[ddd] & & &  \ar@{-}[ddd] & \ar@{-}[ddd] & & \ar@{-}[ddd] & \ar@{-}[ddd] & \ar@{-}[ddd] & \ar@{-}[ddd] & & \ar@{-}[ddd] & \ar@{-}[ddd]  \\  &&&&& \bar{3}  &&&  &&&&& &&   \\  0 &&&& 2\ar@{-}[ur] &&&& 1 &&&&& 4 &&  \\  &&&&&&&&&&&&&&&   }\] 
\caption{\; A flag-barred $\chainbipi$ poset with $I = \{0,2,3\}$ and $\pi = 2\bar{3}14$.}
\label{fig: flag-barred chain poset}
\end{figure}

To create a flag-barred chain poset, we begin with a poset $\chainbipi$ and must first place a bar in space $0$ if $0\in I$ and two bars in space $i$ for each $i \in I$, $i \neq 0$.
From there we are free to place any number of bars in space $0$ and any even number of bars in the other spaces indicated by $I$ and on the right end.
This creates a collection of bars with each compartment containing at most one nonempty chain labeled by a subword of $\pi$ and with $0$ in the leftmost compartment.
Also, all the chains (except possibly for a singleton $0$ element) are in compartments with the same parity.
Define $\Omegaflag_{\chainbipi} (j,k)$ to be the number of ordered pairs $(f,P)$ where $P$ is a flag-barred $\chainbipi$ poset with $k$ bars and $f$ is a flag $\chainbipi$-partition with parts less than or equal to $j$.

The following lemma shows that if $\sigma \in \hypn$ then the number of flag-barred zig-zag posets with $k$ bars such that $\sigma$ is a linear extension of the underlying type $B$ zig-zag poset is the same as the number of flag-barred chain posets with $k$ bars such that $\sigma$ is a linear extension of the underlying type $B$ chain poset.

\begin{lem}\label{lem: flag-barred zig chain equiv}
Let $\pi \in \hypn$ and $j,k\geq0$.
Then
\[
\sum_{I \subseteq [0,n-1]}  \Omegaflag_{\zigbipi}(j,k)= \sum_{I \subseteq [0,n-1]} \Omegaflag_{\chainbipi}(j,k).
\]
\end{lem}

\begin{proof}
We must find a bijection between the two types of barred posets that respects the number of bars.
Recall that $\sigma \in \Lin(\zigbipi)$ for a unique $I \subseteq [0,n-1]$ and that $\sigma \in \Lin(C^{B}(J,\pi))$ if and only if $J \supseteq I$.
The desired bijection is simply the map that sends each flag-barred zig-zag poset with underlying poset $\zigbipi$ to the flag-barred chain poset obtained from $\zigbipi$ by removing the relation between $\pi(i)$ and $\pi(i+1)$ for every space $i$ containing at least one bar.
\end{proof}

The bijection in the previous lemma maps Figure \ref{fig: flag-barred zig-zag poset} to Figure \ref{fig: flag-barred chain poset}.
Now that all the pieces are in place, we are ready to prove the existence of the flag descent algebra by computing $F_\pi (s,t)$.

\begin{thm}\label{thm: flag descent algebra}
For every $\pi \in \hypn$,
\begin{equation}\label{eq: flag descent algebra}
\sum_{j,k \geq 0} \binom{\ceiling{\frac{1}{2}(jk+j+k-1-\fdes(\pi))}+n}{n} s^j t^k = \frac{F_\pi (s,t)}{(1-s)(1-s^2)^{n}(1-t)(1-t^2)^{n}}.
\end{equation}
\end{thm}

\begin{proof}
First we expand the right side of equation \eqref{eq: flag descent algebra} to get
\[
\sum_{j,k \geq 0} \sum_{\sigma \tau = \pi} \binom{\ceiling{\frac{1}{2}(j-1)}+n}{n}\binom{\ceiling{\frac{1}{2}(k-1)}+n}{n}s^{j+\fdes(\sigma)} t^{k+\fdes(\tau)}
\]
which is equal to 
\[
\sum_{j,k \geq 0} \sum_{\sigma \tau = \pi} \binom{\ceiling{\frac{1}{2}(j-1-\fdes(\sigma))}+n}{n} \binom{\ceiling{\frac{1}{2}(k-1-\fdes(\tau))}+n}{n} s^j t^k.
\]
Equations \eqref{eq: flag order poly formula} and \eqref{eq: flag-barred zig-zag order polys} imply that
\[
\frac{F_\pi(s,t)}{(1-s)(1-s^2)^{n}(1-t)(1-t^2)^{n}} = \sum_{j,k \geq 0} \sum_{I \subseteq [0,n-1]} \Omegaflag_{\zigbipi}(j,k) s^j t^k.
\]
The bijection in Lemma \ref{lem: flag-barred zig chain equiv} allows us to shift our focus from flag-barred zig-zag posets to flag-barred chain posets and shows that
\[
\frac{F_\pi(s,t)}{(1-s)(1-s^2)^{n}(1-t)(1-t^2)^{n}} = \sum_{j,k \geq 0} \sum_{I \subseteq [0,n-1]} \Omegaflag_{\chainbipi}(j,k) s^j t^k.
\]
The only remaining step is to prove that
\[
\sum_{I \subseteq [0,n-1]} \Omegaflag_{\chainbipi}(j,k) = \binom{\ceiling{\frac{1}{2}(jk+j+k-1-\fdes(\pi))}+n}{n}.
\]
There are two cases to consider.

First, assume that $k$ is even.
We note that $\sum_{I \subseteq [0,n-1]} \Omegaflag_{\chainbipi}(j,k)$ counts ordered pairs $(f,P)$ where $P$ is a flag-barred $\chainbipi$ poset with $k$ bars for some $I \subseteq [0,n-1]$ and $f$ is a flag $\chainbipi$-partition with parts less than or equal to $j$.
Fix a flag-barred $\chainbipi$ poset with $k$ bars and use the bars to define compartments labeled $0,\ldots,k$ from left to right.
Then define $\pi_i$ to be the (possibly empty) subword of $\pi$ in compartment $i$ and denote the length of $\pi_i$ by $L_i$.
Then
\[
\Omegaflag_{\chainbipi}(j) = \Omegaflag_{\pi_0}(j) \prod_{i=1}^{k/2} \Omegaflag_{P(\pi_{2i})}(j). 
\]
First we let $\Omegaflag_{\pi_0}(j)$ count solutions to the inequalities
\[
0=s_{0_0} \leq s_{0_1} \leq  \cdots \leq s_{0_{L_0}} \leq j
\]
with $s_{0_l} < s_{0_{l+1}}$ if $l \in \Desb(\pi_0)$ and with $s_{0_l} \equiv j \mod{2}$ for $l=1,\ldots,L_0$.
For $i=1,\ldots,k/2$, we let $\Omegaflag_{P(\pi_{2i})}(j)$ count solutions to the inequalities
\[
j+2+(i-1)(2j+2) \leq s_{i_1} \leq  \cdots \leq s_{i_{L_{2i}}} \leq j+2+(i-1)(2j+2) + 2j
\]
with $s_{i_l} < s_{i_{l+1}}$ if $l \in \Desa(\pi_i)$ and with $s_{i_l} \equiv j \mod{2}$ for $l=1,\ldots,L_{2i}$.
Note that we use $\Desa$ instead of $\Desb$ because a descent at the beginning of $\pi_{2i}$ does not affect the number of flag $P(\pi_{2i})$-partitions.
Also, we purposefully skip the values $j+1+i(2j+2)$ for $i=0,\ldots,k/2-1$ so that our solutions will all have the same parity.
By concatenating these inequalities, we see that if we sum over all $I \subseteq [0,n-1]$ and all flag-barred $\chainbipi$ posets with $k$ bars then $\sum_{I \subseteq [0,n-1]} \Omegaflag_{\chainbipi}(j,k)$ is equal to the number of solutions to the inequalities
\[
0 = s_0 \leq s_1 \leq \cdots \leq s_n \leq jk+j+k
\]
with $s_i < s_{i+1}$ if $i \in \Desb(\pi)$ and with $s_i \equiv j \mod{2}$ for $i=1,\ldots,n$.
The number of these flag $\pi$-partitions is
\[
\binom{\ceiling{\frac{1}{2}(jk+j+k+1-\fdes(\pi))}+n-1}{n} = \binom{\ceiling{ \frac{1}{2}(jk+j+k-1-\fdes(\pi))}+n}{n}.
\]

Next we consider the case where $k$ is odd.
Fix a flag-barred $\chainbipi$ poset with $k$ bars and use the bars to define compartments labeled $0,\ldots,k$ from left to right.
Then define $\pi_i$ to be the (possibly empty) subword of $\pi$ in compartment $i$ and denote the length of $\pi_i$ by $L_i$.
Then
\[
\Omegaflag_{\chainbipi}(j) = \prod_{i=0}^{(k-1)/2} \Omegaflag_{P(\pi_{2i+1})}(j) .
\]
For $i=0,\ldots,(k-1)/2$, we let $\Omegaflag_{P(\pi_{2i+1})}(j)$ count solutions to the inequalities
\[
i(2j+2) \leq s_{i_1} \leq  \cdots \leq s_{i_{L_{2i+1}}} \leq i(2j+2) + 2j
\]
with $s_{i_l} < s_{i_{l+1}}$ if $l \in \Desa(\pi_i)$ and with $s_{i_l} \equiv j \mod{2}$ for $l=1,\ldots,L_{2i+1}$.
We again use $\Desa$ instead of $\Desb$ because a descent at the beginning of $\pi_{2i+1}$ does not affect the number of flag $P(\pi_{2i+1})$-partitions.
We purposefully skip the values $i(2j+2)-1$ for $i=1,\ldots,(k-1)/2$ so that our solutions will all have the same parity.
By concatenating these inequalities, we see that if we sum over all $I \subseteq [0,n-1]$ and all flag-barred $\chainbipi$ posets with $k$ bars then $\sum_{I \subseteq [0,n-1]} \Omegaflag_{\chainbipi}(j,k)$ is equal to the number of solutions to the inequalities
\[
0 \leq s_1 \leq \cdots \leq s_n \leq jk+j+k-1
\]
with $s_i < s_{i+1}$ if $i \in \Desa(\pi)$ and with $s_i \equiv j \mod{2}$ for $i=1,\ldots,n$.
The number of these flag $\pi$-partitions is
\[
\binom{\ceiling{ \frac{1}{2}(jk+j+k-2\desa(\pi))}+n-1}{n}.
\]
If $\fdes(\pi)$ is even then $\fdes(\pi) = 2\desa(\pi)$ and so the number of flag $\pi$-partitions is
\[
\binom{\ceiling{ \frac{1}{2}(jk+j+k-\fdes(\pi))}+n-1}{n}  = \binom{\ceiling{ \frac{1}{2}(jk+j+k-1-\fdes(\pi))}+n}{n}
\]
because $jk+j+k-\fdes(\pi)$ is odd.
If $\fdes(\pi)$ is odd then $\fdes(\pi) = 2\desa(\pi) + 1$ and we see that the number of flag $\pi$-partitions is
\[
\binom{\ceiling{\frac{1}{2}(jk+j+ k-(\fdes(\pi)-1))}+n-1}{n} = \binom{\ceiling{ \frac{1}{2}(jk+j+k-1-\fdes(\pi))}+n}{n}
\]
which completes the proof.
\end{proof}

Define the flag structure quasi-polynomial $\zeta(x)$ in the group algebra of $\hypn$ by
\[
\zeta(x) = \sum_{\pi \in \hypn} \binom{\ceiling{\frac{1}{2}(x-\fdes(\pi))}+n-1}{n} \pi.
\]
The coefficient of $s^jt^k$ on the left side of equation \eqref{eq: flag descent algebra} is
\[
\binom{\ceiling{ \frac{1}{2}(jk+j + k-1-\fdes(\pi))} +n}{n}.
\]
If we expand the right side of equation \eqref{eq: flag descent algebra} we see that the coefficient of $s^jt^k$ is
\[
\sum_{\sigma\tau = \pi}  \binom{\ceiling{ \frac{1}{2}(j-1-\fdes(\sigma))}+n}{n} \binom{\ceiling{ \frac{1}{2}(k-1-\fdes(\tau)) }+n}{n} .
\]
This implies that $\zeta(j+1)\zeta(k+1) = \zeta((j+1)(k+1))$ for all $j,k \geq 0$ and so we have the following theorem.

\begin{thm}
As quasi-polynomials in $x$ and $y$ with coefficients in the group algebra of $\hypn$,
\[
\zeta(x)\zeta(y) = \zeta(xy).
\]
\end{thm}

Because $\zeta(x)$ is a quasi-polynomial, we cannot simply expand $\zeta(x)$ in powers of~$x$.
However, we can still find a set of orthogonal idempotents which span the flag descent algebra.
First note that
\[
\zeta(x) = \sum_{i = 0}^{2n-1} \binom{\ceiling{ \frac{1}{2}(x-i)}+n-1}{n} F_i.
\]
If $x$ is even we get
\[
\zeta_e(x) := \sum_{i = 0}^{n-1} \binom{ \frac{x}{2} - i +n-1}{n} (F_{2i}+F_{2i+1}) = \phi_A(x/2).
\]
If $x$ is odd we get
\[
\zeta_o(x) := \sum_{i = 0}^{n} \binom{ \frac{x-1}{2} - i +n}{n} (F_{2i-1}+F_{2i}) = \phi_B((x-1)/2)
\]
with $F_{-1} = F_{2n} = 0$.
Thus $\phi_A(x/2) = \zeta_e(x)$ and $\phi_B((x-1)/2) = \zeta_o(x)$.
By analyzing the parity of $x$ and $y$, we have the following:
\begin{align*}
\zeta_e(x)\zeta_e(y) &= \zeta_e(xy) \\
\zeta_e(x)\zeta_o(y) &= \zeta_e(xy) \\
\zeta_o(x)\zeta_e(y) &= \zeta_e(xy) \\
\zeta_o(x)\zeta_o(y) &= \zeta_o(xy).
\end{align*}
We already knew that $\phi_A(x/2)\phi_A(y/2) = \phi_A(xy/2)$ and $\phi_B((x-1)/2)\phi_B((y-1)/2) = \phi_B((xy-1)/2)$ but this also tells us that $\phi_A(x/2)\phi_B((y-1)/2) = \phi_A(xy/2)$ and $\phi_B((x-1)/2)\phi_A(y/2) = \phi_A(xy/2)$.
On the level of the orthogonal idempotents $a_i$ and $b_i$ which span the type~$A$ and type~$B$ Eulerian descent algebras respectively, we have $a_ib_i = a_i = b_i a_i$ and $a_ib_j = 0 = b_ja_i$ if $i\neq j$. 
This result leads us to the study of ideals in the next section.

\looseness-1 For $i=1,\ldots,n$ let $f_{2i} = a_i$ and for $i=0,\ldots,n$ set $f_{2i+1} = b_i - a_i$ with $a_0 := 0$.
Since $a_n = b_n$, we see that $f_{2n+1} = 0$.
However, we claim that $\{f_i \, : \, i\in [2n]\}$ is a basis for the flag descent algebra.
We already know that the $a_i$ are orthogonal idempotents and hence so are the $f_{2i}$ terms.
Next we see that
\[
f_{2i+1}^2 =  (b_i - a_i)^2 = b_i - b_ia_i - a_i b_i + a_i = b_i -a_i = f_{2i+1}.
\]
If $i\neq j$ then clearly $f_{2i+1}f_{2j+1}= 0$ and $f_{2i}f_{2j+1} = 0$.
Lastly,
\[
f_{2i}f_{2i+1} = a_i (b_i - a_i) = a_i-a_i = 0.
\]
Since the $f_i$ span the flag descent algebra, we see that they form a basis of orthogonal idempotents.
This also shows that the algebra spanned by the orthogonal idempotents $a_i$ and $b_i$ is the flag descent algebra.
We comment more on this in the next section.
The style of these idempotents is similar to the orthogonal idempotents in \cite[Theorem 2]{CelliniII1995}.

\begin{ex}
Let $n=2$ and let $F_i$ be the sum of all permutations in $\mathfrak{B}_2$ with flag descent number $i$.
Then
\begin{align*}
a_2 &= \frac{1}{8} (F_0 + F_1 + F_2 + F_3) & b_2 &= \frac{1}{8} (F_0 + F_1 + F_2 + F_3) \\
a_1 &= \frac{1}{8} (2\cdot F_0 + 2\cdot F_1 - 2\cdot F_2 - 2 \cdot F_3) & b_1 &= \frac{1}{8} (4\cdot F_0 - 4 \cdot F_3) \\
& & b_0 &= \frac{1}{8} (3\cdot F_0 - F_1 - F_2 + 3\cdot F_3) .
\end{align*}
Hence the following orthogonal idempotents form a basis for the flag descent algebra:
\begin{align*}
f_1 &= \frac{1}{8} (3\cdot F_0 - F_1 - F_2 + 3\cdot F_3) \\
f_2 &= \frac{1}{8} (2\cdot F_0 + 2\cdot F_1 - 2\cdot F_2 - 2 \cdot F_3) \\
f_3 &= \frac{1}{8} (2\cdot F_0 - 2\cdot F_1 + 2\cdot F_2 - 2 \cdot F_3) \\
f_4 &= \frac{1}{8} (F_0 + F_1 + F_2 + F_3). 
\end{align*}
\end{ex}


\section{Ideals}\label{sec: ideals}

In \cite{Petersen2005}, Petersen gives a formula for the product of the type $B$ structure polynomial $\phi_B(x)$ with the augmented structure polynomial $\psi(y)$ and in doing so proves that the augmented Eulerian descent algebra is an ideal in the algebra spanned by the set of basis elements from the two algebras.
We extend this result to pairs of algebras selected from amongst the type $A$ Eulerian descent algebra, the type $B$ Eulerian descent algebra, and the augmented Eulerian descent algebra.
The proofs follow those from the previous section but we now consider $P$-partitions for barred posets in which the permutation statistic associated with the type of $P$-partitions used and the permutation statistic associated with the type of zig-zag posets used are allowed to be different.
For notational ease, we let $\Omega^{1}$ indicate $P$-partitions of type 1 where the type is either type $A$, type $B$, or augmented.
Let $Z^2(n)$ denote the set of all possible type 2 zig-zag sets where the type is chosen from the same list and let $Z^2(I,\pi)$ denote the type $2$ zig-zag poset corresponding to $I \in Z^2(n)$.
Similarly, let $C^2(n)$ denote the set of all possible type 2 chain sets where the type is chosen from the same list and let $C^2(I,\pi)$ denote the type $2$ chain poset corresponding to $I \in C^2(n)$.

\begin{thm}\label{thm: generic zig-zag order poly}
Define $\Omega^{1}_{Z^2(I,\pi)}(j,k)$ to be the number of ordered pairs $(f,P)$ where $P$ is a barred $Z^2(I,\pi)$ poset with $k$ bars and $f$ is a $Z^2(I,\pi)$-partition of type~$1$ with parts less than or equal to $j$.
Then
\[
\sum_{I \in Z^2(n)} \Omega^1_{Z^2(I,\pi)} (j,k) = \sum_{\sigma \tau = \pi} \Omega^1_{\sigma}(j)\Omega^2_{\tau}(k).
\]
\end{thm}

\begin{proof}
Each $\sigma \in \hypn$ is a linear extension of a unique type $2$ zig-zag poset $Z^2(I,\pi)$.
The number of $\sigma$-partitions of type $1$ with parts less than or equal to $j$ is $\Omega_{\sigma}^1(j)$.
If $\sigma \in \Lin(Z^2(I,\pi))$ then $\Omega_{\sigma^{-1}\pi}^2(k)$ counts the number of barred $Z^2(I,\pi)$ posets with $k$ bars.
\end{proof}

The following theorem shows that we can still shift from barred zig-zag posets to barred chain posets.

\begin{thm}\label{thm: generic zig-zag chain equivalence}
\[
\sum_{I \in Z^2(n)} \Omega^1_{Z^2(I,\pi)}(j,k) = \sum_{I \in C^2(n)} \Omega^1_{C^2(I,\pi)}(j,k)
\]
\end{thm}

\begin{proof}
The analogous bijection works for all possible cases because it relies solely on the fact that $\sigma \in \Lin(Z^2(I,\pi))$ for a unique $I \in Z^2(n)$ and $\sigma \in \Lin(C^2(J,\pi))$ for all $J \in C^2(n)$ such that $J \supseteq I$.
\end{proof}

\begin{thm}\label{thm: desa desb ideal}
For every $\pi \in \hypn$,
\begin{equation}\label{eq: desa desb ideal}
\sum_{j,k \geq 0} \binom{2jk + j + 2k + n - \desa(\pi)}{n} s^j t^k = \frac{\dsum_{\sigma\tau = \pi} s^{\desa(\sigma)} t^{\desb(\tau)}}{(1-s)^{n+1}(1-t)^{n+1}}
\end{equation}
and
\begin{equation}\label{eq: desb desa ideal}
\sum_{j,k \geq 0} \binom{2jk + 2j + k + n - \desa(\pi)}{n} s^j t^k = \frac{\dsum_{\sigma\tau = \pi} s^{\desb(\sigma)} t^{\desa(\tau)}}{(1-s)^{n+1}(1-t)^{n+1}}.
\end{equation}
\end{thm}

\begin{proof}
If we expand the right side of equation \eqref{eq: desa desb ideal} and compare the coefficient of $s^jt^k$ on both sides then we have
\begin{equation}\label{eq: jk desa desb ideal}
\binom{2jk+j+2k+n-\desa(\pi)}{n} = \sum_{\sigma\tau = \pi} \binom{j+n-\desa(\sigma)}{n} \binom{k+n-\desb(\tau)}{n}.
\end{equation}
Using Theorem \ref{thm: generic zig-zag order poly}, we see that
\[
\sum_{\sigma\tau = \pi} \binom{j+n-\desa(\sigma)}{n} \binom{k+n-\desb(\tau)}{n} = \sum_{I \subseteq [0,n-1]} \OmegaA_{\zigbipi} (j,k)
\]
and, using Theorem \ref{thm: generic zig-zag chain equivalence}, we see that
\[
\sum_{I \subseteq [0,n-1]} \OmegaA_{\zigbipi} (j,k) = \sum_{I \subseteq [0,n-1]} \OmegaA_{\chainbipi}(j,k).
\]
Thus to prove equation~\eqref{eq: desa desb ideal} it remains to show that
\[
\sum_{I \subseteq [0,n-1]} \OmegaA_{\chainbipi}(j,k) = \binom{2jk+j+2k+n-\desa(\pi)}{n}.
\]

First we note that $\sum_{I \subseteq [0,n-1]} \OmegaA_{\chainbipi}(j,k)$ counts ordered pairs $(f,P)$ where $P$ is a barred $\chainbipi$ poset with $k$ bars for some $I \subseteq [0,n-1]$ and $f$ is a $\chainbipi$-partition of type~$A$ with parts less than or equal to $j$.
Fix a barred $\chainbipi$ poset with $k$ bars and use the bars to define compartments labeled $0,\ldots,k$ from left to right.
Then define $\pi_i$ to be the (possibly empty) subword of $\pi$ in compartment $i$ and denote the length of $\pi_i$ by $L_i$.
Then
\[
\OmegaA_{\chainbipi}(j) = \OmegaA_{\pi_0}(j)\prod_{i=1}^{k} \OmegaA_{P(\pi_i)}(j) .
\]
First we let $\OmegaA_{\pi_0}(j)$ count solutions to the inequalities
\[
\zeroplus \leq s_{0_1} \leq  \cdots \leq s_{0_{L_0}} \leq  j
\]
with $s_{0_l} < s_{0_{l+1}}$ if $l \in \Desa(\pi_0)$.
For $i=1,\ldots,k$, we let $\OmegaA_{P(\pi_i)}(j)$ count solutions to the inequalities
\[
j+1+(i-1)(2j+2) \leq s_{i_1} \leq  \cdots \leq s_{i_{L_i}} \leq  j+1+ (i-1)(2j+2) + 2j+1
\]
with $s_{i_l} < s_{i_{l+1}}$ if $l \in \Desa(\pi_i)$.
By concatenating these inequalities, we see that if we sum over all $I \subseteq [0,n-1]$ and all barred $\chainbipi$ posets with $k$ bars then $\sum_{I \subseteq [0,n-1]} \OmegaA_{\chainbipi}(j,k)$ is equal to the number of solutions to the inequalities
\[
\zeroplus \leq s_1 \leq \cdots \leq s_n \leq 2jk+j+2k
\]
with $s_i < s_{i+1}$ if $i \in \Desa(\pi)$.
Hence we conclude that
\[
\sum_{I \subseteq [0,n-1]} \OmegaA_{\chainbipi}(j,k)  = \binom{2jk+j+2k +n-\desa(\pi)}{n}.
\]

Next we prove equation~\eqref{eq: desb desa ideal}.
If we expand the right side of equation \eqref{eq: desb desa ideal} and compare the coefficient of $s^jt^k$ on both sides then we have
\begin{equation}\label{eq: jk desb desa ideal}
\binom{2jk+2j+k+n-\desa(\pi)}{n} = \sum_{\sigma\tau = \pi} \binom{j+n-\desb(\sigma)}{n} \binom{k+n-\desa(\tau)}{n}.
\end{equation}
Using Theorem \ref{thm: generic zig-zag order poly}, we see that
\[
\sum_{\sigma\tau = \pi} \binom{j+n-\desb(\sigma)}{n} \binom{k+n-\desa(\tau)}{n} = \sum_{I \subseteq [n-1]} \OmegaB_{\zigaipi} (j,k)
\]
and, using Theorem \ref{thm: generic zig-zag chain equivalence}, we see that
\[
\sum_{I \subseteq [n-1]} \OmegaB_{\zigaipi} (j,k) = \sum_{I \subseteq [n-1]} \OmegaB_{\chainaipi}(j,k).
\] 
Thus it remains to show that
\[
\sum_{I \subseteq [n-1]} \OmegaB_{\chainaipi}(j,k) = \binom{2jk+2j+k+n-\desa(\pi)}{n}.
\]

First we note that $\sum_{I \subseteq [n-1]} \OmegaB_{\chainaipi}(j,k)$ counts ordered pairs $(f,P)$ where $P$ is a barred $\chainaipi$ poset with $k$ bars for some $I \subseteq [n-1]$ and $f$ is a $\chainaipi$-partition of type~$B$ with parts less than or equal to $j$.
Fix a barred $\chainaipi$ poset with $k$ bars and use the bars to define compartments labeled $0,\ldots,k$ from left to right.
Then define $\pi_i$ to be the (possibly empty) subword of $\pi$ in compartment $i$ and denote the length of $\pi_i$ by $L_i$.
Then
\[
\OmegaB_{\chainaipi}(j) = \prod_{i=0}^{k} \OmegaB_{P(\pi_i)}(j) .
\]
For $i=0,\ldots,k$, we let $\OmegaB_{P(\pi_i)}(j)$ count solutions to the inequalities
\[
i(2j+1) \leq s_{i_1} \leq  \cdots \leq s_{i_{L_i}} \leq  i(2j+1) + 2j
\]
with $s_{i_l} < s_{i_{l+1}}$ if $l \in \Desa(\pi_i)$.
Note that a descent in position $0$ does not affect $\OmegaB_{P(\pi_{i})}(j)$ and so we use $\Desa$ instead of $\Desb$.
By concatenating these inequalities, we see that if we sum over all $I \subseteq [n-1]$ and all barred $\chainaipi$ posets with $k$ bars then $\sum_{I \subseteq [n-1]} \OmegaB_{\chainaipi}(j,k)$ is equal to the number of solutions to the inequalities
\[
0 \leq s_1 \leq \cdots \leq s_n \leq 2jk+2j+k
\]
with $s_i < s_{i+1}$ if $i \in \Desa(\pi)$.
Hence we conclude that
\[
\sum_{I \subseteq [n-1]} \OmegaB_{\chainaipi}(j,k)  = \binom{2jk+2j+k +n-\desa(\pi)}{n}.
\qedhere
\]
\end{proof}

Since equations \eqref{eq: jk desa desb ideal} and \eqref{eq: jk desb desa ideal} hold for all $j,k \geq 0$, we have the following theorem.

\begin{thm}\label{thm: desa desb order polys}
As polynomials in $x$ and $y$ with coefficients in the group algebra of $\hypn$,
\[
\phi_A(x/2)\phi_B((y-1)/2) = \phi_A(xy/2)
\]
and
\[
\phi_B((x-1)/2)\phi_A(y/2) = \phi_A(xy/2).
\]
\end{thm}

This theorem shows that the type $A$ Eulerian descent algebra is a two-sided ideal in the algebra spanned by the orthogonal idempotents $a_i$ and $b_i$ which, as we saw in Section \ref{sec: flag descent algebra}, is actually the flag descent algebra.
Next we examine the algebra spanned by the orthogonal idempotents $b_i$ and $\augversion{b}_i$.

\begin{thm}\label{thm: desb ades ideal}
For every $\pi \in \hypn$,
\begin{equation}\label{eq: desb ades ideal}
\sum_{j,k \geq 0} \binom{2jk + k + n - \ades(\pi)}{n} s^j t^k = \frac{\dsum_{\sigma\tau = \pi} s^{\desb(\sigma)} t^{\ades(\tau)}}{(1-s)^{n+1}(1-t)^{n+1}}
\end{equation}
and
\begin{equation}\label{eq: ades desb ideal}
\sum_{j,k \geq 0} \binom{2jk + j + n - \ades(\pi)}{n} s^j t^k = \frac{\dsum_{\sigma\tau = \pi} s^{\ades(\sigma)} t^{\desb(\tau)}}{(1-s)^{n+1}(1-t)^{n+1}}.
\end{equation}
\end{thm}

\begin{proof}
If we expand the right side of equation \eqref{eq: desb ades ideal} and compare the coefficient of $s^jt^k$ on both sides then we have
\begin{equation}\label{eq: jk desb ades ideal}
\binom{2jk+k+n-\ades(\pi)}{n} = \sum_{\sigma\tau = \pi} \binom{j+n-\desb(\sigma)}{n} \binom{k+n-\ades(\tau)}{n}.
\end{equation}
Using Theorem \ref{thm: generic zig-zag order poly}, we see that
\[
\sum_{\sigma\tau = \pi} \binom{j+n-\desb(\sigma)}{n} \binom{k+n-\ades(\tau)}{n} = \sum_{I \in Z(n)} \OmegaB_{\augzigipi} (j,k)
\]
and, using Theorem \ref{thm: generic zig-zag chain equivalence}, we see that
\[
\sum_{I \in Z(n)} \OmegaB_{\augzigipi} (j,k) = \sum_{I \in C(n)} \OmegaB_{\augchainipi}(j,k).
\]
Thus to prove equation~\eqref{eq: desb ades ideal} it remains to show that
\[
\sum_{I \in C(n)} \OmegaB_{\augchainipi}(j,k) = \binom{2jk+k+n-\ades(\pi)}{n}.
\]

First we note that $\sum_{I \in \augchainsets} \OmegaB_{\augchainipi}(j,k)$ counts ordered pairs $(f,P)$ where $P$ is a barred $\augchainipi$ poset with $k$ bars for some $I \in \augchainsets$ and $f$ is a $\augchainipi$-partition of type~$B$ with parts less than or equal to $j$.
Fix a barred $\augchainipi$ poset with $k$ bars and use the bars to define compartments labeled $0,\ldots,k$ from left to right.
Then define $\pi_i$ to be the (possibly empty) subword of $\pi$ in compartment $i$ and denote the length of $\pi_i$ by $L_i$.
Then
\[
\OmegaB_{\augchainipi}(j) = \OmegaB_{\pi_0}(j)\OmegaB_{\overleftarrow{\pi_k}}(j)\prod_{i=1}^{k-1} \OmegaB_{P(\pi_i)}(j) .
\]
First, we let $\OmegaB_{\pi_0}(j)$ count solutions to the inequalities
\[
0 = s_{0_0} \leq s_{0_1} \leq  \cdots \leq s_{0_{L_0}} \leq  j
\]
with $s_{0_l} < s_{0_{l+1}}$ if $l \in \Desb(\pi_0)$.
For $i=1,\ldots,k-1$, we let $\OmegaB_{P(\pi_i)}(j)$ count solutions to the inequalities
\[
j+1+(i-1)(2j+1) \leq s_{i_1} \leq  \cdots \leq s_{i_{L_i}} \leq  j+1+(i-1)(2j+1) + 2j
\]
with $s_{i_l} < s_{i_{l+1}}$ if $l \in \Desa(\pi_i)$.
Note that a descent in position $0$ does not affect $\OmegaB_{P(\pi_{i})}(j)$ and so we use $\Desa$ instead of $\Desb$.
Finally, we let $\OmegaB_{\overleftarrow{\pi_k}}(j)$ count solutions to the inequalities
\[
2jk-j+k \leq s_{k_1} \leq  \cdots \leq s_{k_{L_k}} \leq  2jk+k
\]
with $s_{k_l} < s_{k_{l+1}}$ if $l \in \Desa(\pi_i)$ and with $s_{k_{L_k}} < 2jk+k$ if $\pi(n) > 0$.
By concatenating these inequalities, we see that if we sum over all $I \in \augchainsets$ and all barred $\augchainipi$ posets with $k$ bars then $\sum_{I \in \augchainsets} \OmegaB_{\augchainipi}(j,k)$ is equal to the number of solutions to the inequalities
\[
0 = s_0 \leq s_1 \leq \cdots \leq s_n \leq s_{n+1} =  2jk+k
\]
with $s_i < s_{i+1}$ if $i \in \aDes(\pi)$.
Hence
\[
\sum_{I \in \augchainsets} \OmegaB_{\augchainipi}(j,k)  = \binom{2jk+k +n-\ades(\pi)}{n}.
\]

Next we prove equation~\eqref{eq: ades desb ideal}.
If we expand the right side of equation \eqref{eq: ades desb ideal} and compare the coefficient of $s^jt^k$ on both sides then we have
\begin{equation}\label{eq: jk ades desb ideal}
\binom{2jk+j+n-\ades(\pi)}{n} = \sum_{\sigma\tau = \pi} \binom{j+n-\ades(\sigma)}{n} \binom{k+n-\desb(\tau)}{n}.
\end{equation}
Using Theorem \ref{thm: generic zig-zag order poly}, we see that
\[
\sum_{\sigma\tau = \pi} \binom{j+n-\ades(\sigma)}{n} \binom{k+n-\desb(\tau)}{n} = \sum_{I \subseteq [0,n-1]} \Omegaaug_{\zigbipi} (j,k)
\]
and, using Theorem \ref{thm: generic zig-zag chain equivalence}, we see that
\[
\sum_{I \subseteq [0,n-1]} \Omegaaug_{\zigbipi} (j,k) = \sum_{I \subseteq [0,n-1]} \Omegaaug_{\chainbipi}(j,k).
\]
Thus it remains to show that
\[
\sum_{I \subseteq [0,n-1]} \Omegaaug_{\chainbipi}(j,k) = \binom{2jk+j+n-\ades(\pi)}{n}.
\]

First we note that $\sum_{I \subseteq [0,n-1]} \Omegaaug_{\chainbipi}(j,k)$ counts ordered pairs $(f,P)$ where $P$ is a barred $\chainbipi$ poset with $k$ bars for some $I \subseteq [0,n-1]$ and $f$ is an augmented $\chainbipi$-partition with parts less than or equal to $j$.
Fix a barred $\chainbipi$ poset with $k$ bars and use the bars to define compartments labeled $0,\ldots,k$ from left to right.
Then define $\pi_i$ to be the (possibly empty) subword of $\pi$ in compartment $i$ and denote the length of $\pi_i$ by $L_i$.
Then
\[
\Omegaaug_{\chainbipi}(j) = \Omegaaug_{\pi_0}(j)\prod_{i=1}^{k} \Omegaaug_{P(\pi_i)}(j) .
\]
First we let $\Omegaaug_{\pi_0}(j)$ count solutions to the inequalities
\[
0 = s_{0_0} \leq s_{0_1} \leq  \cdots \leq s_{0_{L_0}} \leq s_{0_{L_0+1}} =  j
\]
with $s_{0_l} < s_{0_{l+1}}$ if $l \in \aDes(\pi_0)$.
For $i=1,\ldots,k$, we let $\Omegaaug_{P(\pi_i)}(j)$ count solutions to the inequalities
\[
j+1+(i-1)(2j+1) = s_{i_0} \leq s_{i_1} \leq  \cdots \leq s_{i_{L_i}} \leq s_{i_{L_i+1}} =   j+1+ (i-1)(2j+1) + 2j
\]
with $s_{i_l} < s_{i_{l+1}}$ if $l \in \aDes(\pi_i)$.
Since we are using augmented $P$-partitions, we form equivalence classes by equating $j+i(2j+1)$ and $j+i(2j+1)+1$ for $i=0,\ldots,k-1$, effectively removing $k$ points from the set.
By concatenating these inequalities, we see that if we sum over all $I \subseteq [0,n-1]$ and all barred $\chainbipi$ posets with $k$ bars then $\sum_{I \subseteq [0,n-1]} \Omegaaug_{\chainbipi}(j,k)$ is equal to the number of solutions to the inequalities
\[
0 = s_0 \leq s_1 \leq \cdots \leq s_n \leq s_{n+1} = 2jk+j
\]
with $s_i < s_{i+1}$ if $i \in \aDes(\pi)$.
Hence we conclude that
\[
\sum_{I \subseteq [0,n-1]} \Omegaaug_{\chainbipi}(j,k)  = \binom{2jk+j +n-\ades(\pi)}{n}.  \qedhere
\]
\end{proof}

Since equations \eqref{eq: jk desb ades ideal} and \eqref{eq: jk ades desb ideal} hold for all $j,k \geq 0$, we have the following theorem.

\begin{thm}\label{thm: desb ades order polys}
As polynomials in $x$ and $y$ with coefficients in the group algebra of $\hypn$,
\[
\phi_B((x-1)/2)\psi(y/2) = \psi(xy/2)
\]
and
\[
\psi(x/2)\phi_B((y-1)/2) = \psi(xy/2).
\]
\end{thm}

We note that Petersen proved a one-sided version of this theorem in \cite[Theorem 1.5]{Petersen2005} and a two-sided version in \cite[Theorem 2.3.2]{PetersenThesis}.
This theorem shows that the augmented Eulerian descent algebra is a two-sided ideal in the algebra spanned by the orthogonal idempotents $b_i$ and $b^{(a)}_i$.
This result is mentioned by Cellini in \cite{Cellini1998}.
Taken together, Theorem \ref{thm: desb order poly}, Theorem \ref{thm: ades order poly}, and Theorem \ref{thm: desb ades order polys} imply Cellini's Theorem A in \cite{CelliniII1995}
where Cellini's $x_{2k+1}$ is the image of  $\phi_B(k)$ under the map that sends $\pi \mapsto \pi^{-1}$ and $x_{2k}$ is the image of  $\psi(k)$ under the same map.

Finally, we examine the algebra spanned by the orthogonal idempotents $a_i$ and $\augversion{b}_i$.
Here we have the somewhat unexpected behavior that both the type $A$ Eulerian descent algebra and the augmented Eulerian descent algebra are only one-sided ideals in the larger algebra.

\begin{thm}\label{thm: desa ades ideal}
For every $\pi \in \hypn$,
\begin{align}
\sum_{j,k \geq 0} \binom{2jk + 2k + n -1- \desa(\pi)}{n} s^j t^k &= \frac{\dsum_{\sigma\tau = \pi} s^{\desa(\sigma)} t^{\ades(\tau)}}{(1-s)^{n+1}(1-t)^{n+1}} \label{eq: desa ades ideal} \\
\intertext{and}
\sum_{j,k \geq 0} \binom{2jk + 2j + n - \ades(\pi)}{n} s^j t^k &= \frac{\dsum_{\sigma\tau = \pi} s^{\ades(\sigma)} t^{\desa(\tau)}}{(1-s)^{n+1}(1-t)^{n+1}} \label{eq: ades desa ideal} .
\end{align}
\end{thm}

\begin{proof}
If we expand the right side of equation \eqref{eq: desa ades ideal} and compare the coefficient of $s^jt^k$ on both sides then we have
\begin{equation}\label{eq: jk desa ades ideal}
\binom{2jk+2k+n-1-\desa(\pi)}{n} = \sum_{\sigma\tau = \pi} \binom{j+n-\desa(\sigma)}{n} \binom{k+n-\ades(\tau)}{n}.
\end{equation}
Using Theorem \ref{thm: generic zig-zag order poly}, we see that
\[
\sum_{\sigma\tau = \pi} \binom{j+n-\desa(\sigma)}{n} \binom{k+n-\ades(\tau)}{n} = \sum_{I \in Z(n)} \OmegaA_{\augzigipi} (j,k)
\]
and, using Theorem \ref{thm: generic zig-zag chain equivalence}, we see that
\[
\sum_{I \in Z(n)} \OmegaA_{\augzigipi} (j,k) = \sum_{I \in C(n)} \OmegaA_{\augchainipi}(j,k).
\]
Thus to prove equation~\eqref{eq: desa ades ideal} it remains to show that
\[
\sum_{I \in C(n)} \OmegaA_{\augchainipi}(j,k) = \binom{2jk+2k+n-1-\desa(\pi)}{n}.
\]

First we note that $\sum_{I \in \augchainsets} \OmegaA_{\augchainipi}(j,k)$ counts ordered pairs $(f,P)$ where $P$ is a barred $\augchainipi$ poset with $k$ bars for some $I \in \augchainsets$ and $f$ is a $\augchainipi$-partition of type~$A$ with parts less than or equal to $j$.
Fix a barred $\augchainipi$ poset with $k$ bars and use the bars to define compartments labeled $0,\ldots,k$ from left to right.
Then define $\pi_i$ to be the (possibly empty) subword of $\pi$ in compartment $i$ and denote the length of $\pi_i$ by $L_i$.
Then
\[
\OmegaA_{\augchainipi}(j) = \OmegaA_{\pi_0}(j)\OmegaA_{\overleftarrow{\pi_k}}(j)\prod_{i=1}^{k-1} \OmegaA_{P(\pi_i)}(j) .
\]
First we let $\OmegaA_{\pi_0}(j)$ count solutions to the inequalities
\[
0  \leq s_{0_1} \leq  \cdots \leq s_{0_{L_0}} \leq  j
\]
with $s_{0_l} < s_{0_{l+1}}$ if $l \in \Desa(\pi_0)$.
For $i=1,\ldots,k-1$, we let $\OmegaA_{P(\pi_i)}(j)$ count solutions to the inequalities
\[
j+1+(i-1)(2j+2) \leq s_{i_1} \leq  \cdots \leq s_{i_{L_i}} \leq  j+1+(i-1)(2j+2) + 2j+1
\]
with $s_{i_l} < s_{i_{l+1}}$ if $l \in \Desa(\pi_i)$.
Finally, we let $\OmegaA_{\overleftarrow{\pi_k}}(j)$ count solutions to the inequalities
\[
 j+1+(k-2)(2j+2) + 2j+2 \leq s_{k_1} \leq  \cdots \leq s_{k_{L_k}} \leq  2jk+2k-1
\]
with $s_{k_l} < s_{k_{l+1}}$ if $l \in \Desa(\pi_i)$.
By concatenating these inequalities, we see that if we sum over all $I \in \augchainsets$ and all barred $\augchainipi$ posets with $k$ bars then $\sum_{I \in \augchainsets} \OmegaA_{\augchainipi}(j,k)$ is equal to the number of solutions to the inequalities
\[
0 \leq s_1 \leq \cdots \leq s_n \leq   2jk+2k-1
\]
with $s_i < s_{i+1}$ if $i \in \Desa(\pi)$.
Hence
\[
\sum_{I \in \augchainsets} \OmegaA_{\augchainipi}(j,k)  = \binom{2jk+2k +n-1-\desa(\pi)}{n}.
\]

Next we prove equation~\eqref{eq: ades desa ideal}.
If we expand the right side of equation \eqref{eq: ades desa ideal} and compare the coefficient of $s^jt^k$ on both sides then we have
\begin{equation}\label{eq: jk ades desa ideal}
\binom{2jk+2j+n-\ades(\pi)}{n} = \sum_{\sigma\tau = \pi} \binom{j+n-\ades(\sigma)}{n} \binom{k+n-\desa(\tau)}{n}.
\end{equation}
Using Theorem \ref{thm: generic zig-zag order poly}, we see that
\[
\sum_{\sigma\tau = \pi} \binom{j+n-\ades(\sigma)}{n} \binom{k+n-\desa(\tau)}{n} = \sum_{I \subseteq [n-1]} \Omegaaug_{\zigaipi} (j,k)
\]
and, using Theorem \ref{thm: generic zig-zag chain equivalence}, we see that
\[
\sum_{I \subseteq [n-1]} \Omegaaug_{\zigaipi} (j,k) = \sum_{I \subseteq [n-1]} \Omegaaug_{\chainaipi}(j,k).
\]
Thus it remains to show that
\[
\sum_{I \subseteq [n-1]} \Omegaaug_{\chainaipi}(j,k) = \binom{2jk+2j+n-\ades(\pi)}{n}.
\]

First we note that $\sum_{I \subseteq [n-1]} \Omegaaug_{\chainaipi}(j,k)$ counts ordered pairs $(f,P)$ where $P$ is a barred $\chainaipi$ poset with $k$ bars for some $I \subseteq [n-1]$ and $f$ is an augmented $\chainaipi$-partition with parts less than or equal to $j$.
Fix a barred $\chainaipi$ poset with $k$ bars and use the bars to define compartments labeled $0,\ldots,k$ from left to right.
Then define $\pi_i$ to be the (possibly empty) subword of $\pi$ in compartment $i$ and denote the length of $\pi_i$ by $L_i$.
Then
\[
\Omegaaug_{\chainaipi}(j) = \prod_{i=0}^{k} \Omegaaug_{P(\pi_i)}(j) .
\]
For $i=0,\ldots,k$, we let $\Omegaaug_{P(\pi_i)}(j)$ count solutions to the inequalities
\[
i(2j+1) = s_{i_0} \leq s_{i_1} \leq  \cdots \leq s_{i_{L_i}} \leq s_{i_{L_i+1}} =  i(2j+1) + 2j
\]
with $s_{i_l} < s_{i_{l+1}}$ if $l \in \aDes(\pi_i)$.
Since we are using augmented $P$-partitions, we form equivalence classes by equating $2j+i(2j+1)$ and $2j+i(2j+1)+1$ for $i=0,\ldots,k-1$, effectively removing $k$ points from the set.
By concatenating these inequalities, we see that if we sum over all $I \subseteq [n-1]$ and all barred $\chainaipi$ posets with $k$ bars then $\sum_{I \subseteq [n-1]} \Omegaaug_{\chainaipi}(j,k)$ is equal to the number of solutions to the inequalities
\[
0 = s_0 \leq s_1 \leq \cdots \leq s_n \leq s_{n+1} = 2jk+2j
\]
with $s_i < s_{i+1}$ if $i \in \aDes(\pi)$.
Hence
\[
\sum_{I \subseteq [n-1]} \Omegaaug_{\chainaipi}(j,k)  = \binom{2jk+2j +n-\ades(\pi)}{n}.
\qedhere
\]
\end{proof}

Since equations \eqref{eq: jk desa ades ideal} and \eqref{eq: jk ades desa ideal} hold for all $j,k \geq 0$, we have the following theorem.

\begin{thm}\label{thm: desa ades order polys}
As polynomials in $x$ and $y$ with coefficients in the group algebra of $\hypn$,
\begin{align*}
\phi_A(x/2)\psi(y/2) &= \phi_A(xy/2) \\
\psi(x/2)\phi_A(y/2) &= \psi(xy/2).
\end{align*}
\end{thm}

This theorem shows that in the algebra spanned by the orthogonal idempotents $a_i$ and $b^{(a)}_i$, the type $A$ Eulerian descent algebra and the augmented Eulerian descent algebra are both right ideals.
It isn't too hard to find an example which shows that neither is a left ideal.



\chapter{Colored Permutation Groups}\label{chpt: Colored Permutation Groups}

In 1995, Mantaci and Reutenauer~\cite{MantaciReutenauer1995} proved the existence of an algebra whose basis elements are formal sums of ``colored permutations'' with the same associated colored compositions.
These colored compositions describe both the colors of the letters of a colored permutation and the descent set of each monochromatic run.
The Mantaci-Reutenauer algebra is well studied and fills much the same role as Solomon's descent algebra does in the theory of the group algebra of the symmetric group (see \cite{BaumannHohlweg2008, Poirier1998}).
While not all colored permutation groups are Coxeter groups, we can still extend some of the theory from previous chapters to study subalgebras of the corresponding group algebra.
The main results of this chapter are the proof of the existence of an Eulerian subalgebra called the ``colored Eulerian descent algebra'' and the description of analogues of the Eulerian idempotents in Sections \ref{sec: Eulerian descent algebra} and \ref{sec: type B Eulerian descent algebra}.
The chapter is organized as follows.

In Section~\ref{sec: colored permutation groups}, we formally define ``colored permutations'' which can be thought of as permutations from $\symn$ together with a choice of a ``color'' from a finite cyclic group for each letter in a given permutation.
We also define the descent set of a colored permutation, taking a definition equivalent to that originally used by Steingr{\'{\i}}msson in \cite{Steingrimsson1994}.
In Section~\ref{sec: colored des maj}, we define colored multiset permutations and count colored multiset permutations by descent number and major index, mirroring the results of Chow and Mansour~\cite{ChowMansour2011} in their study of the flag major index of colored permutations.
In Section~\ref{sec: colored fmaj}, we show how a variable substitution into formulas from Section~\ref{sec: colored des maj} produces multiset analogues of their formulas counting permutations by descent number and flag major index.
We note that the multiset distribution of flag major index alone was computed by Haglund, Loehr, and Remmel in \cite{HaglundLoehrRemmel2005}.

After finishing our sections on enumeration, we begin the study of the colored Eulerian descent algebra.
In Section~\ref{sec: colored posets and p-partitions}, we define colored posets and colored \linebreak $P$-partitions which are different from those defined by Hsiao and Petersen in \cite{HsiaoPetersen2010}.
They define colored posets to be posets with colored labels and by doing so obtain results about the Mantaci-Reutenauer algebra.
Our approach differs in the way we define the set of linear extensions of a given colored poset and allows us to count colored permutations by descent number.
Finally, we show in Section~\ref{sec: colored descent algebra} that descent number induces an algebra which is a subalgebra of the Mantaci-Reutenauer algebra and describe a set of colored Eulerian idempotents which span this colored Eulerian descent algebra.
The colored Eulerian idempotents reduce to the well-known Eulerian idempotents in the group algebra of $\symn$ when considering permutations of a single color.


\section{The colored permutation groups $\colpermrn$}\label{sec: colored permutation groups}

For every totally ordered set $X$,  we denote by $\rfoldx$ the set $[0,r-1] \times X$ ordered lexicographically.
Thus $(i,x) < (j,y)$ in $\rfoldx$ if $i<j$ or if $i=j$ and $x < y$.
For notational ease, we denote the subset $\{k\}\times X$ by $X_k$ and write $x_k$ in place of $(k,x)$ whenever doing so is unambiguous. 
For example, $[n]_{(r)} = \{i_j \, | \, 1 \leq i \leq n, 0 \leq j < r \}$ with total order $$1_0 < \cdots < n_0 < 1_1 < \cdots < n_1 < \cdots < 1_{r-1} < \cdots < n_{r-1}.$$
The group of \emph{$r$-colored permutations} $\colpermrn$ is the group of all bijections $\pi: [n]_{(r)} \rightarrow [n]_{(r)}$ such that $\pi(i_j) = k_l$ implies that $\pi(i_{j+a}) = k_{l+a}$ for $1 \leq a < r$ with the sums $j+a$ and $l+a$ evaluated modulo $r$.

All colored permutations are determined by $\pi(i_0)$ for $1 \leq i \leq n$ and so we write colored permutations in one-line notation as $\pi = \pi(1_0) \pi(2_0) \cdots \pi(n_0)$.
In practice, we often suppress this notation and write $\pi(i)$ in place of $\pi(i_0)$.
Every colored permutation is built from an underlying permutation in $\symn$ obtained from $\pi$ by ignoring the colors of the letters.
We denote this permutation by $|\pi|:=|\pi(1)|\cdots |\pi(n)|$ where $|x_k| := x$ for every $x_k \in \rfoldx$.
Define the color map $\epsilon$ by $\epsilon(x_k) = k$ for all $x_k \in \rfoldx$.
Lastly, for $i=0,\ldots,r-1$, let $N_i(\pi)$ be the number of letters of $\pi$ of color $i$.

\begin{ex}
The colored permutation $\pi = 2_0 1_3 3_1 5_2 4_2$ is in $G_{4,5}$. We have $|\pi| = 21354$, $\epsilon(5_2) = 2$, and $N_2(\pi) = 2$.
\end{ex}

The group operation in $\colpermrn$ is function composition.
In one-line notation, this says that if $\pi(i) = j_k$ and $\sigma(j) = l_p$ then $(\sigma \circ \pi)(i) = l_{k+p}$ where, as always, the sum is evaluated modulo $r$.
Thus colored permutation composition depends on $r$ and so even though $\pi = 2_0 1_3 3_1 5_2 4_2$ and $\sigma = 3_1 1_1 5_0 2_1 4_3$ are colored permutations in both $G_{4,5}$ and $G_{5,5}$, $\sigma \pi$ as a permutation in $G_{4,5}$ is different from $\sigma \pi$ as a permutation in~$G_{5,5}$.
As a group, $\colpermrn$ is isomorphic to the wreath product $\Z_r \wr \symn$.

\begin{ex}
If $\pi = 2_0 1_3 3_1 5_2 4_2$ and $\sigma = 3_1 1_1 5_0 2_1 4_3$ are colored permutations in $G_{4,5}$ then $\sigma \pi = 1_1 3_0 5_1 4_1 2_3$.
\end{ex}

Define the descent set $\Des(\pi)$ of a colored permutation $\pi$ to be the set of all $i \in [n]$ such that $\pi(i) > \pi(i+1)$ with respect to the total order on $[0,n]_{(r)}$ and with $\pi(n+1) := 0_1$.
Thus $n \in \Des(\pi)$ if and only if $\epsilon(\pi(n)) \neq 0$.
As before, define $\des(\pi) = |\Des(\pi)|$ and $\maj(\pi) = \sum_{i \in \Des(\pi)} i$.
In preparation for later results, we also define the internal descent set $\intDes(\pi) = \Des(\pi) \cap [n-1]$ and the internal descent number $\intdes(\pi) = |\intDes(\pi)|$.

\begin{ex}
If $\pi = 3_1 1_1 4_0 2_3$ then $\Des(\pi) = \{1,2,4\}$ with $\des(\pi) = 3$ and $\maj(\pi) = 7$.
Also, $\intDes(\pi) = \{1,2\}$ and $\intdes(\pi) = 2$.
\end{ex}

Our definition of the descent set of a colored permutation is equivalent to that of Steingr{\'{\i}}msson \cite[Definition 2]{Steingrimsson1994}.
We could, however, have defined descents by setting $\pi(n+1) = 0_j$ and $\pi(0) = 0_k$ for arbitrary $j$ and $k$.
Our methods allow for the enumeration of colored permutations by these various descent numbers but we do not consider such results here because, as we will see later on, only Steingr{\'{\i}}msson's definition induces an algebra.


\section{Descent number and major index}\label{sec: colored des maj}

In this section, we count colored permutations by descent number and major index.
Since the barred permutation technique used here works equally well for multiset permutations, we first count ``colored multiset permutations'' and then reduce to $\colpermrn$.
Following \cite{HaglundLoehrRemmel2005}, we define a \emph{colored multiset permutation} to be a multiset permutation from $\multn$ together with a choice of a ``color'' from $\Z_r$ for each letter in the permutation.
If $\alpha = (\alpha_1,\ldots,\alpha_m)$ is a weak composition of $n$ then we define $\multrn$ to be the set of all $r$-colored permutations of the multiset $\{1^{\alpha_1},\ldots, m^{\alpha_m}\}$.
The descent set $\Des(\pi)$ of a colored multiset permutation $\pi$ is the set of all $i \in [n]$ such that $\pi(i) > \pi(i+1)$ with $\pi(n+1) = 0_1$.
We let
\[
M_{\alpha}^r (t,q) = \sum_{\pi \in \multrn} t^{\des(\pi)} q^{\maj(\pi)}
\]
and
\[
M_{\alpha}^r (t,q,z_0,\ldots,z_{r-1}) = \sum_{\pi \in \multrn} t^{\des(\pi)} q^{\maj(\pi)} z_0^{N_0(\pi)} \cdots z_{r-1}^{N_{r-1}(\pi)}.
\]
Next we define a \emph{barred colored multiset permutation} to be a shuffle of a colored multiset permutation with a sequence of bars such that between any two bars the letters of the colored multiset permutation are weakly increasing and such that only letters of color 0 are allowed in the rightmost compartment.
We prove the following theorem by counting weighted barred colored multiset permutations in two different ways.

\begin{thm}\label{thm: col mult perm des maj}
If $\alpha = (\alpha_1,\ldots,\alpha_m)$ and $|\alpha| = n$ then
\begin{equation}\label{eq: col mult perm des maj}
\sum_{j \geq 0} t^j \prod_{i = 1}^m \left( \sum_{L \vDash \alpha_i} \qbinom{j + L_0}{L_0} z_0^{L_0} \prod_{k=1}^{r-1} \qbinom{j + L_k -1}{L_k} q^{L_k}z_k^{L_k}  \right)  = \frac{M_{\alpha}^r (t,q,z_0,\ldots,z_{r-1})}{(t;q)_{n+1}}
\end{equation}
where the inner sum is over all weak compositions $L$ of $\alpha_i$ with $r$ parts.
\end{thm}

\begin{proof}
If we begin with a colored multiset permutation $\pi \in \multrn$ then we must first place a bar in each descent.
We weight each bar in position $i$ by $tq^i$ and each letter of color $i$ by $z_i$.
Define the weight of a barred colored multiset permutation to be the product of the weights of the bars and the weights of the letters.
The first bars placed in descents contribute a factor of $t^{\des(\pi)}q^{\maj(\pi)}$.
We are now free to place bars in any of the $n+1$ positions and so $\pi$ contributes
\[
\frac{t^{\des(\pi)}q^{\maj(\pi)} z_0^{N_0(\pi)} \cdots z_{r-1}^{N_{r-1}(\pi)}}{(t;q)_{n+1}}
\]
to the sum of the weights of all barred colored multiset permutations.
Summing over all $\pi \in \multrn$, we have
\[
\frac{\dsum_{\pi \in \multrn} t^{\des(\pi)} q^{\maj(\pi)} z_0^{N_0(\pi)} \cdots z_{r-1}^{N_{r-1}(\pi)}}{(t;q)_{n+1}} = \frac{M_{\alpha}^r (t,q,z_0,\ldots,z_{r-1})}{(t;q)_{n+1}}.
\]

Next we count the same barred colored multiset permutations but instead start with $j$ bars and thus have $j+1$ compartments.
Recall that the compartments are labeled $0,1,\ldots, j$ from right to left.
As before, rather than weighting each bar in position $i$ by $tq^i$, we instead weight each bar by $t$ and each letter of color $k$ in compartment $l$ by $q^l z_k$.
The weight of a barred colored multiset permutation is still the product of the weights of each of the bars and each of the letters of $\pi$ and the two different weighting methods agree for every barred colored multiset permutation.

Define $\Omega_\alpha (j,q,z_0,\ldots,z_{r-1})$ to be the sum of the products of the weights of the letters over all barred colored multiset permutations with $j$ bars.
To create a barred colored multiset permutation given $j$ initial bars, we must choose a color and a compartment for each element of $\{1^{\alpha_1}, \ldots, m^{\alpha_m}\}$.
Given such a choice, there is a unique way to place all of the colored letters within a compartment in weakly increasing order (up to permutation of identical letters).
This implies that we can make the necessary choices for each collection of $\alpha_i$ $i$'s independently of the others and so $\Omega_\alpha (j,q,z_0,\ldots,z_{r-1})$ factors as
\[
\Omega_\alpha (j,q,z_0,\ldots,z_{r-1}) = \prod_{i=1}^m \Omega_{(\alpha_i)} (j,q,z_0,\ldots,z_{r-1}).
\]

Consider the placement and color selection of the $\alpha_i$ $i$'s.
First we assign colors to each $i$.
Each choice of how many $i$'s are color $0$, how many are color $1$, and so on can be represented by a weak composition $L = (L_0, \ldots, L_{r-1})$ of $\alpha_i$ with $L_k$ equal to the number of $i$'s colored $k$.
The same argument that there is a unique way to place letters within a compartment in weakly increasing order shows that
\[
\Omega_{(\alpha_i)} (j,q,z_0,\ldots,z_{r-1}) = \sum_{L \vDash \alpha_i} \prod_{k=0}^{r-1} \Omega_{(L_k)} (j,q,z_0,\ldots,z_{r-1})
\]
where the sum is over all weak compositions $L$ of $\alpha_i$ with $r$ parts.

Letters assigned color $0$ can be placed in any of the $j+1$ compartments and the sum of the products of their possible weights is
\[
\Omega_{(L_0)} (j,q,z_0,\ldots,z_{r-1}) = \qbinom{j+L_0}{L_0} z_0^{L_0}.
\]
Letters assigned color $k\neq 0$ can be placed in any compartment except the rightmost compartment and the sum of the products of their possible weights is
\[
\Omega_{(L_k)} (j,q,z_0,\ldots,z_{r-1}) = \qbinom{j+L_k -1}{L_k} q^{L_k} z_k^{L_k}.
\]
Hence we have
\[
\Omega_{(\alpha_i)} (j,q,z_0,\ldots,z_{r-1}) =  \sum_{L \vDash \alpha_i} \qbinom{j + L_0}{L_0} z_0^{L_0} \prod_{k=1}^{r-1} \qbinom{j + L_k -1}{L_k} q^{L_k} z_k^{L_k}
\]
and thus the sum of the products of the weights of the letters of all barred colored multiset permutations with $j$ bars is
\[
\Omega_\alpha (j,q,z_0,\ldots,z_{r-1}) = \prod_{i = 1}^m \left( \sum_{L \vDash \alpha_i} \qbinom{j + L_0}{L_0} z_0^{L_0} \prod_{k=1}^{r-1} \qbinom{j + L_k -1}{L_k} q^{L_k} z_k^{L_k} \right).
\]
Finally, multiplying by $t^j$ and summing over $j$ completes the proof. \\
\end{proof}

Equation \eqref{eq: col mult perm des maj} is not aesthetically pleasing since $\Omega_{(\alpha_i)} (j,q,z_0,\ldots,z_{r-1})$ is a sum which cannot be evaluated and so we now consider several ways to reduce this equation to something simpler.
One option is to define a new statistic which allows us to evaluate the sum $\Omega_{(\alpha_i)} (j,q,z_0,\ldots,z_{r-1})$.
This is the approach we take in the next section.
Another option is to set $q=1$ and $z_i = 1$ for $i=0,\ldots,r-1$ and thus count colored multiset permutations by descents alone.
This reduces $q$-binomial coefficients to binomial coefficients and shows that
\begin{align*}
\frac{\dsum_{\pi \in \multrn} t^{\des(\pi)}}{(1-t)^{n+1}} &= \sum_{j \geq 0} t^j \prod_{i = 1}^m \left( \sum_{L \vDash \alpha_i} \binom{j + L_0}{L_0} \prod_{k=1}^{r-1} \binom{j + L_k -1}{L_k}  \right) \\
&= \sum_{j \geq 0} t^j \prod_{i = 1}^m \binom{rj + \alpha_i}{\alpha_i} .
\end{align*}
Lastly, we assume that $\alpha = (1,\ldots,1)$ so that $\multrn = \colpermrn$.
We define the colored $q$-Eulerian polynomials $C_{r,n}(t,q)$ by
\[
C_{r,n} (t,q) = \dsum_{\pi \in \colpermrn} t^{\des(\pi)} q^{\maj(\pi)}.
\]
Following the proof of the previous theorem, we see that if $\alpha = (1,\ldots, 1)$ then
\[
\Omega_{(1)} (j,q,z_0,\ldots,z_{r-1}) = [j+1]_q z_0 + (z_1 + \cdots + z_{r-1}) q [j]_q.
\]
If we set $z_0 = \cdots = z_{r-1} = 1$ then
\[
\Omega_{(1)} (j,q) =  [j+1]_q + (r-1)q[j]_q = 1+ rq[j]_q
\]
and we have the following corollary.

\begin{cor}\label{col perm des maj}
\begin{equation}\label{eq: col perm des maj}
\sum_{j \geq 0} (1+rq[j]_q)^n t^j = \frac{C_{r,n}(t,q)}{(t;q)_{n+1}}
\end{equation}
\end{cor}

We can say more about these colored $q$-Eulerian polynomials.
Write
\[
C_{r,n}(t,q) = \sum_{k=0}^n C_{r,n,k}(q) t^k
\]
where $C_{r,n,k}(q) = \sum_{\des(\pi)=k} q^{\maj(\pi)}$.
Expanding equation \eqref{eq: col perm des maj} and comparing coefficients of $t^j$ gives the following corollary.

\begin{cor}\label{cor: q-Worpitzky}
\[
(1+rq[j]_q)^n = \sum_{k=0}^n C_{r,n,k}(q) \qbinom{j+n-k}{n}
\]
\end{cor}

Thus the $C_{r,n,k}(q)$ satisfy a $q$-Worpitzky identity.
Also, following the insertion proof described in \cite[Theorem 9(i)]{ChowMansour2011}, we have the following theorem.
The proof is straightforward but tedious and is omitted.

\begin{thm}\label{thm: recurrence}
Let $C_{r,0,0}(q) = 1$ and note that $C_{r,n,0}(q) = 1$ for $n \geq 1$. The colored $q$-Eulerian polynomials satisfy the following recurrence for $n,k \geq 1$.
\[
C_{r,n,k}(q) = (1+rq[k]_q) C_{r,n-1,k}(q) + (rq^k[n-k]_q + (r-1)q^n) C_{r,n-1,k-1}(q)
\]
\end{thm}

Next define the colored Eulerian polynomials $C_{r,n}(t)$ by
\[
C_{r,n}(t) = \dsum_{\pi \in \colpermrn} t^{\des(\pi)}.
\]
Setting $q=1$ in Corollary \ref{col perm des maj}, we have the following corollary, originally proven by Steingr{\'{\i}}msson \cite[Theorem 17]{Steingrimsson1994}.

\begin{cor}[Steingr{\'{\i}}msson]\label{Col Eul Poly}
\[
\sum_{j \geq 0} (rj+1)^n t^j = \frac{C_{r,n}(t)}{(1-t)^{n+1}}
\]
\end{cor}

When $r=1$, this reduces to the familiar formula for the Eulerian polynomials and when $r=2$ we have a formula identical to the one used to compute the signed Eulerian polynomials (even though here descents are defined differently and a different total order is used).
When $r=3$, the first few $r$-colored Eulerian polynomials are shown in Table \ref{tab: three colored eulerian polys}.
When $r=4$, the first few $r$-colored Eulerian polynomials are shown in Table \ref{tab: four colored eulerian polys}.

While not symmetric like the Eulerian polynomials, they are in fact unimodal \cite[Theorem 19]{Steingrimsson1994}.
Chow and Mansour \cite{ChowMansour2011} conjecture a similar property for the colored $q$-Eulerian polynomials.
The following proposition gives a recurrence relation for the colored Eulerian polynomials which can easily be derived from \cite[Theorem 9(ii)]{ChowMansour2011} by setting $q=1$.

\begin{table}[htdp]
\begin{center}
\begin{tabular}{ccl} \hline $n$ & & $C_{3,n}(t)$ \\ \hline $1$ & & $1+2t$ \\ $2$ & & $1+13t + 4t^2$ \\ $3$ & & $1+60t + 93t^2 + 8t^3$ \\ $4$ & & $1+251t + 1131t^2 + 545t^3 + 16t^4$ \\ \hline
\end{tabular}
\caption{$C_{3,n}(t)$ for $n=1,2,3,4$\label{tab: three colored eulerian polys}}
\end{center}
\end{table}

\begin{table}[htdp]
\begin{center}
\begin{tabular}{ccl} \hline $n$ & & $C_{4,n}(t)$ \\ \hline $1$ & & $1+3t$ \\ $2$ & & $1+22t + 9t^2$ \\ $3$ & & $1+121t + 235t^2 + 27t^3$ \\ $4$ & & $1+620t + 3446t^2 + 1996 t^3 + 81t^4$ \\ \hline
\end{tabular}
\caption{$C_{4,n}(t)$ for $n=1,2,3,4$\label{tab: four colored eulerian polys}}
\end{center}
\end{table}

\begin{prop}\label{Col Eul Poly Recurrence}
If we set $C_{r,0}(t) = 1$ then the $r$-colored Eulerian polynomials satisfy the following recurrence relation for $r,n \geq 1$.
\[
C_{r,n}(t) = (1+(rn-1)t)C_{r,n-1}(t) + rt(1-t)C'_{r,n-1}(t)
\]
\end{prop}

\begin{proof}
Corollary \ref{Col Eul Poly} tells us that
\[
\sum_{j \geq 0} (rj+1)^{n-1} t^j = \frac{C_{r,n-1}(t)}{(1-t)^{n}}.
\]
Differentiating both sides with respect to $t$ and then multiplying by $rt$ gives us
\[
\sum_{j \geq 0} rj(rj+1)^{n-1} t^j = rnt\frac{C_{r,n-1}(t)}{(1-t)^{n+1}}+rt\frac{C'_{r,n-1}(t)}{(1-t)^{n}}.
\]
Adding $\sum_{j \geq 0} (rj+1)^{n-1} t^j$ to the left side and adding $C_{r,n-1}(t) / (1-t)^n$ to the right side of the previous equation we have
\[
\sum_{j \geq 0} (rj+1)^{n} t^j = rnt\frac{C_{r,n-1}(t)}{(1-t)^{n+1}}+\frac{C_{r,n-1}(t)}{(1-t)^{n}}+rt\frac{C'_{r,n-1}(t)}{(1-t)^{n}}.
\]
The left side of the equation is equal to $C_{r,n}(t) / (1-t)^{n+1}$ and so multiplying both sides of the equation by $(1-t)^{n+1}$ gives us
\begin{align*}
C_{r,n}(t) &= rntC_{r,n-1}(t)+(1-t)C_{r,n-1}(t)+rt(1-t)C'_{r,n-1}(t) \\
&= (1+(rn-1)t) C_{r,n-1}(t)+rt(1-t)C'_{r,n-1}(t). 
\qedhere
\end{align*}
\end{proof}

Next let us denote by $C_{r,n,k}$ the number of colored permutations in $\colpermrn$ with $k$ descents.
Then
\[
C_{r,n}(t) = \sum_{k=0}^n C_{r,n,k} t^k.
\]
Either by comparing coefficients of $t^k$ on both sides of the recurrence in the previous proposition or by setting $q=1$ in Theorem \ref{thm: recurrence}, we arrive at the following corollary originally proven by Steingr{\'{\i}}msson \cite[Lemma 16]{Steingrimsson1994}.

\begin{cor}[Steingr{\'{\i}}msson]
We have $C_{r,n,0} = 1$ for $n \geq 0$. The $C_{r,n,k}$ satisfy the recurrence relation
\[
C_{r,n,k} = (rk+1) C_{r,n-1,k} + (r(n+1)-(rk+1))C_{r,n-1,k-1}
\]
for $k \geq 1$.
\end{cor}


\section{Flag major index}\label{sec: colored fmaj}

In Section \ref{sec: signed multiset permutations}, we noted that the flag major index gives a generalization of the Carlitz identity to $\hypn$.
We now extend this generalization to $\colpermrn$.
We define the \emph{flag major index} of a colored permutation $\pi \in \colpermrn$, denoted here by $\fmaj$, by
\[
\fmaj(\pi) = r\cdot \maj(\pi) - \sum_{i=1}^n \epsilon(\pi(i)).
\]
This definition is equivalent to that of Chow and Mansour in \cite{ChowMansour2011}, although they define $\fmaj$ differently and our definition is actually Theorem 5 in their paper.
It turns out that the barred colored permutation technique used to prove Theorem \ref{thm: col mult perm des maj} allows us to compute the joint distribution $(\des, \fmaj)$ for colored multiset permutations.

\begin{thm}
\begin{equation}\label{eq: colored mult Carlitz}
\sum_{j \geq 0} t^j \prod_{i=1}^m \qbinom{rj+\alpha_i}{\alpha_i} = \frac{\dsum_{\pi \in \multrn} t^{\des(\pi)} q^{\fmaj(\pi)}}{(t;q^r)_{n+1}}
\end{equation}
\end{thm}

\begin{proof}
The proof follows from substituting $q\leftarrow q^r$ and $z_k \leftarrow q^{-k}$ for $k = 0,1,\ldots, r-1$ into equation \eqref{eq: col mult perm des maj}.
It is clear that such a substitution yields the right side of equation \eqref{eq: colored mult Carlitz}.
We must, however, show that such a substitution allows us to evaluate the sum $\Omega_{(\alpha_i)} (j,q^r,1,q^{-1},\ldots,q^{-(r-1)})$.
First note that $\Omega_{(\alpha_i)} (j,q,z_0,\ldots,z_{r-1})$ is the coefficient of $t^{\alpha_{i}}$ in
\[
\frac{1}{(tz_0,q)_{j+1} \prod_{k=1}^{r-1} (tqz_k;q)_j}.
\]
If we substitute $q \leftarrow q^r$ and $z_k \leftarrow q^{-k}$ for $k=0,1,\ldots,r-1$ then this becomes
\[
\frac{1}{(t,q^r)_{j+1} \prod_{k=1}^{r-1} (tq^{r-k};q^r)_j} = \frac{1}{(t;q)_{rj+1}}.
\]
Hence
\[
\Omega_{(\alpha_i)} (j,q^r,1,q^{-1},\ldots,q^{-(r-1)}) = \qbinom{rj+\alpha_i}{\alpha_i}.
\qedhere
\]
\end{proof}

If we set $\alpha = (1,\ldots,1)$ then we have the following corollary originally proven by Chow and Mansour \cite[Theorem 9(iv)]{ChowMansour2011}.

\begin{cor}[Chow and Mansour]
\begin{equation}\label{eq: colored Carlitz}
\sum_{j \geq 0} [rj+1]_q^n t^j = \frac{\dsum_{\pi \in \colpermrn} t^{\des(\pi)} q^{\fmaj(\pi)}}{(t;q^r)_{n+1}}
\end{equation}
\end{cor}

The corollary shows that the polynomials $G_{r,n}(t,q) := \sum_{\pi \in \colpermrn} t^{\des(\pi)} q^{\fmaj(\pi)}$ satisfy a generalized Carlitz identity.
It should be noted that the flag major index used here is different from the flag major index originally defined by Adin and Roichman~\cite{AdinRoichman2001} and used by Bagno and Biagioli~\cite{BagnoBiagioli2007} in an alternate generalization of the Carlitz identity to $\colpermrn$.
The same techniques used here (modified slightly) could also be used to count colored permutations by this alternate flag major index.


\section{Colored posets and colored $P$-partitions}\label{sec: colored posets and p-partitions}

In \cite{ChowThesis2001}, Chow defined signed posets and signed $P$-partitions to capture the group structure of $\hypn$.
The obvious extension of this definition to $\colpermrn$ is to define colored $P$-partitions as maps $f:P \rightarrow \N_{(r)}$ such that if $x<_P y$ then $f(x) \leq f(y)$ in $\N_{(r)}$ with the added restrictions that if $x<_P y$ and $x>y$ in $[n]_{(r)}$ then $f(x) < f(y)$ in $\N_{(r)}$ and if $f(x_j) = k_l$ then $f(x_{j+a}) = k_{l+a}$ for $a=1,\ldots,r-1$.
However, this approach does not work because these conditions may be contradictory.
As colors cycle through $\Z_r$, the relationship between $x$ and $y$ can change.
For example, if $\pi = 2_0 1_1$ is a colored permutation in $G_{3,2}$ and we define a poset $P(\pi)$ by $2_0 <_P 1_1$ then we would allow $f(2_0) = f(1_1)$ but this would imply that $f(2_2) = f(1_0)$ which would not be allowed since $2_2 > 1_0$.
This shifting relationship is the biggest hurdle to overcome in extending previous definitions to colored permutation groups.

Despite these limitations we can say enough to prove the existence of an algebra induced by descent number.
For every $\pi \in \colpermrn$, let $\shufflezeropi$ denote the set of shuffles of $\pi$ with the word $\vec{0} = 0_10_2\cdots 0_{r-1}$.
We define the set of \emph{anchored $r$-colored permutations of length $n$}, denoted by $\anchoredcolpermrn$, by
\[
\anchoredcolpermrn = \coprod_{\pi \in \colpermrn} \shufflezeropi.
\]
Let $\pos$ denote the set of positive integers.
We define an \emph{$r$-colored poset} $P$ to be a partially ordered finite subset of $\pos_{(r)} \cup \{0_1,\ldots,0_{r-1}\}$ such that the elements from $\pos_{(r)}$ all have distinct absolute values and such that $0_1 <_P 0_2 <_P \cdots <_P 0_{r-1}$.
Figure \ref{fig: 4-colored poset} gives an example of an $r$-colored poset.

\begin{figure}[htbp]
\[\xymatrix{ & & & 0_3 \\ & & 3_1\ar@{-}[ur] & \\  & & 1_0\ar@{-}[u] & \\  & 0_2\ar@{-}[ur] & & 2_1\ar@{-}[ul]  \\  0_1\ar@{-}[ur] & & &  }\] 
\caption{\; A $4$-colored poset. \label{fig: 4-colored poset}}
\end{figure}
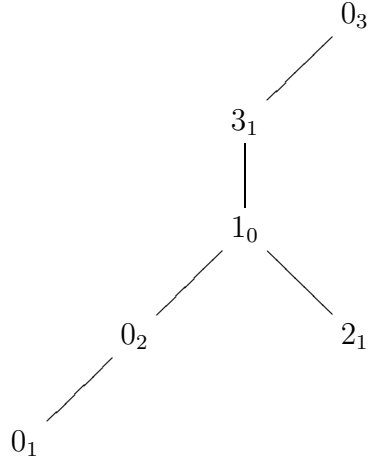

We define colored linear extensions of these colored posets in two stages.
A word $w\in \anchoredcolpermrn$ is compatible with a poset $P$ if whenever $x <_P y$ then $x$ comes before $y$ in $w$.
We define the set of linear extensions of $P$, denoted by $\Lin(P)$, to be the set of words from $\anchoredcolpermrn$ compatible with $P$.
Next we know that each $w\in \anchoredcolpermrn$ is a shuffle of a permutation $\pi \in \colpermrn$ with the word $0_1\cdots 0_{r-1}$.
The $0_i$ letters split the permutation $\pi$ into $r$ (possibly empty) subwords.
Let $\pi_0$ denote the subword to the left of $0_1$ in~$w$ and let $\pi_{r-1}$ denote the subword obtained by subtracting $r-1$ from the color of each letter to the right of $0_{r-1}$ in $w$.
Then for $i=1,\ldots,r-2$ define $\pi_i$ to be the subword between $0_i$ and $0_{i+1}$ in $w$ with $i$ subtracted from each color.
We then define the colored linear extensions of $w \in \anchoredcolpermrn$ by $\collin(w) := \shuffle{\pi_0,\ldots,\pi_{r-1}}.$
Finally, we define the \emph{colored linear extensions} of a poset $P$ by
\[
\collin(P) = \coprod_{w \in \Lin(P)} \shuffle{\pi_0,\ldots,\pi_{r-1}}.
\]

\begin{ex}
Let $P$ be the colored poset in Figure \ref{fig: 4-colored poset}.
Then
\[
\Lin(P) = \{0_1 0_2 2_1 1_0 3_1 0_3, 0_1 2_1 0_2 1_0 3_1 0_3, 2_1 0_1 0_2 1_0 3_1 0_3\}
\]
and
\begin{align*}
\collin(P) &= \{2_3 1_2 3_3\} \cup \shuffle{2_0,1_23_3} \cup \shuffle{2_1, 1_2 3_3} \\
&= \{1_22_03_3, 1_22_13_3, 1_23_32_0, 1_23_32_1, 2_01_23_3, 2_11_23_3, 2_31_23_3\}.
\end{align*}
\end{ex}

We are now ready to define colored $P$-partitions.
Let $X = \{x_0,x_1,\ldots,x_{\infty}\}$ be a totally ordered countable set with a maximal element and total order
\[
x_0 < x_1 < \cdots < x_{\infty}
\]
and let $P$ be an $r$-colored poset.
A \emph{colored $P$-partition} is a function $f: P \rightarrow \rfoldx$ such that:
\begin{enumerate}
\item[(i)] $f(0_k) = (k,x_{0})$ for every $0_k \in P$
\item[(ii)] $f(a) \leq f(b)$ if $a <_P b$
\item[(iii)] if $f(a),f(b) \in X_k$ then $f(a) < f(b)$ if $a <_P b$ and  $|a|_{\epsilon(a)-k} > |b|_{\epsilon(b)-k}$
\item[(iv)] if $f(a) = (k,x_{\infty})$ then $\epsilon(a) = k$.
\end{enumerate}

Condition $\mathrm{(iii})$ says that when $f(a)$ and $f(b)$ are both mapped into the same color block $X_k$ then the inequality between them is weak or strict as determined by their shifted relationship, not their relationship in the one-line notation of $\pi$.
Let $\mathcal{A}(P)$ denote the set of all colored $P$-partitions.
We can represent every colored permutation $\pi \in \colpermrn$ by the colored poset
\[
\pi(1) <_P \pi(2) <_P \cdots <_P \pi(n) <_P 0_1 <_P \cdots <_P 0_{r-1}.
\]
For notational ease, we call these colored $P$-partitions simply colored $\pi$-partitions and denote the set of all such colored $\pi$-partitions by $\A(\pi)$.
Similarly, if $w\in \anchoredcolpermrn$ then we denote the set of all colored $w$-partitions by $\A(w)$ with $w$ represented by the poset
\[
w(1) <_P w(2) <_P \cdots <_P w(n+r-1).
\]
We have the following theorem which resembles the fundamental theorems from previous chapters and so is called the \emph{Fundamental Theorem of Colored $P$-partitions}.

\begin{thm}[FTCPP]
There is a bijection between the set of all colored $P$-partitions $\A(P)$ and the disjoint union of the set of colored $\pi$-partitions over all colored linear extensions $\pi$ of $P$:
\[
\A(P) \leftrightarrow \coprod_{\pi \in \collin(P)} \A(\pi).
\]
\end{thm}

\begin{proof}
We first prove that
\[
\A(P) = \coprod_{w \in \Lin(P)} \A(w).
\]
The proof is by induction on the number of incomparable pairs $(a, b)$.
If $a$ and $b$ are incomparable in $P$ then let $P_{ab}$ denote the poset obtained from $P$ by adding the relation $a <_P b$ and similarly let $P_{ba}$ denote the poset obtained from $P$ by adding the relation $b <_P a$.
Then $\Lin(P) = \Lin(P_{ab}) \sqcup \Lin(P_{ba})$.
If we assume that $|a| \neq 0$ and $b = 0_k$ for some $k\in [0,r-1]$ then it is clear that $\Lin(P) = \Lin(P_{ab}) \sqcup \Lin(P_{ba})$ implies that $\mathcal{A}(P) = \mathcal{A}(P_{ab}) \sqcup \mathcal{A}(P_{ba})$ since $\A(P_{ab})$ contains all colored $P$-partitions $f$ such that $f(a) \in X_l$ for $l < k$ and $\A(P_{ba})$ contains all colored $P$-partitions $f$ such that $f(a) \in X_l$ for $l \geq k$.
Thus, by induction, we can assume that the set $\{a, 0_1, \ldots, 0_{r-1}\}$ is totally ordered for every $a \in P$ with $|a| \neq 0$.

Next we consider incomparable pairs $(a, b)$ with $|a|\neq 0 $ and $|b|\neq 0$.
By the previous argument, we can assume that $f(a), f(b) \in X_k$ for some $k\in [0,r-1]$.
We again let $P_{ab}$ denote the poset obtained from $P$ by adding the relation $a <_P b$ and similarly let $P_{ba}$ denote the poset obtained from $P$ by adding the relation $b <_P a$.
As before, $\Lin(P) = \Lin(P_{ab}) \sqcup \Lin(P_{ba})$ and so $\mathcal{A}(P) = \mathcal{A}(P_{ab}) \sqcup \mathcal{A}(P_{ba})$, which follows from condition $(\mathrm{iii})$.
More explicitly, if $|a|_{\epsilon(a)-k} < |b|_{\epsilon(b)-k}$ then $\A(P_{ab})$ contains all colored $P$-partitions $f$ such that $f(a) \leq f(b)$ and $\A(P_{ba})$ contains all colored $P$-partitions $f$ such that $f(b) < f(a)$.
Similarly, if $|a|_{\epsilon(a)-k} > |b|_{\epsilon(b)-k}$ then $\A(P_{ab})$ contains all colored $P$-partitions $f$ such that $f(a) < f(b)$ and $\A(P_{ba})$ contains all colored $P$-partitions $f$ such that $f(b) \leq f(a)$.
Note that condition $\mathrm{(iii})$ also ensures that if $f(a) = (k,x_{\infty})$ then $f$ is in exactly one of $\mathcal{A}(P_{ab})$ and $\mathcal{A}(P_{ba})$.

The final step is to describe a bijection between $\A(w)$ and $\coprod_{\sigma \in \collin(w)}\A(\sigma)$.
First, given a colored $\sigma$-partition $g$ for some $\sigma \in \collin(w)$, we define a colored $w$-partition $f$.
Suppose $w$ has underlying colored permutation $\pi$.
Since $\sigma \in \collin(w)$, we know that $\sigma(k)$ is a letter in $\pi_i$ for some $i\in [0,r-1]$.
Hence there exists an $l\in [n]$ such that $\sigma(k) = |\pi(l)|_{\epsilon(\pi(l))-i}$.
We then define $f(\pi(l)) = |g(\sigma(k))|_{i}$.
Repeating this for every $k \in [n]$ defines the desired colored $w$-partition $f$.
Next, given a colored $w$-partition $f$, we define a colored $\sigma$-partition $g$ for some $\sigma \in \collin(w)$.
Since $\sigma(k) = |\pi(l)|_{\epsilon(\pi(l))-i}$, we set $g(\sigma(k)) = |f(\pi(l))|_0$.
Repeating this for every $k\in [n]$ defines a colored $\sigma$-partition and the composition of these two bijections is the identity. \\
\end{proof}

\begin{ex}
Let $P$ be the colored poset in Figure \ref{fig: 4-colored poset}.
We first see that $0_2 <_P 1_0 <_P 3_1 <_P 0_3$ and so both $f(1_0)\in X_2$ and $f(3_1) \in X_2$.
Next we see that $1_{0-2} = 1_2 < 3_3 = 3_{1-2}$ and so $f(1_0) \leq f(3_1)$.
Then, since $\epsilon(3_1) \neq 2$, we have $f(3_1) < (2,x_{\infty})$.
Taken together, we have $(2,x_0) \leq f(1_0) \leq f(3_1) < (2,x_{\infty})$.
By examining the image of $2_1$, $\A(P)$ can be decomposed into the three sets.
If $f(2_1) \in X_2$ then we have the set
\[
\A(0_1 0_2 2_1 1_0 3_1 0_3) = \{f \, : \, (2,x_{0}) \leq f(2_1) < f(1_0) \leq f(3_1) < (2,x_{\infty}) \}
\]
where $f(2_1) < f(1_0)$ because $2_1 <_P 1_0$ and $2_{1-2} = 2_3 > 1_2 = 1_{0-2}$.
If $f(2_1) \in X_1$ then we have the set
\[
\A(0_1 2_1 0_2 1_0 3_1 0_3) =  \{f \, : \,  (1,x_0) \leq f(2_1) \leq (1,x_{\infty}), (2,x_{0}) \leq f(1_0) \leq f(3_1) < (2,x_{\infty})\}
\]
where $f(2_1) \leq (1,x_{\infty})$ because $\epsilon(2_1) = 1$.
If $f(2_1) \in X_0$ then we have the set
\[
\A(2_1 0_1 0_2 1_0 3_1 0_3) = \{f \, : \,  (0,x_0) \leq f(2_1) < (0,x_{\infty}), (2,x_{0}) \leq f(1_0) \leq f(3_1) < (2,x_{\infty})\}
\]
where $f(2_1) < (0,x_{\infty})$ because $\epsilon(2_1) \neq 0$.
Thus we have shown that
\[
\A(P) = \A(0_1 0_2 2_1 1_0 3_1 0_3) \sqcup \A(0_1 2_1 0_2 1_0 3_1 0_3) \sqcup \A(2_1 0_1 0_2 1_0 3_1 0_3).
\]
Next we examine the sets $\A(w)$ for $w\in \Lin(P)$ and illustrate the bijection between $\A(w)$ and $\coprod_{\sigma \in \collin(w)} \A(\sigma)$.
It is easy to see the bijection between the two sets
\begin{align*}
\A(0_1 0_2 2_1 1_0 3_1 0_3) &=  \{f \, : \, (2,x_{0}) \leq f(2_1) < f(1_0) \leq f(3_1) < (2,x_{\infty}) \} \\
\A(2_3 1_2 3_3) &=  \{f \, : \, (0,x_{0}) \leq f(2_3) < f(1_2) \leq f(3_3) < (0,x_{\infty}) \}.
\end{align*}
Then if we define
\[
\coprod_{\pi \in \shuffle{2_0,1_23_3}} \A(\pi) = \{f \, : \, (0,x_0) \leq f(2_0) \leq (0,x_{\infty}), (0,x_{0}) \leq f(1_2) \leq f(3_3) < (0,x_{\infty}) \}
\]
we also see the bijection
\[
\A(0_1 2_1 0_2 1_0 3_1 0_3) \leftrightarrow \A(1_22_03_3) \sqcup \A(1_23_32_0) \sqcup \A(2_01_23_3).
\]
Note that every $f\in \A(0_1 2_1 0_2 1_0 3_1 0_3)$ with $f(2_1) = (1,x_{\infty})$ is mapped to an element of $\A(1_23_32_0)$.
Finally, the bijection
\[
\A(2_1 0_1 0_2 1_0 3_1 0_3) \leftrightarrow \A(1_22_13_3) \sqcup \A(1_23_32_1) \sqcup \A(2_11_23_3)
\]
is analogous.
\end{ex}

The previous notation can be cumbersome and we will not need such generality to prove the existence of the colored Eulerian descent algebra.
Thus for the rest of the section we set $X = [0,j]$ and define the order polynomial $\Omega_{P} (j)$ to be the number of colored $P$-partitions with parts in $\rversion{[0,j]}$.
Also define $\Omega_{\pi}(j)$ to be the number of colored $\pi$-partitions with parts in $\rversion{[0,j]}$.
We then have the following corollary of the Fundamental Theorem of Colored $P$-partitions.

\begin{cor}
\[
\Omega_{P} (j) = \sum_{\pi \in \collin(P)} \Omega_{\pi} (j)
\]
\end{cor}

We know that $\Omega_\pi (j)$ counts colored $\pi$-partitions $f$ with
\[
0_0 \leq f(\pi(1)) \leq f(\pi(2)) \leq \cdots \leq f(\pi(n)) \leq j_0
\]
with $f(\pi(i)) < f(\pi(i+1))$ if $\pi(i) > \pi(i+1)$ and with $\pi(n+1) = 0_1$.
Thus
\begin{equation}\label{eq: colored order poly}
\Omega_{\pi}(j) = \mchoose{j+1-\des(\pi)}{n} = \binom{j+n-\des(\pi)}{n}
\end{equation}
and so
\[
\sum_{j \geq 0} \Omega_{\pi}(j) t^j = \frac{t^{\des(\pi)}}{(1-t)^{n+1}}.
\]
The previous corollary tells us that
\[
\sum_{j \geq 0} \Omega_{P}(j)  t^j = \frac{\dsum_{\pi \in \collin(P)} t^{\des(\pi)}}{(1-t)^{n+1}}.
\]

Let $P_1 \sqcup P_2$ denote the disjoint union of the two colored posets $P_1$ and $P_2$.
If $\{|a| \, : \, a \in P_1, |a| \neq 0 \}$ and $\{|a| \, : \, a \in P_2, |a| \neq 0 \}$ are disjoint then the map which sends a colored $(P_1 \sqcup P_2)$-partition $f$ to the ordered pair $(g,h)$ where $g = f|_{P_1}$ and $h = f|_{P_2}$ gives us the following theorem.

\begin{thm}
There is a bijection between the set $\mathcal{A}(P_1 \sqcup P_2)$ and the cartesian product $\mathcal{A}(P_1) \times \mathcal{A}(P_2)$.
\[
\mathcal{A}(P_1 \sqcup P_2) \leftrightarrow \mathcal{A}(P_1) \times \mathcal{A}(P_2)
\]
\end{thm}

\begin{cor}
\[
\Omega_{P_1 \sqcup P_2} (j) = \Omega_{P_1}(j) \Omega_{P_2}(j)
\]
\end{cor}

If $P$ is the singleton poset $a_0$ with $a \in \pos$ then $\Omega_P(j) = j+1 + (r-1)j  = rj+1$.
If $P$ is the antichain consisting of $1_0, 2_0, \ldots, n_0$ then $\collin(P) = \colpermrn$ and, by the previous corollary, we know that $\Omega_P(j) = (rj+1)^n$.
This gives an alternate proof of Corollary~\ref{Col Eul Poly}.

Next let $P(\pi)$ denote the union of the chains
\[
0_1 <_P 0_2 <_P \cdots <_P 0_{r-1}
\]
and
\[
\pi(1) <_P \pi(2) <_P \cdots <_P \pi(n).
\]
At first glance it seems difficult to compute $\Omega_{P(\pi)}(j)$.
However, if $\pi$ is monochromatic, i.e., if $\epsilon(\pi(1)) = \cdots = \epsilon(\pi(n))$ then it is easily computable and we have the following theorem.

\begin{thm}\label{thm: mono order poly}
If $\pi \in \colpermrn$ is monochromatic then
\[
\Omega_{P(\pi)}(j) = \binom{rj+n-\intdes(\pi)}{n}.
\]
\end{thm}

\begin{proof}
Suppose $\pi$ is monochromatic. The relationship between $f(\pi(i))$ and $f(\pi(i+1))$ in $X_k$ is the same for all $k=0,\ldots,r-1$ since $\epsilon(\pi(i))-k = \epsilon(\pi(i+1))-k$ for all $k=0,\ldots,r-1$.
For notational ease, let $\epsilon(\pi) = \epsilon(\pi(1))$.
We also have $f(\pi(i)) \neq j_k$ for $k \neq \epsilon(\pi)$.
Thus $\Omega_{P(\pi)}(j)$ counts maps $f$ from $P(\pi)$ to the totally ordered set
\[
0_0 < \cdots < (j-1)_0  < \cdots < 0_{\epsilon(\pi)} < \cdots < j_{\epsilon(\pi)} < \cdots < 0_{r-1} < \cdots < (j-1)_{r-1}
\]
such that $f(\pi(i)) \leq f(\pi(i+1))$ and $f(\pi(i)) < f(\pi(i+1))$ if $\pi(i) > \pi(i+1)$.
This image set has size $j+1+ (r-1)j = rj+1$ and a descent at position $n$ does not affect $\Omega_{P(\pi)} (j)$.
Hence
\[
\Omega_{P(\pi)}(j) = \binom{rj+n-\intdes(\pi)}{n}.
\qedhere
\]
\end{proof}

As it turns out, the previous theorem holds for all $\pi \in \colpermrn$ but to prove this we need an intermediate lemma.
If $\pi$ is not monochromatic then let $i$ be the largest value such that $\epsilon(\pi(1)) = \cdots = \epsilon(\pi(i))$.
Thus $i$ is the length of the monochromatic run beginning with $\pi(1)$.
Let $b=\epsilon(\pi(i+1))$ and define a new permutation $\rho(\pi)$ by setting $$\rho(\pi)(s) = \left\{\begin{array}{cl}|\pi(s)|_b, & \text{if } 1\leq s \leq i \\ \pi(s), & \text{if } i < s \leq n \end{array}\right. .$$
This map simply shifts the color of each of the letters in the initial monochromatic run to $b$ and in doing so increases the length of run.
 
\begin{lem}\label{lem: MonoRunShift}
If $\des(\pi) = \des(\rho(\pi))$ then $\Omega_{P(\pi)}(j) = \Omega_{P(\rho(\pi))}(j)$.
\end{lem}

\begin{proof}
This proof is notationally unpleasant but is really just a careful description of a bijection between the sets $\A(P(\pi))$ and $\A(P(\rho(\pi)))$.
This bijection is obtained by shifting (or renaming) points in $\rversion{[0,j]}$.
Suppose that the length of the initial monochromatic run is $i$ and $a = \epsilon(\pi(i)) < \epsilon(\pi(i+1)) =b$.
Define a new colored permutation $\sigma$ by setting $\sigma(s) = |\pi(s)|_{a+1}$ for $1 \leq s \leq i$ and $\sigma(s) = \pi(s)$ for $i < s \leq n$.
If $\des(\pi) = \des(\sigma)$ then we can find a bijection between colored $P(\pi)$-partitions and colored $P(\sigma)$-partitions and thus show that $\Omega_{P(\pi)}(j) = \Omega_{P(\sigma)}(j)$.

First note that for $s \neq i$, if $\pi(s)$ and $\pi(s+1)$ are mapped to $X_k$ then the relation $f(\pi(s)) \sim f(\pi(s+1))$ is the same as the relation $f(\sigma(s)) \sim f(\sigma(s+1))$ when they are both mapped to $X_k$.
That is, they are either both weak or both strict inequalities.
This is because if $s < i$ then $\epsilon(\pi(s)) = \epsilon(\pi(s+1))$ and $\epsilon(\sigma(s)) = \epsilon(\sigma(s+1))$.
If $s > i$ then $\epsilon(\pi(s)) = \epsilon(\sigma(s))$.
Also, if $s \leq i$ then $f(\pi(s)) \neq j_k$ unless $k=a$ and $f(\sigma(s)) \neq j_k$ unless $k = a+1$.
If $s> i$ then, since $\epsilon(\pi(s)) = \epsilon(\sigma(s))$, both $f(\pi(s)) \neq j_k$ and $f(\sigma(s)) \neq j_k$ unless $\epsilon(\pi(s)) = k$.

Next consider $\pi(i)$ and $\pi(i+1)$.
We claim that if  $\pi(i)$ and $\pi(i+1)$ are mapped to $X_k$ for $k\neq a+1$ then the relation $f(\pi(i)) \sim f(\pi(i+1))$ is the same as the relation $f(\sigma(i)) \sim f(\sigma(i+1))$ when they are both mapped to $X_k$.
However, if they are both mapped to $X_{a+1}$ then $f(\pi(i)) < f(\pi(i+1))$ and $f(\sigma(i)) \leq f(\sigma(i+1))$.
Lastly, note that $f(\pi(i)) \neq j_k$ unless $k= a$ and $f(\sigma(i)) \neq j_k$ unless $k = a+1$.

We define a bijection between colored $P(\pi)$-partitions and colored $P(\sigma)$-partitions as follows.
For every colored $P(\pi)$-partition $f$, define a colored $P(\sigma)$-partition $g$ by setting $g(\sigma(s)) = f(\pi(s))$ for $i < s \leq n$.
If $1 \leq s \leq i$ then define
\[
g(\sigma(s)) = \left\{\begin{array}{rl} f(\pi(s)), & \text{if } f(\pi(s)) < j_a \\ 0_{a+1}, & \text{if } f(\pi(s)) = j_a  \\ (l+1)_{a+1}, &  \text{if } f(\pi(s)) = l_{a+1} \text{ and } 0 \leq l < j \\ f(\pi(s)), & \text{if } f(\pi(s)) > j_{a+1} \end{array}\right. .
\]
This bijection shows that $\Omega_{P(\pi)}(j) = \Omega_{P(\sigma)}(j)$.
This process can be repeated until $\epsilon(\pi(i))+1 = \epsilon(\pi(i+1))$.
Then, as long as doing so does not change the number of descents, it can be repeated once more so that $\sigma = \rho(\pi)$ and thus $\Omega_{P(\pi)}(j) = \Omega_{P(\rho(\pi))}(j)$.

Next suppose that the length of the initial monochromatic run is $i$ and \linebreak $a = \epsilon(\pi(i)) > \epsilon(\pi(i+1)) =b$.
Define a new permutation $\sigma$ by setting $\sigma(s) = |\pi(s)|_{a-1}$ for $1 \leq s \leq i$ and $\sigma(s) = \pi(s)$ for $i < s \leq n$.
If $\des(\pi) = \des(\sigma)$ then we can find a bijection between colored $P(\pi)$-partitions and colored $P(\sigma)$-partitions.

As before, for $s \neq i$, if $\pi(s)$ and $\pi(s+1)$ are mapped to $X_k$ then the relation $f(\pi(s)) \sim f(\pi(s+1))$ is the same as the relation $f(\sigma(s)) \sim f(\sigma(s+1))$ when they are both mapped to $X_k$.
This is because $\epsilon(\pi(s)) = \epsilon(\pi(s+1))$ and $\epsilon(\sigma(s)) = \epsilon(\sigma(s+1))$ for $s < i$ and $\epsilon(\pi(s)) = \epsilon(\sigma(s))$ for $s > i$.
Also, if $s \leq i$ then $f(\pi(s)) \neq j_k$ unless $k=a$ and $f(\sigma(s)) \neq j_k$ unless $k=a-1$.
If $s> i$ then $\epsilon(\pi(s)) = \epsilon(\sigma(s))$ and so both $f(\pi(s)) \neq j_k$ and $f(\sigma(s)) \neq j_k$ unless $\epsilon(\pi(s)) = k$.

Next consider $\pi(i)$ and $\pi(i+1)$.
We claim that if  $\pi(i)$ and $\pi(i+1)$ are both mapped to $X_k$ for $k\neq a$ then the relation $f(\pi(i)) \sim f(\pi(i+1))$ is the same as the relation $f(\sigma(i)) \sim f(\sigma(i+1))$ when they are both mapped to $X_k$.
However, if they are both mapped to $X_a$ then $f(\pi(i)) \leq f(\pi(i+1))$ and $f(\sigma(i)) < f(\sigma(i+1))$.
Lastly, note that $f(\pi(i)) \neq j_k$ unless $k = a$ and $f(\sigma(i)) \neq j_k$ unless $k = a-1$.

We define a bijection between colored $P(\pi)$-partitions and colored $P(\sigma)$-partitions as follows.
For every colored $P(\pi)$-partition $f$, define a colored $P(\sigma)$-partition $g$ by setting $g(\sigma(s)) = f(\pi(s))$ for $i < s \leq n$.
If $1 \leq s \leq i$ then define
\[
g(\sigma(s)) = \left\{\begin{array}{rl} f(\pi(s)), & \text{if } f(\pi(s)) < 0_a \\ j_{a-1}, & \text{if } f(\pi(s)) = 0_a  \\ (l-1)_a, &  \text{if } f(\pi(s)) = l_a \text{ and } 0 < l \leq j \\ f(\pi(s)), & \text{if } f(\pi(s)) > j_a \end{array}\right. .
\]
This bijection shows that $\Omega_{P(\pi)}(j) = \Omega_{P(\sigma)}(j)$.
This process can be repeated until $\epsilon(\pi(i))-1 = \epsilon(\pi(i+1))$.
Then, as long as doing so does not change the number of descents, it can be repeated once more so that $\sigma = \rho(\pi)$ and thus $\Omega_{P(\pi)}(j) = \Omega_{P(\rho(\pi))}(j)$.
\end{proof}

\begin{ex}
Table \ref{tab: gross p-partition shift} shows how the bijection from the previous lemma maps colored $P(\pi)$-partitions with $\pi = 2_0 1_0 3_2$ to colored $P(\rho(\pi))$-partitions with $\rho(\pi) = 2_2 1_2 3_2$.
\end{ex}

\begin{table}[htdp]
\begin{center}
\begin{tabular}{c|c|c} $P(\pi)$-partitions &  $P(\sigma)$-partitions  &  $P(\rho(\pi))$-partitions \\ \hline 
$(0_0,1_0,0_1)$ & $(0_0,0_1,0_1)$  & $(0_0,0_1,0_1)$ \\
$(0_0,1_0,0_2)$ & $(0_0,0_1,0_2)$  & $(0_0,0_1,0_2)$ \\
$(0_0,1_0,1_2)$ & $(0_0,0_1,1_2)$  & $(0_0,0_1,1_2)$ \\
$(0_0,0_1,0_2)$ & $(0_0,1_1,0_2)$  & $(0_0,0_2,0_2)$ \\
$(0_0,0_1,1_2)$ & $(0_0,1_1,1_2)$  & $(0_0,0_2,1_2)$ \\
$(0_0,0_2,1_2)$ & $(0_0,0_2,1_2)$  & $(0_0,1_2,1_2)$ \\ 
$(1_0,0_1,0_2)$ & $(0_1,1_1,0_2)$  & $(0_1,0_2,0_2)$ \\
$(1_0,0_1,1_2)$ & $(0_1,1_1,1_2)$  & $(0_1,0_2,1_2)$ \\
$(1_0,0_2,1_2)$ & $(0_1,0_2,1_2)$  & $(0_1,1_2,1_2)$ \\
$(0_1,0_2,1_2)$ & $(1_1,0_2,1_2)$  & $(0_2,1_2,1_2)$ 
\end{tabular}
\bigskip
\caption{Colored $P(\pi)$-partitions and colored $P(\rho(\pi))$-partitions for $X=\{0,1\}$ and $r=3$. \label{tab: gross p-partition shift}}
\end{center}
\end{table}

\begin{thm}\label{thm: P(pi) order poly}
If $\pi \in \colpermrn$ then
\[
\Omega_{P(\pi)}(j) = \binom{rj+n-\intdes(\pi)}{n}.
\]
\end{thm}

\begin{proof}
First we replace $\pi$ with an alternate permutation $\sigma$ such that $|\sigma|$ has the same descent set as the internal descent set of $\pi$.
We define $\sigma$ as follows.
If $\pi(i)$ is the smallest letter in $\pi$ then define $\sigma(i) = 1_{\epsilon(\pi(i))}$.
Similarly, if $\pi(j)$ is the second smallest letter in $\pi$ then set $\sigma(j) = 2_{\epsilon(\pi(j))}$.
Continue this process until all $n$ letters of $\sigma$ have been defined.
This produces a new permutation $\sigma$ with the same descent set as $\pi$.
To see why, note that $\epsilon(\pi(i)) = \epsilon(\sigma(i))$ for $i = 1,\ldots,n$.
Thus all descents between colors are preserved.
Also, if $\pi$ has a descent at position $i$ with $\epsilon(\pi(i)) = \epsilon(\pi(i+1))$ then $\pi(i) > \pi(i+1)$ and so $|\sigma|(i) > |\sigma|(i+1)$ and thus this descent is also preserved.
Next note that $\Omega_{P(\pi)}(j) = \Omega_{P(\sigma)}(j)$ because both descent set and the colors of the letters are preserved under the standardization.
The purpose of the standardization is to find a permutation $\sigma$ such that in applying Lemma \ref{lem: MonoRunShift} we can always shift the color of the initial monochromatic run to match the color of the next term without creating a new descent.
Hence repeatedly using Lemma \ref{lem: MonoRunShift} allows us to shift from $\sigma $ to $\overline{\sigma}$ defined by $\overline{\sigma}(i) = |\sigma(i)|_{\epsilon(\pi(n))}$ for $i = 1,\ldots, n$.
Since $\overline{\sigma}$ is monochromatic, $\Omega_{P(\overline{\sigma})}(j)$ is given by Theorem \ref{thm: mono order poly}.
Lastly, we know that $\intdes(\overline{\sigma}) = \intdes(\pi)$ and that $\Omega_{P(\pi)}(j) = \Omega_{P(\sigma)}(j) = \Omega_{P(\overline{\sigma})}(j)$.
Hence
\[
\Omega_{P(\pi)}(j) = \binom{rj+n-\intdes(\pi)}{n}.
\qedhere
\]
\end{proof}

In \cite{Steingrimsson1994}, Steingr{\'{\i}}msson described a shelling of the unit hypercube that corresponds to colored permutations and descents.
This suggests that a connection between colored $P$-partitions and the Ehrhart polynomials of regions bounded by certain hyperplanes in $\R^n$ should be possible but it is not immediately clear how this would work.
This will be the subject of future research.


\section{The colored Eulerian descent algebra}\label{sec: colored descent algebra}

In this section, we will prove that the Mantaci-Reutenauer algebra contains a subalgebra induced by descent number with basis elements $C_i$ defined by
\[
C_i = \sum_{\des(\pi) = i} \pi.
\]
To prove that such an algebra exists, we compute $C_\pi(s,t) := \sum_{\sigma \tau = \pi} s^{\des(\sigma)}t^{\des(\tau)}$ and show that $C_\pi(s,t)$ is determined by $\des(\pi)$.

For every $I \subseteq [n]$ and $\pi \in \colpermrn$, define the \emph{colored zig-zag poset} $\zigipi$ by setting $\pi(i) <_Z \pi(i+1)$ if $i \notin I$ and $\pi(i) >_Z \pi(i+1)$ if $i \in I$ with $\pi(n+1) = 0_1$.
We still require that $0_1 < 0_2 < \cdots < 0_{r-1}$ but will not mention this explicitly except where necessary.

\begin{lem}\label{Col Zig-Zag Extensions}
Let $\sigma \in \colpermrn$ and consider $\zigipi$ with $\pi \in \colpermrn$ and $I \subseteq [n]$.
Then $\sigma \in \collin(\zigipi)$ if and only if $\Des(\sigma^{-1}\pi) = I$.
\end{lem}

\begin{proof}
Suppose $\pi(i) = A_a$ and $\pi(i+1) = B_b$.
If $i \notin I$ then we have $A_a <_Z B_b$ and thus there exist $c,d,C,D$ such that $\sigma^{-1}(A_{a-c}) = C_0$ and $ \sigma^{-1}(B_{b-d}) = D_0$ with $c\leq d$ and if $c < d$ then $C < D$.
Hence  $\sigma^{-1}\pi(i) = C_c < D_d = \sigma^{-1}\pi(i+1)$.
Similarly, if $i \in I$ then $\sigma^{-1}\pi(i)  > \sigma^{-1}\pi(i+1)$.
This also shows that $\sigma^{-1}\pi(n) = C_c$ where $c=0$ if and only if $n \notin I$.
Hence $n \in \Des(\sigma^{-1}\pi)$ if and only if $n \in I$.
\end{proof}

\begin{ex}
If $r=3$, $n=3$, $\pi = 2_1 1_2 3_2$, and $I = \{1\}$ then $\zigipi$ is the poset $2_1 >_Z 1_2 <_Z 3_2 <_Z 0_1$.
We have
\begin{align*}\collin(\zigipi) &= \{1_2 2_1 3_2\} \cup \{1_2 3_2 2_1\} \cup \shuffle{1_23_2,2_0} \cup \shuffle{1_23_2,2_2} \\
&= \{1_2 2_0 3_2, 1_2 2_1 3_2, 1_2 2_2 3_2, 1_2 3_2 2_0, 1_2 3_2 2_1, 1_2 3_2 2_2, 2_0 1_2 3_2, 2_2 1_2 3_2\}
\end{align*}
 and
 \[
 \coprod_{\sigma \in \collin(\zigipi)} \{\sigma^{-1}\pi  \} = \{2_0 1_0 3_0, 3_0 1_0 2_0, 1_1 2_0 3_0, 2_1 1_0 3_0, 3_1 1_0 2_0, 1_2 2_0 3_0, 2_2 1_0 3_0, 3_2 1_0 2_0\}.
 \]
Here $\Des(\sigma^{-1}\pi) = \{1\}$ for every $\sigma \in \collin(\zigipi)$.
\end{ex}

Next we consider a second collection of posets.
For every $I \subseteq [n]$ and $\pi \in \colpermrn$, define the \emph{colored chain poset} $\chainipi$ by setting $\pi(i) <_C \pi(i+1)$ if $i \notin I$ with $\pi(n+1) = 0_1$.
The following lemma about linear extensions of colored chain posets is similar to that for colored zig-zag posets and follows from the same proof.

\begin{lem}\label{Col Chain Extensions}
Let $\sigma \in \colpermrn$ and consider $\chainipi$ with $\pi \in \colpermrn$ and $I \subseteq [n]$.
Then $\sigma \in \collin(\chainipi)$ if and only if $\Des(\sigma^{-1}\pi) \subseteq I$.
\end{lem}

\begin{ex}
If $r=3$, $n=3$, $\pi = 2_1 1_2 3_2$, and $I = \{1\}$ then $\chainipi$ is the poset $ 1_2 <_C 3_2 <_C 0_1$ with no relation between $2_1$ and any other elements.
Then $\collin(\chainipi) = \collin(\zigipi) \cup \{2_1 1_2 3_2\}$ and so
\[
\{\sigma^{-1}\pi \, | \, \sigma \in \collin(\chainipi)\} = \{\sigma^{-1}\pi \, | \, \sigma \in \collin(\zigipi)\} \cup \{1_0 2_0 3_0\}
\]
and $\Des(1_0 2_0 3_0) = \varnothing$.
\end{ex}

Next we define barred versions of both colored zig-zag posets and colored chain posets.
First, a \emph{barred colored zig-zag poset} is defined to be a colored zig-zag poset $\zigipi$ with an arbitrary number of bars placed in each of the $n+1$ spaces to the left of $0_1$ such that between any two bars the elements of the poset (not necessarily their labels) are increasing.
This can be thought of as requiring at least one bar in each ``descent'' of the colored zig-zag poset.
Note that $0_1$ will always be to the right of all the bars.
Figure \ref{fig: barred col zig poset} gives an example of a barred $\zigipi$ poset with $I = \{3\}$ and $\pi = 2_0 1_2 3_1 4_1$.

\begin{figure}[htbp]
\[\xymatrix @!R @!C @C=5pt{ \ar@{-}[dddd] &  & \ar@{-}[dddd] & & &  &  \ar@{-}[dddd] &  \ar@{-}[dddd] & &  & \\   &  &  & & &  3_1 \ar@{-}[ddrrr] &  &   & &  & \\  &  &  & 1_2 \ar@{-}[urr] &  & & &  & & & 0_1 \\  & 2_0 \ar@{-}[urr]  &  &  & & &  &  & 4_1 \ar@{-}[urr]& &  \\  &&&&&&&&&& }\] 
\caption{\; A barred $\zigipi$ poset with $I = \{3\}$ and $\pi = 2_0 1_2 3_1 4_1$.}
\label{fig: barred col zig poset}
\end{figure}

If we begin with a colored zig-zag poset $\zigipi$ then we must first place a bar in space $i$ for each $i \in I$.
From there we are free to place any number of bars in any of the $n+1$ spaces.
Define $\Omega_{\zigipi}(j,k)$ to be the number of ordered pairs $(f,P)$ where $P$ is a barred $\zigipi$ poset with $k$ bars and $f$ is a colored $\zigipi$-partition with parts in $\rversion{[0,j]}$.
Recall that $\sigma \in \collin(\zigipi)$ if and only if $\Des(\sigma^{-1}\pi) = I$.
If we begin with the colored zig-zag poset $\zigipi$ then we first place one bar in space $i$ for each $i\in I$.
Next there are
\[
\mchoose{n+1}{k-|I|} = \binom{k+n - |I|}{n}
\]
ways to place the remaining $k-|I|$ bars in the $n+1$ spaces and hence there are $\binom{k+n-|I|}{n}$ barred $\zigipi$ posets with $k$ bars.
Thus
\[
\Omega_{\zigipi}(j,k) = \sum_{\sigma \in \collin(\zigipi)} \Omega_{\sigma}(j) \binom{k+n-\des(\sigma^{-1}\pi)}{n} = \sum_{\sigma \in \collin(\zigipi)} \Omega_{\sigma}(j) \Omega_{\sigma^{-1}\pi}(k).
\]
If we set $\tau = \sigma^{-1}\pi$ and sum over all $I \subseteq [n]$ then we have
\begin{equation}\label{eq: col barred equiv}
\sum_{I \subseteq [n]} \Omega_{\zigipi}(j,k) = \sum_{\sigma \tau = \pi} \Omega_{\sigma}(j) \Omega_{\tau}(k).
\end{equation}

Next we define a \emph{barred colored chain poset} to be a colored chain poset $\chainipi$ with at least one bar in space $i$ for each $i \in I$ and with an arbitrary number of bars placed on the left end.
No bars are allowed in space $i$ for $i\in [n] \setminus I$.
Thus we place at least one bar between each chain of the colored chain poset and allow for bars on the left end.
Figure \ref{fig: barred col chain poset} gives an example of a barred $\chainipi$ poset with $I = \{1,3\}$ and $\pi = 2_0 1_2 3_1 4_1$.

\begin{figure}[htbp]
\[\xymatrix @!R @!C @R=23pt @C=10pt{  \ar@{-}[ddd] &  & \ar@{-}[ddd]  & &  &  \ar@{-}[ddd] &  \ar@{-}[ddd] & &  \\   &  & & & 3_1   &  & &  &  0_1 \\  & 2_0  &  & 1_2\ar@{-}[ur] & & & & 4_1 \ar@{-}[ur] & \\  &   &  &  & & & & &  }\] 
\caption{\; A barred $\chainipi$ poset with $I = \{1, 3\}$ and $\pi = 2_0 1_2 3_1 4_1$.}
\label{fig: barred col chain poset}
\end{figure}

If we begin with a colored chain poset $\chainipi$ then we must first place a bar in space $i$ for each $i \in I$.
From there we are free to place any number of bars in space $i$ for any $i \in I$ and any number of bars on the left end.
This creates a collection of bars with each compartment between adjacent bars containing at most one nonempty chain.
Define $\Omega_{\chainbipi} (j,k)$ to be the number of ordered pairs $(f,P)$ where $P$ is a barred $\chainipi$ poset with $k$ bars and $f$ is a colored $\chainipi$-partition with parts in $\rversion{[0,j]}$.
Recall that $\sigma \in \collin(\chainipi)$ if and only if $\Des(\sigma^{-1}\pi) \subseteq I$.
If we begin with the colored chain poset $\chainipi$ then we first place one bar in space $i$ for each $i\in I$.
Next there are
\[
\mchoose{|I|+1}{k-|I|} = \binom{k}{k-|I|}
\]
ways to place the remaining $k-|I|$ bars in the $|I|+1$ allowable spaces.
Hence there are $\binom{k}{k-|I|}$ barred $\chainipi$ posets with $k$ bars and we see that
\[
\Omega_{\chainipi}(j,k) = \sum_{\sigma \in \collin(\chainipi)} \Omega_{\sigma}(j) \binom{k}{k-|I|}.
\]
The following lemma allows us to compare colored $P$-partitions for barred colored zig-zag posets and barred colored chain posets.

\begin{lem}\label{Barred Col Poset Extensions}
For every $\pi \in \colpermrn$,
\[
\sum_{I \subseteq [n]}  \Omega_{\zigipi}(j,k) = \sum_{I \subseteq [n]} \Omega_{\chainipi}(j,k).
\]
\end{lem}

\begin{proof}
We show that for every $\sigma \in \colpermrn$ there is a bijection between barred colored zig-zag posets with $k$ bars such that $\sigma$ is a linear extension of the underlying colored zig-zag poset and barred colored chain posets with $k$ bars such that $\sigma$ is a linear extension of the underlying colored chain poset.
This bijection is simply the map that sends each barred colored zig-zag poset to the barred colored chain poset obtained by removing the relation between $\pi(i)$ and $\pi(i+1)$ for every space $i$ containing at least one bar.
\end{proof}

The bijection in the previous lemma maps Figure \ref{fig: barred col zig poset} to Figure \ref{fig: barred col chain poset}.
Now that all the pieces are in place, we are ready to prove the existence of the colored Eulerian descent algebra.

\begin{thm}\label{Colored Descent Alg}
For every $\pi \in \colpermrn$,
\begin{equation}\label{eq: col des algebra}
\sum_{j,k \geq 0} \binom{ rjk+j+k+n-\des(\pi)}{n} s^j t^k = \frac{\dsum_{\sigma\tau = \pi} s^{\des(\sigma)} t^{\des(\tau)}}{(1-s)^{n+1}(1-t)^{n+1}}.
\end{equation}
\end{thm}

\begin{proof}
First we see that
\begin{align*}
\frac{\dsum_{\sigma\tau = \pi} s^{\des(\sigma)} t^{\des(\tau)}}{(1-s)^{n+1}(1-t)^{n+1}} &=  \sum_{j,k \geq 0} \sum_{\sigma \tau = \pi} \binom{j+n}{n}\binom{k+n}{n}s^{j+\des(\sigma)} t^{k+\des(\tau)}\\
&= \sum_{j,k \geq 0} \sum_{\sigma \tau = \pi} \binom{j+n-\des(\sigma)}{n}\binom{k+n-\des(\tau)}{n}s^j t^k \\ 
&= \sum_{j,k \geq 0} \sum_{I \subseteq [n]} \Omega_{\zigipi}(j,k) s^j t^k
\end{align*}
where the last equality follows from \eqref{eq: colored order poly} and \eqref{eq: col barred equiv}.
By Lemma \ref{Barred Col Poset Extensions}, we can switch from colored zig-zag posets to colored chain posets and we have
\[
\frac{\dsum_{\sigma\tau = \pi} s^{\des(\sigma)} t^{\des(\tau)}}{(1-s)^{n+1}(1-t)^{n+1}} = \sum_{j,k \geq 0} \sum_{I \subseteq [n]} \Omega_{\chainipi}(j,k) s^j t^k.
\]
The only remaining step is to prove that
\[
\sum_{I \subseteq [n]} \Omega_{\chainipi}(j,k) =  \binom{rjk+j+k+n-\des(\pi)}{n}.
\]

First we note that $\sum_{I \subseteq [n]} \Omega_{\chainipi}(j,k)$ counts ordered pairs $(f,P)$ where $P$ is a barred $\chainipi$ poset with $k$ bars for some $I \subseteq [n]$ and $f$ is a colored $\chainipi$-partition with parts in $\rversion{[0,j]}$.
Fix a barred $\chainipi$ poset with $k$ bars and use the bars to define compartments labeled $0,\ldots,k$ from left to right.
Then define $\pi_i$ to be the (possibly empty) subword of $\pi$ in compartment $i$ and denote the length of $\pi_i$ by $L_i$.
Then
\[
\Omega_{\chainipi}(j) = \Omega_{\pi_k}(j)\prod_{i=0}^{k-1} \Omega_{P(\pi_i)}(j) 
\]
By Theorem \ref{thm: P(pi) order poly}, we can assume that $\pi$ is monochromatic.
For $i=0,\ldots,k-1$, we let $\Omega_{P(\pi_i)}(j)$ count solutions to the inequalities
\[
i(rj+1) \leq s_{i_1} \leq  \cdots \leq s_{i_{L_i}} \leq i(rj+1) + rj
\]
with $s_{i_l} < s_{i_{l+1}}$ if $l \in \intDes(\pi_i)$.
Next we let $\Omega_{\pi_k}(j)$ count solutions to the inequalities
\[
k(rj+1) \leq s_{k_1} \leq \cdots \leq s_{k_{L_k}} \leq rjk+j+k
\]
with $s_{k_l} < s_{k_{l+1}}$ if $l \in \intDes(\pi_k)$ and with $s_{k_{L_k}} < rjk+j+k$ if $n \in \Des(\pi)$.
By concatenating these inequalities, we see that if we sum over all $I \subseteq [n]$ and all barred $\chainipi$ posets with $k$ bars then $\sum_{I \subseteq [n]} \Omega_{\chainipi}(j,k)$ is equal to the number of solutions to the inequalities
\[
0 \leq s_1 \leq \cdots \leq s_n \leq rjk+j+k
\]
with $s_i < s_{i+1}$ if $i \in \intDes(\pi)$ and with $s_n < rjk+j+k$ if $n \in \Des(\pi)$.
Hence we conclude that
\[
\sum_{I \subseteq [n]} \Omega_{\chainipi}(j,k)  = \binom{rjk+j+k +n-\des(\pi)}{n} . \qedhere
\]
\end{proof}

Figure \ref{fig: every barred colored chain poset} depicts every barred $C(I,2_11_3)$ poset with $2$ bars and $I \subseteq \{1,2\}$.
Thus we can identify each barred $\chainipi$ poset with a barred colored permutation with underlying colored permutation $\pi$ such that $0_1$ is in the rightmost compartment.
This provides a visual representation of the final step in the proof of the previous theorem.

\begin{figure}[htbp]
\[\xymatrix  @R=8pt @C=8pt{ && \ar@{-}[ddd] & \ar@{-}[ddd] &&&&&&&&&& \ar@{-}[ddd] & \ar@{-}[ddd] && \\   & 1_3 &&&&&& \ar@{-}[dd]  && \ar@{-}[dd]  &&&&&&& 0_1 \\ 2_1\ar@{-}[ur] &&&& 0_1 & \qquad & 2_1 && 1_3 && 0_1 & \qquad & 2_1 &&& 1_3\ar@{-}[ur] &   \\  &&&&&&&&&&&&&&&& \\  &&&&&&&&&&&&&&&& \\ &&&&&&&&&&&& \ar@{-}[dddd]  & \ar@{-}[dddd] &&& \\  \ar@{-}[ddd] &&& \ar@{-}[ddd] &&& \ar@{-}[ddd] && \ar@{-}[ddd] &&&&&&&& 0_1 \\ && 1_3 &&&&&&&& 0_1 &&&&& 1_3\ar@{-}[ur] & \\ & 2_1\ar@{-}[ur] &&& 0_1 &&& 2_1 && 1_3\ar@{-}[ur] &&&&& 2_1\ar@{-}[ur] && \\ &&&&&&&&&&&&&&&&      }\] 
\caption{\; Barred $C(I,2_11_3)$ posets with $2$ bars and $I \subseteq \{1,2\}$.}
\label{fig: every barred colored chain poset}
\end{figure}

Define the colored structure polynomial $\phi(x)$ by
\[
\phi(x) = \sum_{\pi \in \colpermrn} \binom{x+n-\des(\pi)}{n} \pi.
\]
If we expand both sides of equation \eqref{eq: col des algebra} and compare coefficients of $s^j t^k$ we see that
\[
\binom{rjk+j+k+n-\des(\pi)}{n} = \sum_{\sigma\tau = \pi}  \binom{j+n -\des(\sigma)}{n} \binom{k+n - \des(\tau)}{n}.
\]
This implies that $\phi(j)\phi(k) = \phi(rjk + j +k)$ for all $j,k \geq 0$ and so we have the following theorem.

\begin{thm}
As polynomials in $x$ and $y$ with coefficients in the group algebra of $\colpermrn$,
\[
\phi(x)\phi(y) = \phi(rxy+x+y).
\]
\end{thm}

If we substitute $x \leftarrow (x-1)/r$ and $y \leftarrow (y-1)/r$ into the previous theorem we see that $\phi((x-1)/r)\phi((y-1)/r) = \phi((xy-1)/r)$.
Thus if we expand $\phi((x-1)/r)$ we have
\[
\phi((x-1)/r) = \sum_{\pi \in \colpermrn} \binom{\frac{x-1}{r}+n-\des(\pi)}{n} \pi = \sum_{i=0}^n c_i x^i
\]
and the $c_i$ are orthogonal idempotents which span the colored Eulerian descent algebra.
The $c_i$ reduce to the Eulerian idempotents $e_i$ when $r=1$.

\begin{ex}
Let $r=5$ and $n=3$.
If $C_i$ is the sum of all permutations in $G_{5,3}$ with $i$ descents then we end up with the following orthogonal idempotents:
\begin{align*}
c_0 &= \frac{1}{750} (504 C_0 - 36 C_1 + 24 C_2 - 66 C_3) \\
c_1 &= \frac{1}{750} (218 C_0 + 23 C_1 - 22 C_2 + 83 C_3) \\
c_2 &= \frac{1}{750} (27 C_0 + 12 C_1 - 3 C_2 - 18 C_3) \\
c_3 &= \frac{1}{750} (C_0 + C_1 + C_2 + C_3) .
\end{align*}
\end{ex}

We mentioned earlier that we need not have set $\pi(n+1) = 0_1$ but did so because this induces an algebra.
If fact, if we consider all possible descent numbers defined by setting $\pi(0) = 0_a$ and $\pi(n+1) = 0_b$ then the definition we use induces the only set partition that gives rise to an algebra.

\begin{ex}
Consider $G_{2,2}$.
We represent the four distinct set partitions by setting $\pi(0) = 0_0$ or $0_1$ and $\pi(n+1) = 0_1$ or $0_2$.
We have already shown that setting $\pi(0) = 0_0$ and $\pi(n+1) = 0_1$ induces an algebra and so now examine the other possibilities.

Case 1: If we begin each colored permutation with $0_0$ and end with $0_2$ then we have two basis elements:
\begin{align*}
C_0 &= 1_0 2_0 + 1_0 2_1 + 1_1 2_1 + 2_0 1_1  \\
C_1 &= 1_1 2_0 + 2_0 1_0 + 2_1 1_0 + 2_1 1_1.
\end{align*}
However,
\[
C_0^2 = 3\cdot 1_0 2_0 + 2\cdot 1_0 2_1 + 2\cdot 1_1 2_0 + 3\cdot 1_1 2_1 + 2_0 1_0 + 2\cdot 2_0 1_1 + 2\cdot 2_1 1_0 + 2_1 1_1
\]
and the difference in coefficients of $1_0 2_0$ and $1_0 2_1$ shows that this product cannot be written as a linear combination of the basis elements.

Case 2: If we begin each colored permutation with $0_1$ and end with $0_2$ then we have three basis elements:
\begin{align*}
C_0 &= 1_1 2_1  \\
C_1 &= 1_0 2_0 + 1_0 2_1 + 1_1 2_0 + 2_0 1_1 + 2_1 1_0 + 2_1 1_1 \\
C_2 &= 2_0 1_0.
\end{align*}
However, $C_0^2 = 1_0 2_0$ which cannot be written as a linear combination of the basis elements because $1_0 2_0$ is not a singleton element.

Case 3: If we begin each colored permutation with $0_1$ and end with $0_1$ then we have two basis elements:
\begin{align*}
C_1 &= 1_0 2_0 + 1_1 2_0 + 1_1 2_1 + 2_1 1_0 \\
C_2 &= 1_0 2_1 + 2_0 1_0 + 2_0 1_1 + 2_1 1_1.
\end{align*}
However,
\[
C_1^2  = 3\cdot 1_0 2_0 + 2\cdot 1_0 2_1 + 2\cdot 1_1 2_0 + 3\cdot 1_1 2_1 + 2_0 1_0 + 2\cdot 2_0 1_1 + 2\cdot 2_1 1_0 + 2_1 1_1
\]
and the difference in coefficients of $1_0 2_0$ and $1_1 2_0$ shows that this product cannot be written as a linear combination of the basis elements.
\end{ex}

Finally, it is natural to ask whether there is a colored descent algebra induced by descent set.
The descent set statistic induces a set partition of $\colpermrn$ that would fall between the set partition induced by descent number and the set partition corresponding to the Mantaci-Reutenauer algebra and thus the descent set statistic would induce an intermediate algebra.
Unfortunately, we can see that the descent set statistic already fails to induce an algebra in the case of $G_{2,2}$.

\begin{ex}
The three permutations in $G_{2,2}$ with descent set $\{2\}$ are $1_0 2_1$, $1_1 2_1$, and $2_0 1_1$.
The only permutation in $G_{2,2}$ with descent set $\{1,2\}$ is $2_1 1_1$.
However,
\[
(1_0 2_1 + 1_1 2_1 + 2_0 1_1)(2_1 1_1) = 1_0 2_1 + 2_0 1_0 + 2_0 1_1
\]
and both $1_0 2_1$ and $2_0 1_1$ appear in this expansion but $1_1 2_1$ does not.
Hence the product cannot be written as a linear combination of basis elements.
\end{ex}

\backmatter
\singlespacing
\bibliographystyle{amsplain}
\bibliography{/Users/mmoynihan/Desktop/Research/MasterBib}

\end{document}